\documentclass[pdflatex,sn-mathphys-num]{sn-jnl}

\usepackage{graphicx}%
\usepackage{multirow}%
\usepackage{amsmath,amssymb,amsfonts}%
\usepackage{amsthm}%
\usepackage{mathrsfs}%
\usepackage[title]{appendix}%
\usepackage{xcolor}%
\usepackage{textcomp}%
\usepackage{manyfoot}%
\usepackage{multicol}
\usepackage{booktabs}%
\usepackage{stmaryrd} 
\usepackage{algorithm}%
\usepackage{algorithmicx}%
\usepackage{algpseudocode}%
\usepackage{listings}%

\usepackage{cases}
\usepackage{subfigure}
\usepackage{tikz}
\usetikzlibrary{arrows.meta, positioning, shapes.geometric}

\theoremstyle{thmstyleone}%
\newtheorem{theorem}{Theorem}
\newtheorem{proposition}[theorem]{Proposition}%

\theoremstyle{thmstyletwo}%
\newtheorem{remark}{Remark}%

\theoremstyle{thmstylethree}%
\newtheorem{definition}{Definition}%

\newtheorem{assumption}{Assumption}%
\newtheorem{corollary}{Corollary}
\newtheorem{lemma}{Lemma}
\newtheorem{alemma}{Lemma}[section]
\begin{document}

\title[Article Title]{Convergence Analysis of an Inexact MBA Method for Constrained DC Problems}







\author[1]{\fnm{Ruyu} \sur{Liu}}\email{maruyuliu@mail.scut.edu.cn}

\author[1]{\fnm{Shaohua} \sur{Pan}}\email{shhpan@scut.edu.cn}

\author*[1]{\fnm{Shujun} \sur{Bi}}\email{bishj@scut.edu.cn}

\affil*[1]{\orgdiv{School of Mathematics}, \orgname{South China University of Technology}, \orgaddress{\city{Guangzhou},  \country{China}}}


\abstract{This paper concerns a class of constrained difference-of-convex (DC) optimization problems in which, the constraint functions are continuously differentiable and their gradients are strictly continuous. For such nonconvex and nonsmooth optimization problems, we develop an inexact moving balls approximation (MBA) method by a workable inexactness criterion for the solution of subproblems. This criterion is proposed by leveraging a global error bound for the strongly convex program associated with parametric optimization problems. We establish the full convergence of the iterate sequence under the Kurdyka-{\L}ojasiewicz (KL) property of the constructed potential function, achieve the local convergence rate of the iterate and objective value sequences under the KL property of the potential function with exponent $q\in[1/2,1)$, and provide the iteration complexity of $O(1/\epsilon^2)$ to seek an $\epsilon$-KKT point. A verifiable condition is also presented to check whether the potential function has the KL property of exponent $q\in[1/2,1)$. To our knowledge, this is the first implementable inexact MBA method with a complete convergence certificate. Numerical comparison with DCA-MOSEK, a DC algorithm with subproblems solved by MOSEK, is conducted on $\ell_1\!-\!\ell_2$ regularized quadratically constrained optimization problems, which demonstrates the advantage of the inexact MBA in the quality of solutions and running time.}




\maketitle
\section{Introduction}\label{sec1}

Constrained optimization problems are a class of tough composite optimization problems because the composite involves the indicator function of a complicated closed set. They naturally arise when one attempts to seek a solution that minimizes an objective under some restrictions. On the other hand, nonconvex optimization problems abound with difference-of-convexity (DC) functions \cite{Tao1997dc,Le2005dc,Le2018dc}; for example, a $\mathcal{C}^{1,1}$ function (a differentiable function with strictly continuous\footnote{This paper follows the definitions of \cite{RW98}, and strict continuity means local Lipschitz continuity.} gradient) is locally DC. We are interested in developing an efficient algorithm for the constrained DC problem
\begin{equation}\label{prob}
\min_{x\in\mathbb{R}^n}F(x):=f(x)+\phi(x)-\psi(x)+\delta_{\mathbb{R}_{-}^m}(g(x)),
\end{equation}
where the functions $f,\phi,\psi$ and $g$ satisfy the basic conditions stated in Assumption \ref{ass1}, and $\delta_{\mathbb{R}_{-}^m}(\cdot)$ is the indicator function of the non-positive orthant $\mathbb{R}_{-}^m$. We always assume that problem \eqref{prob} has a nonempty set of strong stationary points by Definition \ref{sspoint-def}.
\begin{assumption}\label{ass1}
 {\bf(i)} $g\!:\mathbb{R}^n\to\mathbb{R}^m$ is a $\mathcal{C}^{1,1}$ mapping such that $\Gamma\!:=g^{-1}(\mathbb{R}_{-}^m)\ne\emptyset$; 
	
 \noindent
 {\bf(ii)} $f\!:\mathbb{R}^n\to\overline{\mathbb{R}}:=(-\infty,\infty]$ is a proper lsc function that is $\mathcal{C}^{1,1}$ on an open convex set $\mathcal{O}\supset\Gamma$;
	
 \noindent
 {\bf(iii)} $\phi:\mathbb{R}^n\to\mathbb{R}$ and $\psi\!:\mathbb{R}^n\to\mathbb{R}$ are convex functions;
 
 \noindent
 {\bf(iv)} $F$ is bounded from below on the set $\Gamma$, i.e., ${\rm val}^*:=\inf_{x\in\Gamma}F(x)>-\infty$. 
\end{assumption}

Model \eqref{prob} includes a variety of convex and nonconvex constrained optimization problems depending on the assumptions on $f$ and $g_1,\ldots,g_m$. It allows $\psi$ to be a finite weakly convex function, and also covers the case that $\phi$ is an extended real-valued convex function with the domain of the form $H^{-1}(\mathbb{R}_{-}^p)$ for a $\mathcal{C}^{1,1}$ mapping $H\!:\mathbb{R}^n\to\mathbb{R}^p$, since such $\phi$ can be decomposed to the sum of a finite convex function and the composite function $\delta_{\mathbb{R}_{-}^p}(H(\cdot))$. In addition, when replacing the nonsmooth convex functions in the constraints of \cite[problem (1.1)]{Boob24} with their Moreau envelopes, the problem considered there can be addressed as a special case of \eqref{prob}. Problem \eqref{prob} arises in many science and engineering fields. In typical applications, the function $f$ appears as a measure for data-fidelity, the mapping $g$ is used to model restrictions on the decision variable $x$, and the DC function $\phi-\psi$ is a regularizer to induce desirable structures; see \cite[Table 1]{Gong13} for examples of such regularizers. 

Many methods have been developed for constrained optimization problems; for example, the classical penalty method \cite{Fiacco90,NesterovNemirovskii94}, the popular augmented Lagrangian method \cite{Powell69,Roc76,Bertsekas82}, and the sequential quadratic programming (SQP) method \cite{Solodov09,Solodov10}, but they cannot be directly applied to solve \eqref{prob} due to the nonsmooth DC term in the objective function. One common approach to tackle nonconvex and nonsmooth optimization problems is to resort to the majorization-minimization (MM) procedure, which iteratively constructs and minimizes a surrogate function that majorizes $F$ locally; see \cite{Sun16,Lewis16,Bolte16,Drusvyatskiy21} for related models and discussions. Since problem \eqref{prob} is a special case of the DC composite program considered in \cite[Eq.\,(2)]{Le24}, it can be solved with the DC composite algorithm, a class of MM method developed in \cite{Le24} by linearizing $g$ and $\psi$. However, the generated iterate sequence was only proved to have a subsequential convergence under the extended Robinson's constraint qualification (RCQ). In what follows, we mainly review another class of MM procedure for constrained optimization problems, i.e., the MBA method, closely related to this work, and make comparisons among some significant works on this method in Table \ref{table_compare}.
\begin{table}[h]
\caption{Comparison of algorithms with full convergence and complexities}\label{table_compare}%
\setlength{\tabcolsep}{2pt}
\begin{tabular}{@{}llllllll@{}}
\midrule
 Algorithms &\quad\ \ $f$ &\ $(\phi,\psi)$ & constraints &\ \ CQ & inexact\ \ &  full convergence & complexity\footnotemark[1]\\
 \midrule
 MBA\cite{Auslender10}  & $L$-smooth &(No,\,No) & $L$-smooth & MFCQ & \ \ No & convexity & $O(\log(1/\epsilon))$\\
 MBA\cite{Bolte16}    & $L$-smooth &(No,\,No)   & $L$-smooth & MFCQ  & \ \ No & semialgebraicity & \quad\ --\\
 Multiprox\cite{Bolte19}  & $L$-smooth &(Yes,\,No)  & $L$-smooth & Slater & \ \ No & convexity & $O(1/\epsilon)$ \\
 ${\rm SCP_{ls}}$\cite{YuPongLv21} & $L$-smooth&(Yes,\,Yes\footnotemark[2]) & $L$-smooth & MFCQ   &\ \ No &\quad\ KL &\quad\ --\\
 ${\rm MBA}$\cite{Nabou24} & $L$-smooth&(Yes,\,No) & $L$-smooth & MFCQ   &\ \ No &\quad\ KL &$O(1/\epsilon^2)$\\
 LCPG\cite{Boob24} & $L$-smooth &(Yes,\,No) & composite\footnotemark[3] & MFCQ &\ \ Yes &\quad\ --  & $O(1/\epsilon^2)$ \\
 iMBAdc & $\quad\mathcal{C}^{1,1}$ &(Yes,\,Yes) & $\quad \ \mathcal{C}^{1,1}$ & BMP  &\ \  Yes & \quad\ KL & $O(1/\epsilon^2)$\\
 \botrule
 \end{tabular}
 \footnotetext[1]{The complexity in \cite{Auslender10} was achieved for the iterate sequence under the strong convexity of the objective function, the complexity in \cite{Bolte19}  was obtained for the objective values, and the others are all established for reaching an $\epsilon$-KKT point (see Definition \ref{aKKT}).}
\footnotetext[2]{$\psi$ is required to be $\mathcal{C}^{1,1}$ on the set of stationary points.} 
\footnotetext[3]{Every $g_i$ is the sum of an L-smooth function and a continuous convex function.}
\end{table}
\subsection{Related work}\label{sec1.1}

The MBA method was first proposed by Auslender et al. \cite{Auslender10} for solving problem \eqref{prob} without the DC term $\phi-\psi$ but with L-smooth $f$ and $g$ (i.e., differentiable functions on $\mathbb{R}^n$ with Lipschitz continuous gradients). They constructed a surrogate of $F$ with the second-order Taylor's expansions of $f$ and $g_i$ involving only the spectral norms of their Hessians. Since the feasible set is approximated by an intersection of $m$ closed balls, the resulting iterative scheme is called the moving balls approximation method. Different from the standard SQP method, it generates a sequence of feasible iterates from a feasible starting point. The authors of \cite{Auslender10} proved that, under the MFCQ, any cluster point of the iterate sequence is a stationary point, the whole sequence converges to a minimizer if $f,g_1,\ldots,g_m$ are also convex, and the convergence has a linear rate if in addition $f$ is strongly convex. Later, for semialgebraic $f$ and $g$, Bolte and Pauwels \cite{Bolte16} established the full convergence and convergence rate of the iterate sequence under the MFCQ, thereby enhancing the convergence results of \cite{Auslender10}.

For problem \eqref{prob} with $f\equiv 0\equiv\psi$ and only one convex constraint $g_1(x)\le 0$, Shefi and Teboulle \cite{Shefi16} proposed the method DUMBA by adopting the ball approximation of the smooth constraint, keeping the objective function intact and solving the resulting model by duality. Under the Slater's CQ, they achieved the full convergence of the iterate sequence and provided the global convergence rate of the objective values with $O(1/\varepsilon)$ evaluation for finding an $\epsilon$-optimal solution. For the convex composite problem 
\begin{equation}\label{composite}
\min_{x\in\mathbb{R}^n}\vartheta(H(x)),
\end{equation}
where $\vartheta\!:\mathbb{R}^m\to\overline{\mathbb{R}}$ is a proper convex lsc function and $H\!:\mathbb{R}^n\to\mathbb{R}^m$ is a L-smooth mapping with all components being convex, Bolte et al. \cite{Bolte19} proposed a multiproximal linearization method, a variant of the MBA method, and achieved the same convergence result and global rate of convergence as in \cite{Shefi16} under the same CQ. 

Recently, Yu et al. \cite{YuPongLv21} studied a line-search variant of the MBA method for solving problem \eqref{prob} with $L$-smooth $f$ and $g$. Under the MFCQ and the assumption that $\psi$ is $\mathcal{C}^{1,1}$ on an open set containing the stationary point set of \eqref{prob}, they established the full convergence of the iterate sequence when the constructed potential function has the KL property, and derived the local convergence rate when the potential function has the KL property of exponent $q\in[1/2,1)$. Nabou and Necoara \cite{Nabou24} developed a moving Taylor approximation (MTA) method for problem \eqref{prob} with $\psi\equiv 0$, which reduces to the MBA method when the involved $p$ and $q$ both take $1$. Under the MFCQ, they established the full convergence of the iterate sequence if the constructed proximal Lagrange potential function has the KL property, and the local convergence rate when the potential function has the KL property of exponent $q\in[1/2,1)$. In addition, under the MFCQ, they also derived the complexity of $O(1/\epsilon^2)$ for seeking an $\epsilon$-KKT point. 

It is worth pointing out that the above MBA methods and their variants all require the exact solutions of  subproblems, which is unrealistic especially in large-scale applications due to computation error and computation cost. To overcome the issue, for problem \eqref{prob} with $\psi\equiv 0$ but every $g_i$ being a sum of an L-smooth function and a continuous convex function, Boob et al. \cite{Boob24} developed a level constrained proximal gradient (LCPG) method to allow the inexact solutions of subproblems. Their method will become an inexact MBA method if the convex function in every $g_i$ is removed. They put more emphasis on the complexity analysis of LCPG under the MFCQ for stochastic and deterministic setting, though the subsequential convergence of the iterate sequence was also proved. In the deterministic case, their inexact LCPG method was proved to require $O(1/\epsilon)$ evaluations for finding an $(\epsilon,\delta)$ type-II KKT point, which is allowed to have the $\delta$-distance from a $\sqrt{\epsilon}$-KKT point. We notice that their inexactness criterion (see \cite[Definition 7]{Boob24}) involves the exact primal-dual optimal solutions of subproblems and the optimal values of subproblems. Now it is unclear whether such an inexactness criterion is workable or not in practical computation. Thus, there remains a need to develop an MBA method with genuinely practical inexactness criteria.  

On the other hand, the existing asymptotic convergence results of the MBA method and its variants are all obtained under the MFCQ for problem \eqref{prob} with L-smooth $f,g$ and $\psi\equiv 0$ or $\psi$ being $\mathcal{C}^{1,1}$ on the stationary point set. Notice that $\psi\equiv 0$ or $\psi$ being $\mathcal{C}^{1,1}$ on the stationary point set greatly restricts the applications of \eqref{prob} in nonconvex and nonsmooth optimization. In addition, for classical nonlinear programming problems, the MFCQ is relatively strong. Recently, it has witnessed that the weaker metric subregularity constraint qualification (MSCQ) was used to derive optimality conditions and study stability of constraint system \cite{Gfrerer17}. Then, it is natural to ask whether the convergence of MBA-type methods can be established under MSCQ.  
\subsection{Main contribution}\label{sec1.2}

In this work, we attempt to address the above-mentioned issues. Our main contribution is the development of a practical inexact MBA method with a complete theoretical certificate for solving problem \eqref{prob}. It includes the following several aspects.
\begin{itemize}
\item First, we formulate a local MBA model by adopting the line search strategy to seek a tight upper estimation for the Lipschitz modulus of $\nabla\!f$ and $\nabla g$ at the current iterate, and  propose an implementable inexactness criterion for the solving of subproblems by leveraging a global error bound, which is derived for the strongly convex programs associated with a parametric optimization problem arising from the approximation to \eqref{prob}. Our inexactness criterion, unlike the one proposed in \cite{Boob24}, does not involve the unknown dual solution of subproblems. Though the optimal values of subproblems likewise appear in the criterion, they can be replaced by their lower bounds that are products at hand if a dual method is employed to solve subproblems. 

\item Second, the convergence analysis of the proposed inexact MBA method is conducted under the MSCQ for the constraint systems of subproblems (see Assumption \ref{ass2}), the partial bounded multiplier property (BMP) on the set of cluster points (see Assumption \ref{ass4}), and the KL property of a potential function $\Phi_{\widetilde{c}}$. Under Assumptions \ref{ass2} and \ref{ass4} and the boundedness of the iterate sequence, we prove that the whole iterate sequence converges to a strong stationary point as long as $\Phi_{\widetilde{c}}$ has the KL property on the set of cluster points, and if $\Phi_{\widetilde{c}}$ has the KL property of exponent $q\in[1/2,1)$, the convergence has the linear rate for $q=1/2$ and sublinear rate for $q\in(1/2,1)$. As shown in \cite{Gfrerer17}, the partial BMP is implied by the MFCQ and the constant rank constraint qualification (CRCQ), so for the first time, the full convergence of the iterate sequence of the MBA method is achieved without MFCQ. We also derive a verifiable criterion to check whether $\Phi_{\widetilde{c}}$ has the KL property of exponent $q\in[1/2,1)$, which is shown to be weaker than the one obtained in \cite[Theorem 3.2]{LiPong18}.  

\item Third, the complexity of $O(1/\epsilon^2)$ is established for seeking an $\epsilon$-KKT point under the MFCQ and L-smoothness of $f$ and $g$. The complexity bound has the same order as the exact MBA method for problem \eqref{prob} with $\psi\equiv 0$ (see \cite[Theorem 4.3]{Nabou24}), and as the inexact LCPG for seeking an $(\epsilon^2,\delta)$ type-II KKT point (see \cite[Theorem 8]{Boob24}). Taking into account the distinction between the $(\epsilon^2,\delta)$ type-II KKT point and the $\epsilon$-KKT point, our inexact MBA method exhibits an improved complexity guarantee. 
\end{itemize}

We also implement the proposed inexact MBA method by solving the dual of its subproblems with the proximal gradient method with line-search (PGls), and compare the performance of iMBA-PGls with that of the commercial software MOSEK on $\ell_1$-norm regularized convex quadratically constrained quadratic programs (QCQPs), and with that of DCA-MOSEK, a DC algorithm with subproblems solved by MOSEK, on $\ell_1\!-\!\ell_2$ regularized QCQPs. The comparison results indicate that our iMBA-PGls can produce solutions of better quality within less time for large-scale problems.

This paper is organized as follows. Section \ref{sec2} includes some background knowledge and preliminary results used for the subsequent analysis. Section \ref{sec3} describes the iteration steps of our inexact MBA method and proves its well-definedness. Section \ref{sec4} establishes the full convergence and convergence rate results. Section \ref{sec5} provides the complexity analysis of the inexact MBA method for  finding an $\epsilon$-KKT point. Section \ref{sec6} presents the implementation of our inexact MBA and tests its performance on the convex $\ell_1$-norm regularized QCQPs, the $\ell_1\!-\!\ell_2$ regularized QCQPs, and $\ell_1\!-\!\ell_2$ regularized Student's $t$-regression with constraints. Finally, we conclude this paper. 
\subsection{Notation}\label{sec1.3}

Throughout this paper, $\mathcal{I}$ denotes the identity mapping from $\mathbb{R}^n$ to $\mathbb{R}^n$, $\|\cdot\|$ (resp. $\|\cdot\|_1$ and $\|\cdot\|_{\infty}$) represents the $\ell_2$-norm (resp. $\ell_1$-norm and $\ell_{\infty}$-norm), $\mathbb{B}$ stands for the unit ball centered at the origin, and $\mathbb{B}(x,\delta)$ denotes the closed ball centered at $x$ with radius $\delta>0$. For a linear mapping $\mathcal{Q}\!:\mathbb{R}^n\to\mathbb{R}^n$, $\mathcal{Q}^*$ represents its adjoint, $\mathcal{Q}\succeq 0$ means that $\mathcal{Q}$ is positive semidefinite, i.e., $\mathcal{Q}=\mathcal{Q}^*$ and $\langle x,\mathcal{Q}x\rangle\ge 0$ for all $x\in\mathbb{R}^n$, and  $\|\mathcal{Q}\|\!:=\max_{\|x\|=1}\|\mathcal{Q}x\|$ denotes its spectral norm. For an integer $k\ge 1$, write $[k]:=\{1,\ldots,k\}$. For a vector $x\in\mathbb{R}^n$, $x\in[a,b]$ for $a,b\in\mathbb{R}$ with $a\le b$ means $a\le x_i\le b$ for all $i\in[n]$, and $x_{+}$ denotes the vector whose $i$th component is $\max\{0,x_i\}$. For a closed set $C\subset\mathbb{R}^n$, $\delta_{C}(\cdot)$ stands for the indicator function of $C$, i.e., $\delta_{C}(x)=0$ if $x\in C$; otherwise $\delta_{C}(x)=\infty$, and ${\rm dist}(x,C)$ denotes the distance on the $\ell_2$-norm from $x$ to $C$. For a proper lsc function $h\!:\mathbb{R}^n\to\overline{\mathbb{R}}$, the notation $[a<h<b]$ with real numbers $a<b$ denotes the set $\{x\in\mathbb{R}^n\,|\,a<h(x)<b\}$, $h^*$ represents the conjugate function of $h$, and if $h$ is $\mathcal{C}^{1,1}$ at $\overline{x}$, ${\rm lip}\,\nabla h(\overline{x})$  denotes the Lipschitz modulus of its gradient $\nabla h$ at $\overline{x}$. For a mapping $H\!=(H_1,\ldots,H_m)\!:\mathbb{R}^n\to\mathbb{R}^m$, if $H$ is differentiable at $\overline{x}$, write $\nabla\! H(\overline{x}):=[H'(\overline{x})]^{\top}$ where $H'(\overline{x})$ is the Jacobian of $H$ at $\overline{x}$; and if $H$ is $\mathcal{C}^{1,1}$ at $\overline{x}\in\mathbb{R}^n$, write $\ell_{\nabla\!H}(\overline{x}):=({\rm lip}\,\nabla H_1(\overline{x}),\ldots,{\rm lip}\,\nabla H_m(\overline{x}))^{\top}$. 

\section{Preliminaries}\label{sec2} 
 
We first introduce some background knowledge on multifunctions related to this work, and more details are found in \cite{RW98,Mordu23}. Consider a multifunction  $\mathcal{F}\!:\mathbb{R}^n\rightrightarrows\mathbb{R}^m$. It is said to be outer semicontinuous (osc) at $\overline{x}\in\mathbb{R}^n$ if $\limsup_{x\to\overline{x}}\mathcal{F}(x)\subset\mathcal{F}(\overline{x})$ or equivalently $\limsup_{x\to\overline{x}}\mathcal{F}(x)=\mathcal{F}(\overline{x})$, and it is said to be locally bounded at $\overline{x}\in\mathbb{R}^n$ if for some neighborhood $V$ of $\overline{x}$, the set $\mathcal{F}(V):=\bigcup_{x\in V}\mathcal{F}(x)$ is bounded. Denote by ${\rm gph}\,\mathcal{F}$ the graph of the mapping $\mathcal{F}$, and by $\mathcal{F}^{-1}$ the inverse mapping of $\mathcal{F}$. 

Now we recall the basic subdifferential of an extended real-valued function.
\begin{definition}\label{def-subdiff} 
 (see \cite[Definition 8.3]{RW98}) Consider a function $h\!:\mathbb{R}^n\to\overline{\mathbb{R}}$ and a point $x\in{\rm dom}\,h$. The regular (or Fr\'echet) subdifferential of $h$ at $x$ is defined as
 \[
  \widehat{\partial}h(x):=\bigg\{v\in\mathbb{R}^n\ |\ \liminf_{x'\to x\atop x'\neq x}\frac{h(x')-h(x)-\langle v,x'-x\rangle}{\|x'-x\|}\ge 0\bigg\}, 
 \]
 and its basic (also known as limiting or Morduhovich) subdifferential at $x$ is defined as 
 \[
  \!\partial h(x)\!:=\!\Big\{v\in \mathbb{R}^n\ |\ \exists\, x^k\xrightarrow[h]{} x, v^k\!\in \widehat{\partial}h(x^k)\ {\rm with}\ v^k \to v \Big\}, 
 \]
 where $x^k\xrightarrow[h]{} x$ signifies $x^k\to x$ with $h(x^k)\to h(x)$.
\end{definition}
\begin{remark}\label{remark-subdiff}
 In view of Assumption \ref{ass1} (iii), the mappings $\partial(-\psi)\!:\mathbb{R}^n\rightrightarrows\mathbb{R}^n$ and $\partial\phi\!:\mathbb{R}^n\rightrightarrows\mathbb{R}^n$ are osc. Furthermore, according to \cite[Theorem 9.13 (d)]{RW98}, they are locally bounded.
\end{remark}

When $h$ is an indicator function of a closed set $C\subset\mathbb{R}^n$, its regular subdifferential at $x\in {\rm dom}\,h$ becomes the regular normal cone to $C$ at $x$, denoted by $\widehat{\mathcal{N}}_{C}(x)$, while its subdifferential at $x\in {\rm dom}\,h$ reduces to the (limiting) normal cone to $C$ at $x$, denoted by $\mathcal{N}_{C}(x)$. When $C$ is convex, $\widehat{\mathcal{N}}_{C}(x)=\mathcal{N}_{C}(x)=\big\{v\in\mathbb{R}^n\ |\ \langle v,z-x\rangle\le0\ \ \forall z\in C\big\}$. 
\subsection{Metric regularity and subregularity}\label{sec2.1}

A multifunction $\mathcal{F}\!:\mathbb{R}^n\rightrightarrows\mathbb{R}^m$ is said to be metrically regular at a point $(\overline{x},\overline{y})\in{\rm gph}\,\mathcal{F}$ if there exist $\delta>0$ and $\kappa>0$ such that for all $x,y\in\mathbb{B}((\overline{x},\overline{y}),\delta)$, 
\[
{\rm dist}(x,\mathcal{F}^{-1}(y))\le\kappa\,{\rm dist}(y,\mathcal{F}(x)),
\]
and that $\mathcal{F}$ is (metrically) subregular at $(\overline{x},\overline{y})$ if there exist $\delta>0$ and $\kappa>0$ such that 
\[
{\rm dist}(x,\mathcal{F}^{-1}(\overline{y}))\le\kappa\,{\rm dist}(\overline{y},\mathcal{F}(x))\quad\forall x\in\mathbb{B}(\overline{x},\delta).
\]
Clearly, the metric regularity of $\mathcal{F}$ at $(\overline{x},\overline{y})$ implies its subregularity at this point. 

Let $H\!:\mathbb{R}^n\to\mathbb{R}^m$ be a continuously differentiable mapping.
From \cite[Sections 2.3.2-2.3.4]{BS00}, the system $H(z)\in\mathbb{R}_{-}^m$ satisfies the RCQ (also known as the MFCQ for the constraint $H(z)\le 0$) at $x\in C:=H^{-1}(\mathbb{R}_{-}^m)$ if and only if the multifunction
\begin{equation}\label{mapFg}
\mathcal{F}_{\!H}(z):=H(z)-\mathbb{R}_{-}^m\quad{\rm for}\ z\in\mathbb{R}^n
\end{equation}
is metrically regular at $(x,0)\in{\rm gph}\mathcal{F}_{\!H}$. Due to the smoothness of $H$, according to \cite[Proposition 2.97]{BS00}, the RCQ at a point $x\in C$ can equivalently be characterized as 
\[
  {\rm Ker}(\nabla\!H(x))\cap\mathcal{N}_{\mathbb{R}_{-}^m}(H(x))=\{0_{\mathbb{R}^m}\}.
\]
If the mapping $\mathcal{F}_{\!H}$ in \eqref{mapFg} is subregular at $(x,0)\in{\rm gph}\mathcal{F}_{\!H}$, from \cite[Page 211]{Ioffe09} it follows $\mathcal{N}_{C}(x)\subset\nabla H(x)\mathcal{N}_{\mathbb{R}_{-}^m}(H(x))$, which along with $\mathcal{N}_{C}(x)\supset\widehat{\mathcal{N}}_{C}(x)\supset\nabla H(x)\mathcal{N}_{\mathbb{R}_{-}^m}(H(x))$ implies $\mathcal{N}_{C}(x)=\nabla H(x)\mathcal{N}_{\mathbb{R}_{-}^m}(H(x))=\widehat{\mathcal{N}}_{C}(x)$, i.e., the Clarke regularity of $C$. 
\begin{remark}\label{remark-Abadie}
 When all components of $H$ are also convex, from the proof for the sufficiency of \cite[Theorem 3.5]{Liwu97}, the subregularity of $\mathcal{F}_{\!H}$ at $(x,0)$ with $x\in C$ is implied by the Abadie's CQ for the system $H(z)\in\mathbb{R}_{-}^m$ at $x$. As well known, for differentiable convex optimization problems, the Abadie's CQ is the weakest condition to guarantee the characterization of an optimal solution by Karush-Kuhn-Tucker (KKT) conditions.
 
\end{remark}

\subsection{Stationary points of problem \eqref{prob}}\label{sec2.2}

Let $x^*$ be a local optimal solution of \eqref{prob}. From  \cite[Theorem 10.1 \& Corollary 10.9]{RW98}, $0\in\partial F(x^*)\subset\nabla\!f(x^*)+\partial\phi(x^*)+\partial(-\psi)(x^*)+\mathcal{N}_{\Gamma}(x^*)$, where $\mathcal{N}_{\Gamma}(x^*)=\nabla g(x^*)\mathcal{N}_{\mathbb{R}_{-}^m}(g(x^*))$ if the mapping $\mathcal{F}_{\!g}(\cdot):=g(\cdot)-\mathbb{R}_{-}^m$ in \eqref{mapFg} is subregular at $(x^*,0)$. This motivates us to introduce the following two classes of stationary points for problem \eqref{prob}. 
\begin{definition}\label{sspoint-def}
 A vector $x\in\Gamma$ is said to be a strong stationary point of problem \eqref{prob} if
 \[
0\in\nabla\!f(x)+\partial\phi(x)+\partial(-\psi)(x)+\nabla g(x)\mathcal{N}_{\mathbb{R}_{-}^m}(g(x)),
 \]
 or equivalently, there exists a vector $\lambda\in\mathbb{R}^m$ such that $(x,\lambda)$ satisfies
 \begin{subnumcases}{}
 0\in\nabla\!f(x)+\partial\phi(x)+ \partial (-\psi)(x)+\nabla g(x)\lambda,\nonumber\\  g(x)\in\mathbb{R}_{-}^m,\,\lambda\in\mathbb{R}_{+}^m,\,\langle g(x),\lambda\rangle=0.\nonumber
 \end{subnumcases}
 In the sequel, we denote by $X_{s}^*$ the set of strong stationary points of problem \eqref{prob}.
\end{definition}
\begin{definition}\label{spoint-def}
 A vector $x\in\Gamma$ is said to be a stationary point of problem \eqref{prob} if
 \[
0\in\nabla\!f(x)+\partial\phi(x)-\partial\psi(x)+\nabla g(x)\mathcal{N}_{\mathbb{R}_{-}^m}(g(x)), 
\]
while it is called a singular stationary point of \eqref{prob} if $0\in\nabla\!f(x)+\partial\phi(x)-\partial\psi(x)$. We denote by $X^*$ (resp. $\widehat{X}^*$) the set of stationary points (resp. singular stationary points) of \eqref{prob}.
\end{definition}

The stationary point in terms of Definition \ref{spoint-def} is common in the literature on DC problems (see \cite[Definition 2.2]{YuPongLv21} or \cite[Theorem 4]{Le24} with $F_1(x)\equiv x$). Since the inclusion $\partial(-\psi)(x)\subset-\partial\psi(x)$ is generally strict, the set $X_{s}^*$ is smaller than $X^*$. This interprets why we use ``strong'' to decorate the stationary point in Definition \ref{sspoint-def}. Obviously, the set of singular stationary points is also strictly contained in the set of stationary points, but has no inclusion relation with the set of strong stationary points.
\subsection{Partial BMP of parametric constraint system}\label{sec2.3}

In this section, we concentrate on the following parametric optimization problem
\begin{equation}\label{para-prob}
\min_{x\in\mathbb{R}^n}\big\{\theta(u,x)+\phi(x)\ \ {\rm s.t.}\ \ H(u,x)\in\mathbb{R}_{-}^m\big\},
\end{equation} 
where $u$ is the perturbation parameter from a real vector space $\mathbb{U}$, and $\theta\!:\mathbb{U}\times\mathbb{R}^n\to\mathbb{R}$ and $H\!:\mathbb{U}\times\mathbb{R}^n\to\mathbb{R}^m$ are continuously differentiable functions. The partial BMP was first used in \cite{Gfrerer17} to study the stability of the parametric constraint system. To present its formal definition, we define the solution mapping to the system $H(u,x)\in\mathbb{R}_{-}^m$ by
\begin{equation}\label{MapS}
\mathcal{S}(u):=\big\{x\in\mathbb{R}^n\ |\  H(u,x)\in\mathbb{R}_{-}^m\big\}\quad{\rm for\ all}\  u\in\mathbb{U}.
\end{equation}
\begin{definition}\label{def-BMP}
(see \cite[Definition 3.1]{Gfrerer17}) The constraint system of \eqref{para-prob} is said to satisfy the partial BMP with respect to $x$ at $(u^*,x^*)\in{\rm gph}\,\mathcal{S}$ if there exist $\kappa>0$ and the neighborhoods $\mathcal{U}$ of $u^*$ and $\mathcal{V}$ of $x^*$ such that for all $u\in\mathcal{U},x\in\mathcal{V}\cap \mathcal{S}(u)$ and $y\in\mathcal{N}_{\mathcal{S}(u)}(x)$,
\[
 \Lambda(u,x,y)\cap\kappa \|y\|\mathbb{B}\ne\emptyset\ \ {\rm with}\ \Lambda(u,x,y)\!:=\big\{\lambda\in\mathcal{N}_{\mathbb{R}_{-}^m}(H(u,x))\ |\ \nabla_{\!x}H(u,x)\lambda=y\big\}.
\]
\end{definition}  

Recall that $x^*\in\mathbb{R}^n$ is a stationary point of \eqref{para-prob} associated with $u=u^*\in\mathbb{U}$ if there exists  $\lambda^*\!\in\mathcal{N}_{\mathbb{R}_{-}^m}(H(u^*,x^*))$ such that
$0\in\nabla_{\!x} \theta(u^*,x^*)+\partial\phi(x^*)+\nabla_{\!x}H(u^*,x^*)\lambda^*$. In view of this, we define the Lagrange multiplier set of \eqref{para-prob} associated with $(u,x)$ by
\begin{equation}\label{para-multiplier}
 \mathcal{M}(u,x)\!:=\!\Big\{\lambda\in\mathcal{N}_{\mathbb{R}_{-}^m}(H(u,x))\ |\ 0\in\nabla\!_x\theta(u,x)\!+\!\partial\phi(x)+\!\nabla\!_xH(u,x)\lambda\Big\}.
\end{equation}
Clearly, $\mathcal{M}(u^*,x^*)\ne\emptyset$ if $x^*$ is a stationary point of \eqref{para-prob} associated with $u=u^*$. By comparing with the expression of the mapping $\Lambda$ in Definition \ref{def-BMP}, if $x^*$ is a stationary point of \eqref{para-prob} with $u=u^*$, then
$\mathcal{M}(u^*,x^*)\supset\bigcup_{v^*\in\partial\phi(x^*)}\Lambda(u^*,x^*,-\nabla_{\!x}\theta(u^*,x^*)-v^*)$.

Next, for a strongly convex problem \eqref{para-prob} associated with $\overline{u}\in\mathbb{U}$, we establish a global error bound for the distance of any $x\in\mathbb{R}^n$ from its unique solution in terms of the violation of KKT conditions, which will be used in Section \ref{sec4} to establish Proposition \ref{prop-eboundk}. Its proof is similar to that of \cite[Theorem 2.2]{Mangasarian88} but needs a weaker CQ. 
\begin{proposition}\label{prop-ebound} 
 Fix any $\overline{u}\in\mathbb{U}$. Suppose that $\theta(\overline{u},\cdot)$ is strongly convex with modulus $\rho(\overline{u})$ and that all components of $H(\overline{u},\cdot)$ are convex with $\mathcal{S}(\overline{u})\ne\emptyset$. Let $\overline{x}$ be the unique optimal solution of \eqref{para-prob} associated with $\overline{u}\in\mathbb{U}$. If the mapping $\mathcal{H}(\cdot):=H(\overline{u},\cdot)-\mathbb{R}_{-}^m$ is subregular at $(\overline{x},0)$, then for any $\overline{\lambda}\in\mathcal{N}_{\mathbb{R}_{-}^m}(H(\overline{u},\overline{x}))$, $(x,\lambda)\in\mathbb{R}^n\times\mathbb{R}_+^m$ and $v\in\partial\phi(x)$, it holds that
 \begin{align*}
 \rho(\overline{u})\|x-\overline{x}\|^2&\!\le\langle x,\nabla_{\!x}\theta(\overline{u},x)+v+\nabla_{\!x} H(\overline{u},x)\lambda\rangle-\langle\lambda,H(\overline{u},x)\rangle\\
 &\quad +\interleave\overline{x}\interleave\interleave\nabla_{\!x}\theta(\overline{u},x)\!+\!v\!+\!\nabla_{\!x}H(\overline{u},x)\lambda\interleave_{*}+\interleave\overline{\lambda}\interleave\interleave[H(\overline{u},x)]_{+}\interleave_{*},
\end{align*} 
where $\interleave\cdot\interleave$ represents an arbitrary norm in $\mathbb{R}^n$ and $\interleave\cdot\interleave_{*}$ is the dual norm of $\interleave\cdot\interleave$.
\end{proposition}
\begin{proof}
 Since $\mathcal{S}(\overline{u})$ is a nonempty closed convex set and $\theta(\overline{u},\cdot)$ is strongly convex, the problem \eqref{para-prob} associated with $\overline{u}$ has a unique optimal solution $\overline{x}$. By the first-order optimality condition,
 \[   
 0\in\nabla_{\!x}\theta(\overline{u},\overline{x})+\partial\phi(\overline{x})+\mathcal{N}_{\mathcal{S}(\overline{u})}(\overline{x}).
 \] 
 Furthermore, in view of Section \ref{sec2.1}, the subregularity of $\mathcal{H}$ at $(\overline{x},0)$ implies that
 \[
 \mathcal{N}_{\mathcal{S}(\overline{u})}(\overline{x})=\nabla_{\!x} H(\overline{u},\overline{x})\mathcal{N}_{\mathbb{R}_{-}^m}(H(\overline{u},\overline{x})).
 \]
 Pick any $\overline{\lambda}\in\mathcal{N}_{\mathbb{R}_{-}^m}H(\overline{u},\overline{x}))$. From the above two equations, there exists $\overline{v}\in\partial\phi(\overline{x})$ such that 
 \begin{equation}\label{ineq-KKT}
\nabla_{\!x}\theta(\overline{u},\overline{x})+\overline{v}+\nabla_{x} H(\overline{u},\overline{x})\overline{\lambda}=0.
\end{equation}
 Since every $H_{i}(\overline{u},\cdot)$ for $i\in[m]$ is assumed to be convex, it holds that
 \begin{equation*}
 H(\overline{u},y)-H(\overline{u},z)\ge H_{x}'(\overline{u},z)(y-z)\quad\forall y,z\in\mathbb{R}^n.
\end{equation*}  
Now fix any $(x,\lambda)\in\mathbb{R}^n\times\mathbb{R}^m_+$ and $v\in\partial\phi(x)$. Using the nonnegativity of $\lambda$ and $\overline{\lambda}$ and invoking the above inequality with $(y,z)=(x,\overline{x})$ and $(y,z)=(\overline{x},x)$, respectively, yields 
\begin{equation}\label{Hconvex-ineq}
\left\{\begin{array}{cl}
 \langle\overline{\lambda},H(\overline{u},x)-H(\overline{u},\overline{x})\rangle\ge \langle\nabla_{x} H(\overline{u},\overline{x})\overline{\lambda}, x-\overline{x}\rangle,\\
\langle\lambda,H(\overline{u},\overline{x})-H(\overline{u},x)\rangle\ge \langle\nabla_{x} H(\overline{u},x)\lambda, \overline{x}-x\rangle.
\end{array}\right.
\end{equation}
From the strong convexity of $\theta(\overline{u},\cdot)$ with modulus $\rho(\overline{u})$ and the convexity of $\phi$, we have
\begin{align*}
\rho(\overline{u})\|x-\overline{x}\|^2&\le\langle x-\overline{x},\nabla_{\!x}\theta(\overline{u},x)-\nabla_{\!x}\theta(\overline{u},\overline{x})\rangle+\langle x-\overline{x},v-\overline{v}\rangle\nonumber\\
 &=\langle x-\overline{x},\nabla_{\!x}\theta(\overline{u},x)+\nabla_{\!x} H(\overline{u},x)\lambda-\nabla_{\!x}\theta(\overline{u},\overline{x})-\nabla_{\!x} H(\overline{u},\overline{x})\overline{\lambda}\rangle\\
 &\quad\ +\langle x-\overline{x},\nabla_{\!x} H(\overline{u},\overline{x})\overline{\lambda}\rangle-\langle x-\overline{x},\nabla_{\!x} H(\overline{u},x)\lambda\rangle+\langle x-\overline{x},v-\overline{v}\rangle\nonumber\\
 &\stackrel{\eqref{Hconvex-ineq}}{\le}\langle x-\overline{x},\nabla_{\!x}\theta(\overline{u},x)+\nabla_{\!x} H(\overline{u},x)\lambda-\nabla_{\!x}\theta(\overline{u},\overline{x})-\nabla_{\!x} H(\overline{u},\overline{x})\overline{\lambda}\rangle\\
 &\quad +\langle\overline{\lambda},H(\overline{u},x)-H(\overline{u},\overline{x})\rangle+\langle\lambda,H(\overline{u},\overline{x})-H(\overline{u},x)\rangle+\langle x-\overline{x},v-\overline{v}\rangle\nonumber\\
 &\stackrel{\eqref{ineq-KKT}}{\le}\langle x-\overline{x},\nabla_{\!x}\theta(\overline{u},x)+v+\nabla_{\!x} H(\overline{u},x)\lambda\rangle+\langle\overline{\lambda},H(\overline{u},x)\rangle-\langle\lambda,H(\overline{u},x)\rangle\nonumber\\
 &\le\langle x,\nabla_{\!x}\theta(\overline{u},x)+v+\nabla_{\!x} H(\overline{u},x)\lambda\rangle+\interleave\overline{\lambda}\interleave\interleave[H(\overline{u},x)]_{+}\interleave_{*}\\
&\quad+\interleave\overline{x}\interleave\interleave\nabla_{\!x}\theta(\overline{u},x)\!+\!v\!+\!\nabla_{\!x}H(\overline{u},x)\lambda\interleave_{*}-\langle\lambda,H(\overline{u},x)\rangle,
\end{align*}
where the third inequality is also due to  $\overline{\lambda}\in\mathcal{N}_{\mathbb{R}_{-}^m}(H(\overline{u},\overline{x}))$ and $\lambda\in\mathbb{R}_+^m$, and the fourth is obtained by using Cauchy-Schwarz inequality. The proof is finished. 
\end{proof}
\subsection{Kurdyka-{\L}ojasiewicz property}\label{sec2.4}

It is well recognized that the KL property plays a crucial role in the study of global convergence and convergence rates of descent methods for nonconvex and nonsmooth optimization after the seminal works \cite{Attouch09,Attouch13,Bolte14}. Here, we present its formal definition.  
\begin{definition}\label{KL-def}
 For every $\eta>0$, denote $\Upsilon_{\!\eta}$ by the set of continuous concave functions $\varphi\!:[0,\eta)\to\mathbb{R}_{+}$ that are continuously differentiable in $(0,\eta)$ with $\varphi(0)=0$ and $\varphi'(s)>0$ for all $s\in(0,\eta)$. A proper function $h\!:\mathbb{R}^n\to\overline{\mathbb{R}}$ is said to have the KL property at a point $\overline{x}\in{\rm dom}\,\partial h$ if there exist $\delta>0,\eta\in(0,\infty]$ and $\varphi\in\Upsilon_{\!\eta}$ such that for all $x\in\mathbb{B}(\overline{x},\delta)\cap[h(\overline{x})<h<h(\overline{x})+\eta]$,
 \[
  \varphi'(h(x)\!-\!h(\overline{x})){\rm dist}(0,\partial h(x))\ge 1;
 \]
 and it is said to have the KL property of exponent $q\in[0,1)$ at $\overline{x}$ if the above $\varphi$ can be chosen to be the function $\mathbb{R}_{+}\ni t\mapsto ct^{1-q}$ for some $c>0$. If $h$ has the KL property (of exponent $q$) at every point of ${\rm dom}\,\partial h$, it is called a KL function (of exponent $q$).
\end{definition}

By \cite[Lemma 2.1]{Attouch10}, to prove that a proper lsc function has the KL property (of exponent $q\in[0,1)$), it suffices to check if the property holds at its critical points. As discussed in \cite[Section 4]{Attouch10}, the KL functions are ubiquitous, and the functions definable in an o-minimal structure over the real field admit this property. Moreover, definable functions or mappings conform to the chain rules friendly to optimization.
\section{An inexact MBA method}\label{sec3}

The basic idea of our inexact MBA method is to produce in each iteration a feasible solution of \eqref{prob} with an improved objective value by solving inexactly a strongly convex subproblem, constructed with a local majorization of constraint and objective functions at the current iterate. Before describing the iteration steps of the forthcoming inexact MBA method, we first take a closer look at the construction of subproblems. 

Let $x^k\in\Gamma$ be the current iterate. Fix any $\mathbb{R}^m\!\ni L^k>\ell_{\nabla\!g}(x^k)$. According to Assumption \ref{ass1} (i), applying the descent lemma (see \cite[Lemma 5.7]{Beck17}) to all components of $g$, there exists $\delta_{k}>0$ such that $g(x)\le g(x^k)+g'(x^k)(x-x^k)+\frac{1}{2}\|x-x^k\|^2L^k$ for all $x\in\mathbb{B}(x^k,\delta_{k})$. Let $G\!:\mathbb{R}^n\times\mathbb{R}^n\times\mathbb{R}_{++}^m\to\mathbb{R}^m$ be the mapping defined by 
\begin{equation}\label{def-Gmap}
G(x,s,L):=g(s)+g'(s)(x\!-\!s)+\frac{1}{2}\|x\!-\!s\|^2L.
\end{equation}
Then, with 
$\Gamma_{\!k}:=\big\{x\in\mathbb{R}^n\,|\,G(x,x^k,L^k)\le 0\big\}$, it holds
$\Gamma_{\!k}\cap\mathbb{B}(x^k,\delta_{k})\subset\Gamma$. That is, the set $\Gamma_{\!k}\cap\mathbb{B}(x^k,\delta_{k})$ can be regarded as an inner approximation to the feasible set $\Gamma$. Motivated by this, we use $\Gamma_{\!k}$ as the feasible set of the $k$th subproblem. To formulate the objective function of the $k$th subproblem, we choose a positive definite (PD) linear mapping $\mathcal{Q}_k\!:\mathbb{R}^n\to\mathbb{R}^n$ that carries partial or whole second-order information of $f$ at $x^k$. In view of Assumption \ref{ass1} (ii), whenever $\lambda_{\rm min}(\mathcal{Q}_k)> {\rm lip}\,\nabla\!f(x^k)$, applying the descent lemma in \cite[Lemma 5.7]{Beck17}, for any $x$ close enough to $x^k$, 
\[
f(x)\le f(x^k)+\langle\nabla\!f(x^k),x-x^k\rangle+\frac{1}{2}\langle x-x^k,\mathcal{Q}_k(x-x^k)\rangle.
\]
In addition, Assumption \ref{ass1} (iii) and \cite[Proposition 8.32]{RW98} entail $\partial(-\psi)(x^k)\ne\emptyset$. Then, with any $\xi^k\in\partial(-\psi)(x^k)\subset-\partial\psi(x^k)$, the convexity of $\psi$ implies that
\begin{equation}\label{psi-cvx}
 \psi(x)\ge \psi(x^k)-\langle\xi^k,x-x^k\rangle\quad\ \forall x\in\mathbb{R}^n.
\end{equation}
The above two equations show that the following strongly convex quadratic function  
\begin{equation}\label{varthetak}
\vartheta_k(x):=f(x^k)+\langle\nabla\!f(x^k)+\xi^k,x-x^k\rangle+\frac{1}{2}\langle x-x^k,\mathcal{Q}_k(x-x^k)\rangle-\psi(x^k)\quad\forall x\in\mathbb{R}^n
\end{equation}
is a local majorization of the nonconvex and nonsmooth function $f-\psi$ at $x^k$. Our method seeks a stationary point of \eqref{prob} by solving a series of strongly convex problems 
\begin{equation}\label{subprob}
	\min_{x\in\Gamma_{\!k}}\,F_{k}(x):=\vartheta_k(x)+\phi(x).
\end{equation}

Considering that the approximation set $\Gamma_k$ is still an intersection of $m$ closed balls, we continue to use the name ``moving balls approximation'' for our method. Unlike the existing MBA methods, our subproblems capture the second-order information of $f$ at $x^k$ via a variable metric PD linear mapping $\mathcal{Q}_k$, rather than the Lipschitz constant of its gradient. Another notable difference from them is that our MBA method requires in each iteration an inexact solution $y^k$ of subproblem \eqref{subprob}, rather than its exact one. This is extremely significant in practice because the exact solution, denoted by $\overline{x}^k$, is not available due to computation error or computation cost. Then, how to design an implementable inexactness criterion becomes a challenging task. 

For the inexact solution $y^k$ of \eqref{subprob}, it is natural to require its objective value $F_k(y^k)$ to be neither far from the optimal value $F_k(\overline{x}^k)$ of \eqref{subprob} nor greater than the current objective value $F(x^k)=F_k(x^k)$ of \eqref{prob}. However, such a restriction is not enough to achieve the desirable convergence results due to the lack of multiplier information. To address this issue, we require that the KKT residual at $y^k$ with respect to (w.r.t.) some $(v^k,\lambda^k)\in\partial\phi(y^k)\times\mathbb{R}_{+}^m$ can be controlled by the difference between $y^k$ and $x^k$. Consider that under a mild CQ, $\overline{x}^k$ is the unique optimal solution to \eqref{subprob} iff there exist $\overline{v}^k\in\partial\phi(\overline{x}^k)$ and $\overline{\lambda}^k\!\in\mathcal{N}_{\mathbb{R}_{-}^m}(G(\overline{x}^k,x^k,L^k))$ such that $\nabla\vartheta(\overline{x}^k)+\overline{v}^k+\nabla_{\!x} G(\overline{x}^k,x^k,L^k)\overline{\lambda}^k=0$. We define the KKT residual function $R_k\!:\mathbb{R}^n\times\mathbb{R}^n\times\mathbb{R}^m_+\to\mathbb{R}$ of \eqref{subprob} by
\begin{align}\label{resi-sub}
 R_k(x,v,\lambda)&:=\big(\langle x,\nabla\vartheta_k(x)\!+\!v\!+\!\nabla_{\!x} G(x,x^k,L^k)\lambda\rangle\big)_{+}+\big(\!-\!\langle\lambda,G(x,x^k,L^k)\rangle\big)_{+}\nonumber\\
&\quad\ +\big\|\nabla\vartheta_k(x)\!+\!v\!+\!\nabla_{\!x} G(x,x^k,L^k)\lambda\big\|+\big\|[G(x,x^k,L^k)]_{+}\big\|_{\infty}.
\end{align}
According to the above discussions, we adopt the conditions \eqref{inexact1}-\eqref{inexact2} as the inexactness criterion for $y^k$. Although it involves the unknown optimal value $F_{k,j}(\overline{x}^{k,j})$, as will be discussed in Remark \ref{remark-alg} (b) below, a lower bound of $F_{k,j}(\overline{x}^{k,j})$ can be achieved by solving the dual of \eqref{subprob}. Thus, our inexactness criterion is implementable. 

Recall that $\vartheta_{k}$ is a local majorization of $f-\psi$ at $x^{k}$ when $\lambda_{\rm min}(\mathcal{Q}_k)>{\rm lip}\,{\nabla\!f}(x^k)$, and the set $\Gamma_k$ near $x^k$ is a local inner approximation of the feasible set $\Gamma$. At the $k$th iteration, the approximation effect of subproblem \eqref{subprob} to the original problem \eqref{prob} depends on that of $\lambda_{\rm min}(\mathcal{Q}_k)$ to ${\rm lip}\,{\nabla\!f}(x^k)$ and $L^{k}$ to $\ell_{\nabla g}(x^k)$ from above, and the latter also has a great influence on the efficiency of algorithms to solve subproblem \eqref{subprob}. Since in many scenarios ${\rm lip}\,{\nabla\!f}(x^k)$ and $\ell_{\nabla g}(x^k)$ are unknown, our inexact MBA method in each iteration searches for tight upper estimations $\lambda_{\rm min}(\mathcal{Q}_k)$ and $L^k$, and meanwhile solves the associated subproblem. The iteration steps of our inexact MBA method are described in Algorithm \ref{iMBA} below where, for each $k,j\in\mathbb{N}$, 
\[
 \vartheta_{k,j}(x):=\langle\nabla\!f(x^k)+\xi^k,x\!-\!x^k\rangle+\frac{1}{2}\langle x\!-\!x^k,\mathcal{Q}_{k,j}(x\!-\!x^k)\rangle+f(x^k)-\psi(x^k)\ \ \forall x\in\mathbb{R}^n,
\]
the function $R_{k,j}$ is defined as in \eqref{resi-sub} except that $\vartheta_k$ and $L^k$ are replaced by $\vartheta_{k,j}$ and $L^{k,j}$, respectively, and $\overline{x}^{k,j}$ denotes the exact solution of subproblem \eqref{subprobkj}. 
\begin{algorithm}
 \caption{\label{iMBA}{\bf (Inexact MBA method for problem \eqref{prob})}}
\begin{algorithmic}[1]	
 \State{\textbf{Input}: $0<\!\mu_{\min}\le\mu_{\rm max},0<\!L_{\min}\le\! L_{\max},M>0,\beta_{R}>0,\beta_{F}\!>0,\alpha>0,\tau\!>1$ \hspace*{1.1cm} and an initial $x^0\in\Gamma$.}
	
 \For{$k=0,1,2,\ldots$} 
 \State{Choose $\xi^k\in\partial(-\psi)(x^k)$, $\mu_{k,0}\in[\mu_{\min},\mu_{\max}]$ and $L^{k,0}\in[L_{\min},L_{\max}]$.}		
 \For{$j=0,1,2,\ldots$}
 \State{ Choose a linear mapping $\mathcal{Q}_{k,j}\!:\mathbb{R}^n\to\mathbb{R}^n$ with $\mu_{k,j}\mathcal{I}\preceq \mathcal{Q}_{k,j}\preceq (\mu_{k,j}\!+\!M)\mathcal{I}$ \hspace*{1.1cm} to formulate $\vartheta_{k,j}$, and solve the following minimization problem  
\begin{equation}\label{subprobkj}
 \qquad\min_{x\in\Gamma_{k,j}}F_{k,j}(x):=\vartheta_{k,j}(x)+\phi(x)\ \ {\rm with}\ \Gamma_{k,j}\!:=[G(\cdot,x^k,L^{k,j})]^{-1}(\mathbb{R}_{-}^m)
 \end{equation}
 \hspace*{1.0cm} to seek a $y^{k,j}\!\in\mathbb{R}^n$ such that there are $v^{k,j}\in\partial\phi(y^{k,j}),\lambda^{k,j}\!\in\mathbb{R}^m_+$ satisfying
 \begin{subnumcases}{}\label{inexact1}
  R_{k,j}(y^{k,j},v^{k,j},\lambda^{k,j})\le \frac{\beta_{R}}{2}\|y^{k,j}-x^k\|^2,\\
  \label{inexact2}
  F_{k,j}(y^{k,j})\le F_{k,j}(x^k),\,F_{k,j}(y^{k,j})-F_{k,j}(\overline{x}^{k,j})\le \frac{\beta_{F}}{2}\|y^{k,j}-x^k\|^2.
 \end{subnumcases}
 }
			
 \State{If $g(y^{k,j})\le 0$ and $F(y^{k,j})\le F(x^k)-\frac{\alpha}{2}\|y^{k,j}-x^k\|^2$, set $j_k\!:=j$ and go to \hspace*{1.1cm} step 10; else if $g(y^{k,j})\not\leq 0$, go to step 7; else go to step 8.} 
			
 \State{Set $L^{k,j+1}=\tau L^{k,j}$ and $\mu_{k,j+1}=\mu_{k,j}$.}
			
 \State{Set $L^{k,j+1}=L^{k,j}$ and $\mu_{k,j+1}=\tau\mu_{k,j}$.}		    
\EndFor
		
\State{Set $(\mu_k,L^k)\!:=\!(\mu_{k,j_k},L^{k,j_k}), (\overline{x}^{k},x^{k+1},v^{k+1},\lambda^{k+1})\!:=(\overline{x}^{k,j_k},y^{k,j_k},v^{k,j_k},\lambda^{k,j_k})$  \hspace*{0.4cm} and $\mathcal{Q}_k:=\mathcal{Q}_{k,j_k}$.} 
\EndFor
\end{algorithmic}    
\end{algorithm}
\begin{remark}\label{remark-alg}
 {\bf (a)} Algorithm \ref{iMBA} is well defined under the weakest CQ by Lemma \ref{lemma-welldef}. Its inner loop searches for tight upper estimations $\mu_k$ and $L^k$ respectively for ${\rm lip}\,{\nabla\!f}(x^k)$ and $\ell_{\nabla\!g}(x^k)$ to formulate the subproblem and meanwhile computes an inexact solution $y^{k,j}$ of the subproblem to satisfy \eqref{inexact1}-\eqref{inexact2}. Apparently, a good initialization $\mu_{k,0}$ and $L^{k,0}$ can reduce the computation cost of the inner loop. Inspired by the Barzilai-Borwein rule \cite{Barzilai88} for mimicking the Hessian information of a smooth function \cite{Wright09}, we suggest capturing $\mu_{k,0}$ and $L^{k,0}$ via  
 \begin{align}\label{BB-muk0}
 \mu_{k,0}=\min\bigg\{\max\bigg\{\frac{\|\Delta y^k\|^2}{|\langle\Delta x^k,\Delta y^k\rangle|}, \frac{|\langle\Delta x^k,\Delta y^k\rangle|}{\|\Delta x^k\|^2},\mu_{\rm min}\bigg\},\mu_{\rm max}\bigg\},\qquad\quad\\ 
 \label{BB-Lk0}
 L_i^{k,0}=\min\bigg\{\max\bigg\{\frac{\|\Delta z_i^k\|^2}{|\langle\Delta x^k,\Delta z_i^k\rangle|}, \frac{|\langle\Delta x^k,\Delta z_i^k\rangle|}{\|\Delta x^k\|^2},L_{\rm min}\bigg\}, L_{\rm max}\bigg\}\ \ \forall i\in[m],
 \end{align}
 where $\Delta x^k:=x^k-x^{k-1},\Delta y^k:=\nabla\!f(x^k)-\nabla\!f(x^{k-1})$ and $\Delta z_i^k:=\nabla\!g_i(x^k)-\nabla\!g_i(x^{k-1})$.
	
 \noindent
 {\bf(b)} Our inexactness criterion, consisting of \eqref{inexact1}-\eqref{inexact2}, is different from the one in \cite[Definition 7]{Boob24} for problem \eqref{prob} with $\psi\equiv 0$ and every $g_i$ being a sum of an L-smooth function and a continuous convex function. Our criterion controls the feasibility and complementarity violation of $y^{k,j}$ by leveraging the KKT residual, while the criterion of \cite{Boob24} directly controls the feasibility violation of $y^{k,j}$ but requires $\mathcal{L}_{k,j}(y^{k,j},\overline{\lambda}^{k,j})\le\mathcal{L}_{k,j}(\overline{x}^{k,j},\overline{\lambda}^{k,j})+\epsilon$ to control the complementarity violation of $y^{k,j}$, where $\mathcal{L}_{k,j}$ is the Lagrange function of subproblem \eqref{subprobkj},  $(\overline{x}^{k,j},\overline{\lambda}^{k,j})$ is the exact primal-dual solution of \eqref{subprobkj}, and $\epsilon>0$ is the tolerance of inexactness. Since the inexactness criterion of \cite{Boob24} involves the unknown $(\overline{x}^{k,j},\overline{\lambda}^{k,j})$, it is unclear whether it is workable in practice. Although the second inequality in \eqref{inexact2}, like \cite[Definition 7]{Boob24}, also involves the unknown optimal value $F_{k,j}(\overline{x}^{k,j})$, one can achieve a lower bound $F_{k,j}^{\rm LB}$ for it by solving the dual of \eqref{subprobkj}, so any $y^{k,j}$ with $F_{k,j}(y^{k,j})\!-F_{k,j}^{\rm LB}\le\!\frac{\alpha}{2}\|y^{k,j}\!-\!x^k\|^2$ satisfies \eqref{inexact2}.    

 \noindent
 {\bf(c)} The computation cost of Algorithm \ref{iMBA} at each iteration is not expensive despite involving an inner loop. By Lemma \ref{inner-step}, the number of steps in the inner loop can be quantified in terms of the logarithms of the Lipschitz constants of $\nabla\!f$ and $\nabla\!g_i$ for $i\in[m]$ on a compact set and the parameters $\mu_{k,0},L^{k,0},\beta_{F}$ and $\beta_{R}$. Not only that, the computation cost of each step is cheap because it does not involve any new computation of $\nabla\!f(x^k)$ and $\nabla g_i(x^k)$.    
	
 \noindent
 {\bf(d)} We claim that $x^{k}$ is a strong stationary point of \eqref{prob} in the sense of Definition \ref{sspoint-def} as long as $x^{k}=x^{k-1}$. Indeed, when $x^{k}=x^{k-1}$, by the definition of $x^k$ in step 10, $y^{k-1,j_{k-1}}=x^{k-1}$. Together with the inexactness condition \eqref{inexact1}, there exist $v^k\in\partial\phi(x^k)$ and $\lambda^k\in\mathbb{R}_{+}^m$ such that $R_{k-1,j_{k-1}}(x^{k},v^{k},\lambda^{k})=0$, which by the expression of $R_{k-1,j_{k-1}}$ is equivalent to 
 \begin{subnumcases}{}
 \nabla\vartheta_{k-1,j_{k-1}}(x^{k})+v^{k}+\nabla_{\!x} G(x^{k},x^{k-1},L^{k-1})\lambda^{k}=0,\nonumber\\
 \lambda^{k}\ge 0,\,G(x^{k},x^{k-1},L^{k-1})\le 0,\,\langle\lambda^{k},G(x^{k},x^{k-1},L^{k-1})\rangle=0.\nonumber
 \end{subnumcases}
 Then, using the expressions of $\vartheta_{k-1,j_{k-1}}$ and $G(\cdot,x^{k-1},L^{k-1})$ along with $v^k\in\partial\phi(x^k)$ yields
 \begin{subnumcases}{}
 0\in\nabla\!f(x^{k-1})+\partial\phi(x^{k})+\xi^{k-1}\!+\!\nabla g(x^{k-1})\lambda^{k},\\
 \lambda^{k}\ge 0,\,g(x^{k})\le 0,\,\langle\lambda^{k},g(x^{k})\rangle=0.
 \end{subnumcases}
 This, along with $x^{k}=x^{k-1}$ and $\xi^{k-1}\in\partial(-\psi)(x^{k-1})=\partial(-\psi)(x^{k})$, shows that $x^{k}$ (of course $x^{k-1}$) is a strong stationary point of problem \eqref{prob} in the sense of Definition \ref{sspoint-def}. 
\end{remark}

To ensure that Algorithm \ref{iMBA} is well defined, we need the following assumption, the weakest CQ for the multiplier set of \eqref{subprobkj} to be nonempty in view of Remark \ref{remark-Abadie}, since for each $k,j\in\mathbb{N}$ the system $G(\cdot,x^k,L^{k,j})\in\mathbb{R}_{-}^m$ is smooth and convex.
\begin{assumption}\label{ass2}
 For each $k,j\in\mathbb{N}$, the mapping $\mathcal{G}_{k,j}(\cdot):=G(\cdot,x^k,L^{k,j})-\mathbb{R}_{-}^m$ is subregular at $(\overline{x}^{k,j},0)$. 
\end{assumption}

 Now we prove that Algorithm \ref{iMBA} is well-defined by arguing that the conditions \eqref{inexact1}-\eqref{inexact2} are well-defined and the inner loop stops within a finite number of steps.
\begin{lemma}\label{lemma-welldef}
 Under Assumption \ref{ass2}, for each $k\in\mathbb{N}$ such that $x^k\in\Gamma$ is a non-stationary point, 
 \begin{description}
 \item[(i)] the inexactness conditions \eqref{inexact1}-\eqref{inexact2} are well-defined; 
		
 \item[(ii)] the inner loop of Algorithm \ref{iMBA} stops within a finite number of iterations.
 \end{description}
 Consequently, Algorithm \ref{iMBA} is well-defined with $\{x^k\}_{k\in\mathbb{N}}\subset\Gamma$ and $x^k\in\Gamma_{k,j}$ for each $k,j\in\mathbb{N}$.
\end{lemma}
\begin{proof}
 Let $k\in\mathbb{N}$ be an index such that $x^k\in\Gamma$ is a non-stationary point of problem \eqref{prob}. Note that the feasible set $\Gamma_{k,j}$ of \eqref{subprobkj} is nonempty due to $G(x^k,x^k,L^{k,j})=\!g(x^k)\in\mathbb{R}_{-}^m$, while its objective function $F_{k,j}$ is strongly convex, so its optimal solution $\overline{x}^{k,j}$ exists. 
	
\noindent
 {\bf(i)} Fix any $j\in\mathbb{N}$. To prove that the inexactness conditions \eqref{inexact1}-\eqref{inexact2} are well-defined, it suffices to consider that $\overline{x}^{k,j}\neq x^k$. If not, $x^{k}$ is an optimal solution of subproblem \eqref{subprobkj}, so 
 \[
  0\in\partial F_{k,j}(x^k)+\mathcal{N}_{\Gamma_{k,j}}(x^k)=\nabla\!f(x^k)+\xi^k+\partial\phi(x^k)+\mathcal{N}_{\Gamma_{k,j}}(x^k).
 \]
 By virtue of Assumption \ref{ass2}, $\mathcal{N}_{\Gamma_{k,j}}(x^k)=\nabla_{\!x} G(x^{k},x^k,L^{k,j})\mathcal{N}_{\mathbb{R}_{-}^m}(G(x^{k},x^k,L^{k,j}))$. Note that $\nabla_{\!x}G(x^{k},x^k,L^{k,j})=\nabla g(x^k)$ and  $G(x^{k},x^k,L^{k,j})=g(x^k)$. Then, the above inclusion becomes \[
  0\in\nabla\!f(x^k)+\xi^k+\partial\phi(x^k)+\nabla g(x^{k})\mathcal{N}_{\mathbb{R}_{-}^m}(g(x^{k})). 
 \]
 This by Definition \ref{spoint-def} entails $x^k$ being a stationary point of \eqref{prob}, which is impossible. 
 
 From the optimality condition of \eqref{subprobkj} at $\overline{x}^{k,j}$, we have $0\in\partial F_{k,j}(\overline{x}^{k,j})+\mathcal{N}_{\Gamma_{k,j}}(\overline{x}^{k,j})$ which, by the expression of $F_{k,j}$ and the above arguments, can equivalently be written as
 \[
  0\in\nabla\vartheta_{k,j}(\overline{x}^{k,j})+\partial\phi(\overline{x}^{k,j})+\nabla_{\!x}G(\overline{x}^{k,j},x^k,L^{k,j})\mathcal{N}_{\mathbb{R}_{-}^m}(G(\overline{x}^{k,j},x^k,L^{k,j})). 
 \]
 Then, there exist  $\overline{v}^{k,j}\in\partial\phi(\overline{x}^{k,j})$ and  $\overline{\lambda}^{k,j}\in\mathbb{R}^m$ such that $(\overline{x}^{k,j},\overline{v}^{k,j},\overline{\lambda}^{k,j})$ satisfies  
 \begin{subnumcases}{}\label{KKT1-subprob}
  \nabla\vartheta_{k,j}(\overline{x}^{k,j})+\overline{v}^{k,j}+\nabla_{\!x}G(\overline{x}^{k,j},x^k,L^{k,j})\overline{\lambda}^{k,j}=0,\qquad\qquad\\
 \label{KKT2-subprob}
 \overline{\lambda}^{k,j}\ge 0,\,G(\overline{x}^{k,j},x^k,L^{k,j})\le0,\,\langle G(\overline{x}^{k,j},x^k,L^{k,j}),\overline{\lambda}^{k,j}\rangle=0. 
 \end{subnumcases} 
 Combining \eqref{KKT1-subprob}-\eqref{KKT2-subprob} with $\overline{x}^{k,j}\neq x^k$ leads to $R_{k,j}(\overline{x}^{k,j},\overline{v}^{k,j},\overline{\lambda}^{k,j})=0<\frac{\beta_{R}}{2}\|\overline{x}^{k,j}-x^k\|^2$. Recall that $\overline{v}^{k,j}\in\partial\phi(\overline{x}^{k,j})$ and $\partial\phi$ is osc at $\overline{x}^{k,j}$ by Assumption \ref{ass1} (iii). There must exist a sequence $\{(x^{l},v^{l})\}_{l\in\mathbb{N}}\subset\mathbb{R}^n\times\partial\phi(x^{l})$ with $(x^{l},v^{l})\to(\overline{x}^{k,j},\overline{v}^{k,j})$ as $l\to\infty$. From the continuity of $R_{k,j}(\cdot,\cdot,\overline{\lambda}^{k,j})$, there exist $x$ close to $\overline{x}^{k,j}$ and $v\in\partial\phi(x)$ such that $(x,v,\overline{\lambda}^{k,j})$ satisfies \eqref{inexact1}. Note that $F_{k,j}(\overline{x}^{k,j})<F_{k,j}(x^k)$ because $x^k$ is a feasible solution of \eqref{subprobkj} by $G(x^k,x^k,L^{k,j})=g(x^k)\le 0$, and $F_{k,j}(\overline{x}^{k,j})-F_{k,j}(\overline{x}^{k,j})<\frac{\beta_{F}}{2}\|\overline{x}^{k,j}-x^k\|^2$ due to $\overline{x}^{k,j}\ne x^k$. The continuity of $F_{k,j}(\cdot)$ implies that any $x$ close to $\overline{x}^{k,j}$ satisfies \eqref{inexact2}. Thus, there exist $y^{k,j}$ close to $\overline{x}^{k,j}$ and $v^{k,j}\in\partial\phi(y^{k,j})$ such that $(y^{k,j},v^{k,j},\overline{\lambda}^{k,j})$ satisfies \eqref{inexact1}-\eqref{inexact2}.
	
 \noindent
 {\bf(ii)} Suppose on the contrary that the inner loop does not stop within a finite number of steps. Then, there exists an infinite index set $J\subset\mathbb{N}$ such that either of the following cases occurs: {\bf(a)} $g(y^{k,j})\not\leq 0$ for all $j\in J$; {\bf(b)} $g(y^{k,j})\le 0$ but $F(y^{k,j})>F(x^k)-\frac{\alpha}{2}\|y^{k,j}-x^k\|^2$ for all $j\in J$. We consider these two cases separately.   
	
 \noindent
 {\bf Case (a)} Now there must exist an infinite index set $J_1\subset J$ and an index $i_0\in[m]$ such that $g_{i_0}(y^{k,j})>0$ for all $j\in J_1$. Apparently, $\lim_{J_1\ni j\to\infty}\min\{L_1^{k,j},\ldots,L_{m}^{k,j}\}=\infty$ follows step 7 of Algorithm \ref{iMBA}. In view of Assumption \ref{ass1} (i) and the descent lemma (see \cite[Lemma 5.7]{Beck17}) to $g_{i_0}$, there exist $\gamma_{i_0}^k\!>0$ and a closed ball $\mathcal{V}(x^k)$ centered at $x^k$ such that 
 \begin{equation}\label{lipgi}
 g_{i_0}(y)\le g_{i_0}(x)+\langle\nabla g_{i_0}(x),y-x\rangle+\frac{\gamma_{i_0}^k}{2}\|y-x\|^2\quad\forall x,y\in\mathcal{V}(x^k).
 \end{equation}
 From the inexact condition \eqref{inexact1} and the expression of $R_{k,j}$, for each $j\in J_1$, it holds
 \begin{align*}
 \frac{\beta_{R}}{2}\|y^{k,j}-x^k\|^2
 &\stackrel{\eqref{inexact1}}{\ge} R_{k,j}(y^{k,j},v^{k,j},\lambda^{k,j})\ge[G_{i_0}(y^{k,j},x^k,L^{k,j})]_{+}\ge G_{i_0}(y^{k,j},x^k,L^{k,j})\nonumber\\
 &\ge g_{i_0}(x^k)-\|\nabla\!g_{i_0}(x^k)\|\|y^{k,j}-x^k\|+\frac{1}{2}L_{i_0}^{k,j}\|y^{k,j}-x^k\|^2,
 \end{align*}
 where $G_{i_0}$ is the $i_0$th component of the mapping $G$. This along with $\lim_{J_1\ni j\to\infty}L_{i_0}^{k,j}=\infty$ implies $\lim_{J_1\ni j\to\infty}\|y^{k,j}-x^k\|=0$. Then, there exists $\overline{j}\in\mathbb{N}$ such that for all $J_1\ni j\ge\overline{j}$, $y^{k,j}\in\mathcal{V}(x^k)$ and $L_{i_0}^{k,j}>\gamma_{i_0}^k+\beta_{R}$. Using \eqref{lipgi} with $y=y^{k,j}$ and $x=x^k$ for $J_1\ni j\ge\overline{j}$ gives
 \begin{align*}
 g_{i_0}(y^{k,j})&\le g_{i_0}(x^k)+\langle\nabla g_{i_0}(x^k),y^{k,j}\!-\!x^k\rangle+\frac{L_{i_0}^{k,j}}{2}\|y^{k,j}\!-\!x^k\|^2+\frac{\gamma_{i_0}^k\!-\!L_{i_0}^{k,j}}{2}\|y^{k,j}\!-\!x^k\|^2\\
 &\le \|(G(y^{k,j},x^k,L^{k,j}))_{+}\|_{\infty}-\frac{\beta_{R}}{2}\|y^{k,j}-x^k\|^2\\
 &\le R_{k,j}(y^{k,j},v^{k,j},\lambda^{k,j})-\frac{\beta_{R}}{2}\|y^{k,j}-x^k\|^2\stackrel{\eqref{inexact1}}{\le}0,
 \end{align*}
 which is a contradiction to the fact that $g_{i_0}(y^{k,j})>0$ for all $j\in J_1$. 
	
 \noindent
 {\bf Case (b)} In this case, $\lim_{J\ni j\to\infty}\mu_{k,j}=\infty$ follows step 8 of Algorithm \ref{iMBA}. In view of Assumption \ref{ass1} (ii) and $x^k\in\Gamma$, there exist $\gamma_{f}^k>0$ and a closed ball $\mathcal{V}(x^k)$ of $x^k$ such that 
 \begin{equation}\label{fLip}
  f(y)\le f(x)+\langle\nabla\!f(x),y-x\rangle+\frac{\gamma_{f}^k}{2}\|y-x\|^2\quad\forall x,y\in\mathcal{V}(x^k).
 \end{equation}
 From the first inequality of \eqref{inexact2} and the expressions of $F_{k,j}$ and $\vartheta_{k,j}$, for each $j\in J$, it holds
 \begin{align}\label{ineq-wdef2}
  \frac{1}{2}\langle y^{k,j}\!-\!x^k,\mathcal{Q}_{k,j}(y^{k,j}\!-\!x^k)\rangle
  &\le \phi(x^k)-\phi(y^{k,j})-\langle\nabla\!f(x^k)+\xi^k,y^{k,j}\!-\!x^k\rangle\\
  &\le-\langle v^k,y^{k,j}\!-\!x^k\rangle-\langle\nabla\!f(x^k)+\xi^k,y^{k,j}\!-\!x^k\rangle\nonumber\\
  &\le\|\nabla\!f(x^k)\!+\!\xi^k\!+\!v^k\|\|y^{k,j}\!-\!x^k\|,\nonumber
 \end{align}
 where the second inequality is by the convexity of $\phi$ and $v^k\in\partial\phi(x^k)$. 
 Since $\mathcal{Q}_{k,j}\succeq\mu_{k,j}\mathcal{I}$ for all $j\in J$ and $\lim_{J\ni j\to\infty}\mu_{k,j}=\infty$, the above inequality implies $\lim_{J\ni j\to\infty}y^{k,j}=x^k$. Then, there exists $\overline{j}\in\mathbb{N}$ such that for all $J\ni j\ge\overline{j}$, $y^{k,j}\in\mathcal{V}(x^k)$ and $\mu_{k,j}>\gamma_{f}^k+\alpha$. Recall that $F(y^{k,j})>F(x^k)-\frac{\alpha}{2}\|y^{k,j}\!-x^k\|^2$ for all $j\in J$. Then, for all $J\ni j\ge\overline{j}$, it holds 
 \begin{align*}
 \frac{\alpha}{2}\|y^{k,j}\!-\!x^k\|^2&>f(x^k)-f(y^{k,j})+\phi(x^k)-\phi(y^{k,j})-\psi(x^k)+\psi(y^{k,j})\\
 &\stackrel{\eqref{fLip}}{\ge}\!-\langle\nabla\!f(x^k),y^{k,j}\!-\!x^k\rangle\!-\!\frac{\gamma_{f}^k}{2}\|y^{k,j}\!-\!x^k\|^2+\phi(x^k)-\phi(y^{k,j})\!-\!\psi(x^k)+\psi(y^{k,j})\\
 &\stackrel{\eqref{psi-cvx}}{\ge}-\langle\nabla\!f(x^k)+\xi^k,y^{k,j}\!-\!x^k\rangle-\frac{\gamma_{f}^k}{2}\|y^{k,j}\!-\!x^k\|^2+\phi(x^k)-\phi(y^{k,j})\\
 &\stackrel{\eqref{ineq-wdef2}}{\ge}\frac{1}{2}\langle y^{k,j}\!-\!x^k,\mathcal{Q}_{k,j}(y^{k,j}\!-\!x^k)\rangle-\frac{\gamma_{f}^k}{2}\|y^{k,j}\!-\!x^k\|^2\ge\frac{\mu_{k,j}-\gamma_{f}^k}{2}\|y^{k,j}\!-\!x^k\|^2,
 \end{align*}
 which is a contradiction to $\mu_{k,j}>\gamma_{f}^k+\alpha$ due to $y^{k,j}\neq x^k$ for all $J\ni j\ge\overline{j}$. 
	
 From the above two cases, we conclude that the inner loop must stop within a finite number of iterations. Note that Algorithm \ref{iMBA} starts with $x^0\in\Gamma$, and the inner loop returns the point $y^{k,j_k}=x^{k+1}\in\Gamma$. Then, the arguments by induction can prove that for each $k\in\mathbb{N}$,  subproblem \eqref{subprobkj} and the inexact conditions \eqref{inexact1}-\eqref{inexact2} are both well-defined. Thus, Algorithm \ref{iMBA} is well-defined with $\{x^k\}_{k\in\mathbb{N}}\subset\Gamma$ and $x^k\in\Gamma_{k,j}$ for each $k,j\in\mathbb{N}$.
 \end{proof}

 Lemma \ref{lemma-welldef} (ii) states the finite termination of the inner loop of Algorithm \ref{iMBA}. In fact, the number of steps can be quantified when $\nabla\!f$ is strictly continuous on the whole space $\mathbb{R}^n$; see Lemma \ref{inner-step}.
 From the iteration steps of the outer loop in Algorithm \ref{iMBA}, 
\begin{equation}\label{F-decrease}
 F(x^{k+1})\le F(x^k)-\frac{\alpha}{2}\|x^{k+1}-x^k\|^2\quad\ \forall k\in\mathbb{N}.
\end{equation}
Then, the lower boundedness of $F$ on $\Gamma$ by Assumption \ref{ass1} (iv) implies the convergence of $\{F(x^k)\}_{k\in\mathbb{N}}$, so $\sum_{k=0}^{\infty}\|x^{k+1}\!-x^k\|^2<\infty$ follows. 
That is, the following result holds.
\begin{corollary}\label{corollary-objF}
 Under Assumption \ref{ass2}, for each $k\in\mathbb{N}$, the above \eqref{F-decrease} holds. Consequently, the sequence $\{F(x^k)\}_{k\in\mathbb{N}}$ converges to some $\varpi^*\in\mathbb{R}$ and $\sum_{k=0}^{\infty}\|x^{k+1}\!-x^k\|^2<\infty$.
\end{corollary}


\section{Convergence analysis}\label{sec4}

This section is dedicated to the convergence analysis of Algorithm \ref{iMBA} under Assumption \ref{ass3}, a common one for analyzing convergence of iterate sequences. This assumption holds automatically when $F$ is level-coercive relative to the set $\Gamma$ (see \cite[Definition 3.25]{RW98}) due to the inclusion $\{x^k\}_{k\in\mathbb{N}}\subset\mathcal{L}_{F(x^0)}\!:=\{x\in\Gamma\,|\,F(x)\le F(x^0)\}$ implied by \eqref{F-decrease}. 
\begin{assumption}\label{ass3}
$\{x^k\}_{k\in\mathbb{N}}$ is bounded, and denote the set of its cluster points as $\omega(x^0)$.
\end{assumption} 

Under Assumption \ref{ass3}, we can obtain the following two technical lemmas. Lemma \ref{lemma-xbark} states the sufficient closeness of the iterate sequence $\{x^k\}_{k\in\mathbb{N}}$ to the sequence $\{\overline{x}^k\}_{k\in\mathbb{N}}$, while Lemma \ref{lemma-bound} establishes the boundedness of the sequence $\{(\xi^k,\overline{x}^k,L^k,\mathcal{Q}_k)\}_{k\in\mathbb{N}}$. These two lemmas are often used in the subsequent analysis of this section. 
\begin{lemma}\label{lemma-xbark}
 Under Assumptions \ref{ass2}-\ref{ass3}, the set $\omega(x^0)$ is nonempty and compact, and  \begin{equation*}
 \lim_{k\to\infty}\|x^k-\overline{x}^k\|=0\ \ {\rm and}\ \ \lim_{k\to\infty}\|x^{k+1}-\overline{x}^k\|=0.
 \end{equation*}
\end{lemma}  
\begin{proof}
 It suffices to prove the first limit because the second follows the first one and Corollary \ref{corollary-objF}. For each $k\in\mathbb{N}$, in view of the optimality of $\overline{x}^k$ to \eqref{subprobkj}, there exists $\overline{v}^k\in\partial\phi(\overline{x}^k)$ such that $-\nabla\vartheta_k(\overline{x}^k)-\overline{v}^k\in\mathcal{N}_{\Gamma_k}(\overline{x}^k)$. Along with $x^k\!\in\Gamma_k$ by Lemma \ref{lemma-welldef} and the convexity of $\Gamma_k$, 
 \begin{equation*}
  0\ge \langle\nabla\vartheta_k(\overline{x}^k)+\overline{v}^k,\overline{x}^k-x^k\rangle=\langle\nabla\!f(x^k)+\xi^k+\overline{v}^k+\mathcal{Q}_k(\overline{x}^k\!-\!x^k),\overline{x}^k\!-\!x^k\rangle.
 \end{equation*}
 Combining this inequality with the expression of $F_k$ and $\overline{v}^k\in\partial\phi(\overline{x}^k)$ leads to 
 \begin{align*}
  F_k(x^k)-F_k(\overline{x}^k)&=\phi(x^k)-\langle\nabla\!f(x^k)+\xi^k,\overline{x}^k\!-\!x^k\rangle-\frac{1}{2}\langle\overline{x}^k\!-\!x^k,\mathcal{Q}_k(\overline{x}^k\!-\!x^k)\rangle-\phi(\overline{x}^k)\\
  &\ge\frac{1}{2}\langle\overline{x}^k-x^k,\mathcal{Q}_k(\overline{x}^k-x^k)\rangle+\phi(x^k)-\phi(\overline{x}^k)+\langle\overline{v}^k,\overline{x}^k-x^k\rangle\\
  &\ge \frac{1}{2}\langle\overline{x}^k-x^k,\mathcal{Q}_k(\overline{x}^k-x^k)\rangle\ge\frac{\mu_{\min}}{2}\|\overline{x}^k-x^k\|^2\quad\forall k\in\mathbb{N}.
 \end{align*}
 From $F_k(x^{k+1})-F_k(\overline{x}^k)\le\frac{\beta_{F}}{2}\|x^{k+1}-x^k\|^2$ by the second inequality of \eqref{inexact2} for $j=j_k$, 
 \begin{align*}
  \frac{\mu_{\min}}{2}\|\overline{x}^k-x^k\|^2&\le F_k(x^k)-F_k(x^{k+1})+\frac{\beta_{F}}{2}\|x^{k+1}-x^k\|^2\\
  &\le \max_{\zeta\in\partial F_k(x^k)}\|\zeta\|\|x^{k+1}-x^k\|+\frac{\beta_{F}}{2}\|x^{k+1}-x^k\|^2,
 \end{align*} 
 where the second inequality is due to the convexity of $F_k$. Note that $\partial F_k(x^k)=\nabla\!f(x^k)+\xi^k+\partial\phi(x^k)$ with $\xi^k\in\partial(-\psi)(x^k)$. Under Assumption \ref{ass3}, the local boundedness of $\partial(-\psi)$ and $\partial\phi$ (see Remark \ref{remark-subdiff}) and \cite[Proposition 5.15]{RW98} imply that the sets $\bigcup_{j\in\mathbb{N}}\partial(-\psi)(x^j)$ and $\bigcup_{j\in\mathbb{N}}\partial\phi(x^j)$ are bounded, so is the set $\bigcup_{j\in\mathbb{N}}\partial F_j(x^j)$. Then, from the above inequality and $\lim_{k\to\infty}\|x^{k+1}-x^k\|=0$ by Corollary \ref{corollary-objF}, we obtain $\lim_{k\to\infty}\|\overline{x}^{k}-x^k\|=0$.   
\end{proof}
\begin{remark}\label{bound-subdiff}
 Under Assumption \ref{ass3}, the sets $\bigcup_{j\in\mathbb{N}}\partial(-\psi)(x^j)$ and $ \bigcup_{j\in\mathbb{N}}\partial\phi(x^j)$ are both bounded by the proof of Lemma \ref{lemma-xbark}. From the first limit in Lemma \ref{lemma-xbark} and Remark \ref{remark-subdiff}, so is $\bigcup_{j\in\mathbb{N}}\partial F_j(\overline{x}^j)$.
\end{remark}
\begin{lemma}\label{lemma-bound}
 Under Assumptions \ref{ass2}-\ref{ass3}, the following three statements hold true. 
 \begin{description}
 \item[(i)] The sequence $\{(\overline{x}^k,\xi^k)\}_{k\in\mathbb{N}}$ is bounded.

 \item[(ii)] There exists $\beta_{L}>0$ such that $\|L^k\|\le\beta_{L}$ for all $k\in\mathbb{N}$.
		
 \item[(iii)] $\{\mu_k\}_{k\in\mathbb{N}}$ is bounded, so there exists $\beta_{\mathcal{Q}}>0$ such that $\|\mathcal{Q}_k\|\!\le\beta_{\mathcal{Q}}$ for all $k\in\mathbb{N}$. 
\end{description}
\end{lemma}
\begin{proof} 
 {\bf (i)-(ii)} Item (i) directly follows Assumption \ref{ass3}, Lemma \ref{lemma-xbark} and Remark \ref{bound-subdiff}. We next prove item (ii). Suppose on the contrary that $\{L^k\}_{k\in\mathbb{N}}$ is unbounded. For each $k\in\mathbb{N}$, let $\widehat{L}^k:=L^k/\tau$. Then, $\lim_{ k\to\infty}\max\{\widehat{L}_1^k,\ldots,\widehat{L}_m^k\}=\infty$, and there exist an infinite index set $\mathcal{K}\subset\mathbb{N}$ and an index $i_0\in[m]$ such that $\lim_{\mathcal{K}\ni k\to\infty}\widehat{L}_{i_0}^k=\infty$. From steps 6-8 of Algorithm \ref{iMBA}, there exist an infinite index set $\mathcal{K}_1\subset\mathcal{K}$ and an index $\overline{k}\in\mathbb{N}$ such that for each $\mathcal{K}_1\ni k\ge\overline{k}$, $\widehat{L}_{i_0}^k=L_{i_0}^{k,\widehat{j}_k}$ and $g_{i_0}(y^{k,\widehat{j}_k})>0$ for some $1<\widehat{j}_k<j_k$. Then, for each $\mathcal{K}_1\ni k\ge\overline{k}$, 
 \begin{align}\label{ineq-Lbound1}
 0&\stackrel{\eqref{inexact1}}{\ge}\|[G(y^{k,\widehat{j}_k},x^k,\widehat{L}^k)]_{+}\|_{\infty}-({\beta_{R}}/{2})\|y^{k,\widehat{j}_k}-x^k\|^2\nonumber\\
 &\ge g_{i_0}(x^k)+\langle\nabla\!g_{i_0}(x^k),y^{k,\widehat{j}_k}-x^k\rangle+\frac{1}{2}(\widehat{L}^k_{i_0}\!-\!\beta_{R})\|y^{k,\widehat{j}_k}-x^k\|^2.
 \end{align}
 For each $\mathcal{K}_1\ni k\ge\overline{k}$, from the mean-value theorem, there exists $z^k$ on the line segment joining $x^k$ and $y^{k,\widehat{j}_k}$ such that    
 $g_{i_0}(x^k)-g_{i_0}(y^{k,\widehat{j}_k})=\langle\nabla\!g_{i_0}(z^k),x^k\!-\!y^{k,\widehat{j}_k}\rangle$. Along with the above inequality and $g_{i_0}(y^{k,\widehat{j}_k})>0$, it follows that for each $\mathcal{K}_1\ni k\ge\overline{k}$,
\begin{align}\label{temp-gineq}
0&>g_{i_0}(x^k)-g_{i_0}(y^{k,\widehat{j}_k})+\langle\nabla\!g_{i_0}(x^k),y^{k,\widehat{j}_k}-x^k\rangle+\frac{\widehat{L}^k_{i_0}-\beta_{R}}{2}\|y^{k,\widehat{j}_k}-x^k\|^2\\
 &=\langle\nabla\!g_{i_0}(x^k)-\nabla\!g_{i_0}(z^k),y^{k,\widehat{j}_k}-x^k\rangle+\frac{1}{2}(\widehat{L}^k_{i_0}-\beta_{R})\|y^{k,\widehat{j}_k}-x^k\|^2\nonumber\\
 &\ge -\|\nabla\!g_{i_0}(x^k)-\nabla\!g_{i_0}(z^k)\|\|y^{k,\widehat{j}_k}-x^k\|+\frac{1}{2}(\widehat{L}^k_{i_0}-\beta_{R})\|y^{k,\widehat{j}_k}-x^k\|^2.\nonumber
 \end{align}
 From the above \eqref{ineq-Lbound1} and $\lim_{\mathcal{K}_1\ni k\to\infty}\widehat{L}^k_{i_0}=\infty$, we deduce from the boundedness of $\{x^k\}_{k\in\mathbb{N}}$ that $\lim_{\mathcal{K}_1\ni k\to\infty}\|y^{k,\widehat{j}_k}-x^k\|=0$.
 Since the sequence $\{x^k\}_{k\in\mathcal{K}_1}$ is bounded, if necessary by taking a subsequence, we can assume $\lim_{\mathcal{K}_1\ni k\to\infty}x^k=x^*\in\omega(x^0)$. Then, $\lim_{\mathcal{K}_1\ni k\to\infty}z^k=\lim_{\mathcal{K}_1\ni k\to\infty}y^{k,\widehat{j}_k}=x^*$. In view of the strict continuity of $\nabla g_{i_0}$ at $x^*$, there exists $\gamma_{i_0}>0$ such that $\|\nabla\!g_{i_0}(x^k)-\nabla\!g_{i_0}(z^k)\|\le \gamma_{i_0}\|y^{k,\widehat{j}_k}-x^k\|$ for all $\mathcal{K}_1\ni k\ge\overline{k}$ (if necessary by increasing $\overline{k}$), which along with the above inequality implies that for all $\mathcal{K}_1\ni k\ge\overline{k}$,
 \[
  \widehat{L}^k_{i_0}\|y^{k,\widehat{j}_k}-x^k\|^2<(\beta_{R}+2\gamma_{i_0})\|y^{k,\widehat{j}_k}-x^k\|^2.
 \]
 This contradicts $\lim_{\mathcal{K}_1\ni k\to\infty}\widehat{L}^k_{i_0}=\infty$ because $\|y^{k,\widehat{j}_k}\!-\!x^k\|\ne 0$ for all $\mathcal{K}_1\ni k\ge\overline{k}$ by \eqref{temp-gineq}. 
 
 \noindent
 {\bf(iii)} Suppose on the contrary that $\{\mu_k\}_{k\in\mathbb{N}}$ is unbounded. For each $k\in\mathbb{N}$, let $\widehat{\mu}_k:=\tau^{-1}\mu_k$. From step 8 of Algorithm \ref{iMBA}, there exists an index $\overline{k}\in\mathbb{N}$ such that for each $k\ge\overline{k}$, 
 \begin{equation}\label{ineq-mubound0}
 \widehat{\mu}_k=\mu_{k,\widehat{j}_k}\ {\rm and}\  F(y^{k,\widehat{j}_k})>F(x^k)-\frac{\alpha}{2}\|y^{k,\widehat{j}_k}-x^k\|^2\quad{\rm for\ some}\ 1<\widehat{j}_k<j_k.
 \end{equation}
 For each $k\ge\overline{k}$, using the first inequality of \eqref{inexact2} for $j=\widehat{j}_k$ and $\mathcal{Q}_{k,\widehat{j}_k}\succeq\widehat{\mu}_k\mathcal{I}$ leads to 
 \begin{align*}
 \frac{1}{2}\widehat{\mu}_k\|y^{k,\widehat{j}_k}\!-\!x^k\|^2&\le\frac{1}{2}\langle  y^{k,\widehat{j}_k}\!-\!x^k,\mathcal{Q}_{k,\widehat{j}_k}(y^{k,\widehat{j}_k}\!-\!x^k)\rangle\nonumber\\
 &\le\langle\xi^k+\nabla\!f(x^k),x^k-y^{k,\widehat{j}_k}\rangle+\phi(x^k)-\phi(y^{k,\widehat{j}_k})\\
 &\le\langle\xi^k+\nabla\!f(x^k)-v^k,x^k-y^{k,\widehat{j}_k}\rangle\ \ {\rm with}\ v^k\in\partial\phi(x^k),
 \end{align*}
 where the third inequality is due to the convexity of $\phi$. Note that $\{\xi^k\!+\!\nabla\!f(x^k)-v^k\}_{k\in\mathbb{N}}$ is bounded due to Remark \ref{bound-subdiff}, and $y^{k,\widehat{j}_k}\ne x^k$ for each $k\ge\overline{k}$ by \eqref{ineq-mubound0}. The above inequality, along with $\lim_{k\to\infty}\widehat{\mu}_k=\infty$, implies $\lim_{k\to\infty}\|y^{k,\widehat{j}_k}\!-\!x^k\|=0$. By Assumption \ref{ass1} (ii) and the mean-value theorem, for each $k\ge\overline{k}$, there exists $z^k$ on the line segment joining $x^k$ and $y^{k,\widehat{j}_k}$ such that $f(x^k)-f(y^{k,\widehat{j}_k})= \langle \nabla\!f(z^k),x^k\!-\!y^{k,\widehat{j}_k}\rangle$. 	Together with the inequality in  \eqref{ineq-mubound0}, for each $k\ge\overline{k}$ (if necessary by increasing $\overline{k}$), it holds
 \begin{align}\label{ineq-mubound2}
 \frac{\alpha}{2}\|y^{k,\widehat{j}_k}-x^k\|^2
  &>\langle \nabla\!f(z^k),x^k\!-\!y^{k,\widehat{j}_k}\rangle-\psi(x^k)+\psi(y^{k,\widehat{j}_k})+\phi(x^k)-\phi(y^{k,\widehat{j}_k})\nonumber\\
  &\stackrel{\eqref{psi-cvx}}{\ge}\langle \nabla\!f(z^k),x^k\!-\!y^{k,\widehat{j}_k}\rangle+\langle\xi^k,x^k\!-\!y^{k,\widehat{j}_k}\rangle+\phi(x^k)-\phi(y^{k,\widehat{j}_k})\\
  &\stackrel{\eqref{ineq-mubound2}}{\ge}\langle \nabla\!f(z^k)-\nabla f(x^k),x^k\!-\!y^{k,\widehat{j}_k}\rangle+\frac{\widehat{\mu}_k}{2}\|y^{k,\widehat{j}_k}\!-\!x^k\|^2.\nonumber
  \end{align}
 From the boundedness of $\{x^k\}_{k\in\mathbb{N}}$, if necessary by taking a subsequence, we assume that $\lim_{k\to\infty}x^k=x^*\in\omega(x^0)$. Recalling that $\lim_{k\to\infty}\|y^{k,\widehat{j}_k}\!-\!x^k\|=0$, we have $\lim_{k\to\infty}y^{k,\widehat{j}_k}=x^*=\lim_{k\to\infty}z^k$. In view of Assumption \ref{ass1} (ii), there exists $\gamma_{\!f}>0$ such that for all $k\ge\overline{k}$ (if necessary by increasing $\overline{k}$), $\|\nabla\!f(z^k)-\nabla\!f(x^k)\|\le \gamma_{f}\|y^{k,\widehat{j}_k}-x^k\|$. Then, from the above inequality, we have $\frac{\widehat{\mu}_k}{2}\|y^{k,\widehat{j}_k}\!-\!x^k\|^2\le \frac{\alpha+2\gamma_{f}}{2}\|y^{k,\widehat{j}_k}\!-\!x^k\|^2$ for each $k\ge\overline{k}$.
 This is impossible because $\lim_{k\to\infty}\widehat{\mu}_k=\infty$ and  $y^{k,\widehat{j}_k}\neq x^k$ for all $k\ge\overline{k}$ by \eqref{ineq-mubound0}. Thus, $\{\mu_k\}_{k\in\mathbb{N}}$ is bounded, so is $\{\mathcal{Q}_k\}_{k\in\mathbb{N}}$ by recalling that $\mu_k\mathcal{I}\preceq \mathcal{Q}_k\preceq(\mu_k\!+\!M)\mathcal{I}$. 
\end{proof} 
\subsection{Subsequential convergence}\label{sec4.1}

Let $\mathbb{U}:=\mathbb{R}_{+}^m\times\mathbb{S}\times\mathbb{R}^n\times\mathbb{R}^n$, where $\mathbb{S}$ denotes the set of all PD linear mappings from $\mathbb{R}^n$ to $\mathbb{R}^n$. For any $(u,x)\in\mathbb{U}\times\mathbb{R}^n$ with $u=(L,\mathcal{Q},\xi,s)$, define the functions 
\begin{subnumcases}{}\label{theta-def1}
\theta(u,x):=f(s)+\langle\nabla\!f(s),x-s\rangle+\frac{1}{2}\langle x\!-\!s,\mathcal{Q}(x\!-\!s)\rangle-\psi(s)+\langle\xi,x\!-\!s\rangle,\\
\label{h-def1}
 H(u,x):=G(x,s,L).    
\end{subnumcases} For each $k\in\mathbb{N}$, let $u^k\!:=(L^k,\mathcal{Q}_k,\xi^k,x^k)\in\mathbb{U}$. By Lemma \ref{lemma-bound}, under Assumption \ref{ass3}, the sequence $\{(u^k,\overline{x}^{k})\}_{k\in\mathbb{N}}$ is bounded, so its cluster point set, denoted by $\Omega^*$, is nonempty and compact. Furthermore, every $(u^*,\overline{x}^*)\in \Omega^*$ is of the form $(L^*,\mathcal{Q}^*,\xi^*,x^*,\overline{x}^*)$ with $\overline{x}^*=x^*$ by Lemma \ref{lemma-xbark}. 
This section aims at proving that every accumulation point of $\{x^k\}_{k\in\mathbb{N}}$ is a strong stationary point of problem \eqref{prob} under the following assumption.
\begin{assumption}\label{ass4}
 The constraint system of \eqref{para-prob} with $\theta$ and $H$ defined in \eqref{theta-def1}-\eqref{h-def1} satisfies the partial BMP w.r.t. $x$ at all $(u^*,\overline{x}^*)\in \Omega^*$, where $\Omega^*$ is the cluster point set of $\{(u^k,\overline{x}^{k})\}_{k\in\mathbb{N}}$.    
\end{assumption}
\begin{remark}\label{remark-ass4}
 According to \cite[Theorem 3.3 $\&$ Corollary 3.7]{Gfrerer17}, Assumption \ref{ass4} holds if for every $(u^*,\overline{x}^*)\in \Omega^*$ the mapping $\mathcal{H}_{u^*}(\cdot)\!:=H(u^*,\cdot)-\mathbb{R}_{-}^m=G(\cdot,\overline{x}^*,L^*)-\mathbb{R}_{-}^m$ is metrically regular at $(\overline{x}^*,0)$. The latter is weaker than the Slater's CQ for the constraint system $G(\cdot,\overline{x}^*,L^*)\in\mathbb{R}_{-}^m$ by \cite[Proposition 2.104 \& 2.106]{BS00}. Moreover, following the proof of \cite[Proposition 2.1]{Auslender10}, we can prove that the Slater's CQ for the constraint system $G(\cdot,\overline{x}^*,L^*)\in\mathbb{R}_{-}^m$ is implied by the MFCQ for the constraint system $g(x)\in\mathbb{R}_{-}^m$ at $\overline{x}^*$. Thus, Assumption \ref{ass4} is weaker than the MFCQ for the system $g(x)\in\mathbb{R}_{-}^m$ at all feasible points, a common CQ to analyze the convergence or iteration complexity of MBA-like methods (see, e.g., \cite[Assumption A3]{Auslender10}, \cite[Assumption 2.3]{YuPongLv21}, \cite[Assumption 3]{Boob24} and \cite[Assumption 3]{Nabou24}). In addition, from \cite[Proposition 3.2]{Gfrerer17}, Assumption \ref{ass4} is also implied by the partial constant rank constraint qualification w.r.t. $x$ for the constraint system $G(x,s,L)\in\mathbb{R}_{-}^m$ at all $(u^*,\overline{x}^*)\in \Omega^*$, which is known to be independent of the MFCQ of this system at the corresponding reference point. 
\end{remark}

From \eqref{theta-def1}-\eqref{h-def1}, we have
$\vartheta_k(\cdot)\!=\theta(u^k,\cdot)$ and $G(\cdot,x^k,L^k)\!=H(u^k,\cdot)$ for all $k\in\mathbb{N}$. Then, subproblem \eqref{subprob} is precisely the problem \eqref{para-prob} associated with $u=u^k$, and its multiplier set at $\overline{x}^k$ is $\mathcal{M}(u^k,\overline{x}^k)$, where $\mathcal{M}$ is the mapping defined in \eqref{para-multiplier}. These facts will be often used in the subsequent analysis of this subsection. 
\begin{proposition}\label{prop-lambark}
 Under Assumptions \ref{ass2}-\ref{ass4}, there exists $\widehat{k}\in\mathbb{N}$ such that for each $k\ge\widehat{k}$, there is a vector $\overline{\lambda}^k\in\mathcal{M}(u^k,\overline{x}^k)$, and furthermore, the sequence $\{\overline{\lambda}^k\}_{k\ge\widehat{k}}$ is bounded.
\end{proposition}
\begin{proof}
 By Assumption \ref{ass4} and Definition \ref{def-BMP}, for every $(u^{*},\overline{x}^{*})\in \Omega^*$, there exist $\varepsilon_{u^{*},\overline{x}^{*}}>0$ and $\kappa_{u^{*},\overline{x}^{*}}>0$ such that for all $(u,x)\in\mathbb{B}((u^*,\overline{x}^*),\varepsilon_{u^{*},\overline{x}^{*}})\cap[\mathbb{U}\times\mathcal{S}(u)]$ and $y\in\mathcal{N}_{\mathcal{S}(u)}(x)$,  
 \begin{equation}\label{BMP}  \Lambda(u,x,y)\cap\kappa_{u^{*},\overline{x}^{*}}\|y\|\mathbb{B}\ne\emptyset,
 \end{equation}
 where $\Lambda$ is the mapping defined in Definition \ref{def-BMP}.
 Apparently, $\bigcup_{(u^*,\overline{x}^*)\in \Omega^*}\mathbb{B}^{\circ}((u^*,\overline{x}^*),\varepsilon_{u^{*},\overline{x}^{*}})$ is an open covering of the compact set $\Omega^*$. By Heine-Borel covering theorem, there exist $(u^{1,*},\overline{x}^{1,*}),\ldots,(u^{p,*},\overline{x}^{p,*})\in \Omega^*$ for some $p\in\mathbb{N}$ and $\varepsilon_{u^{1,*},\overline{x}^{1,*}}>0,\ldots,\varepsilon_{u^{p,*},\overline{x}^{p,*}}>0$ such that $ \Omega^*\subset\bigcup_{i=1}^p\mathbb{B}^{\circ}((u^{i,*},\overline{x}^{i,*}),\varepsilon_{u^{i,*},\overline{x}^{i,*}})$. Note that $\lim_{k\to\infty}{\rm dist}((u^k,\overline{x}^{k}),\Omega^*)=0$. There exists an index $\widehat{k}\in\mathbb{N}$ such that for all $k\ge\widehat{k}$, 
 \(
  (u^k,\overline{x}^k)\in\bigcup_{i=1}^p\mathbb{B}((u^{i,*},\overline{x}^{i,*}),\varepsilon_{u^{i,*},\overline{x}^{i,*}}).
 \)
 Write $\kappa\!:=\max_{i\in[p]}\kappa_{u^{i,*},\overline{x}^{i,*}}$ and $\varepsilon\!:=\min_{i\in[p]}\varepsilon_{u^{i,*},\overline{x}^{i,*}}$. For each $k\ge\widehat{k}$ (if necessary by increasing $\widehat{k}$), we have $(u^k,\overline{x}^k)\in\mathbb{B}((u^{i,*},\overline{x}^{i,*}),\varepsilon)\cap{\rm gph}\,\mathcal{S}$ for some $i\in[p]$. For each $k\in\mathbb{N}$, from the optimality of $\overline{x}^k$ to the problem \eqref{para-prob} associated with $u=u^k$ and the expression of $\theta$ in \eqref{theta-def1}, we have $0\in\nabla\vartheta_k(\overline{x}^k)+\partial\phi(\overline{x}^k)+\mathcal{N}_{\mathcal{S}(u^k)}(\overline{x}^k)\subset\partial F_k(\overline{x}^k)+\mathcal{N}_{\mathcal{S}(u^k)}(\overline{x}^k)$, so there exists $-y^k\in\partial F_k(\overline{x}^k)$ such that $y^k\in\mathcal{N}_{\mathcal{S}(u^k)}(\overline{x}^k)$. From the above \eqref{BMP}, for each $k\ge\widehat{k}$, there exists $\overline{\lambda}^k\in\Lambda(u^k,\overline{x}^k,y^k)$ such that $\|\overline{\lambda}^k\|\le\kappa\|y^k\|$. Moreover, for each $k\ge\widehat{k}$, since $\Lambda(u^k,\overline{x}^k,y^k)\subset\mathcal{M}(u^k,\overline{x}^k)$, it holds $\overline{\lambda}^k\in\mathcal{M}(u^k,\overline{x}^k)$. Note that the sequence $\{y^k\}_{k\ge\widehat{k}}$ is bounded in view of Remark \ref{bound-subdiff} and $-y^k\in\partial F_k(\overline{x}^k)$ for each $k$. The conclusion then follows. 
 \end{proof}

 We notice that Proposition \ref{prop-lambark} identifies a bounded multiplier sequence from the set sequence $\{\mathcal{M}(u^k,\overline{x}^k)\}_{k\in\mathbb{N}}$, though the set sequence itself is not uniformly bounded. This is significantly different from the existing works on MBA-like methods. They capture such a bounded multiplier sequence under the MFCQ for the system $g(x)\in\mathbb{R}_{-}^m$ at all feasible points, which ensures the uniform boundedness of the set sequence $\{\mathcal{M}(u^k,\overline{x}^k)\}_{k\in\mathbb{N}}$ and is stronger than Assumption \ref{ass4} by the discussion in Remark \ref{remark-ass4}. 

 Now we are ready to apply Proposition \ref{prop-lambark} to achieve the subsequential convergence of $\{x^k\}_{k\in\mathbb{N}}$, i.e., every accumulation point being a strong stationary point of \eqref{prob}.
\begin{theorem}\label{theorem-sconverge}
 Under Assumptions \ref{ass2}-\ref{ass4}, $\omega(x^0)\subset X_{s}^*$ and $F(x)=\varpi^*$ for all $x\in\omega(x^0)$. 
\end{theorem}
\begin{proof}
 Pick any $x^*\!\in\omega(x^0)$. There exists an index set $\mathcal{K}\subset\mathbb{N}$ such that $x^*=\lim_{\mathcal{K}\ni k\to\infty}x^k$, and  $\lim_{\mathcal{K}\ni k\to\infty}\overline{x}^k=x^*$ follows from Lemma \ref{lemma-xbark}. By Lemma \ref{lemma-bound}, if necessary by taking a subsequence, we have $\lim_{\mathcal{K}\ni k\to\infty}u^k=u^*$ for some $u^*=(L^*,\mathcal{Q}^*,\xi^*,x^*)\in\mathbb{U}$. From Proposition \ref{prop-lambark}, for each $k\ge\widehat{k}$, there exists $  \overline{\lambda}^k\in\mathcal{M}(u^k,\overline{x}^k)$ and $\{\overline{\lambda}^k\}_{k\ge\widehat{k}}$ is bounded. If necessary by taking a subsequence, we assume $\lim_{\mathcal{K}\ni k\to\infty}\overline{\lambda}^k=\overline{\lambda}^*$. By the definition of $\mathcal{M}$ in \eqref{para-multiplier}, for each $k\ge\widehat{k}$, 
 \[
  0\in\nabla_{\!x}\theta(u^k,\overline{x}^k)+\partial\phi(\overline{x}^k)+\nabla_{\!x}H(u^k,\overline{x}^k)\overline{\lambda}^k\ \ {\rm with}\ \ \overline{\lambda}^k\in\mathcal{N}_{\mathbb{R}_{-}^m}(H(u^k,\overline{x}^k)).
 \]
 Note that $\nabla_{\!x}\theta(\cdot,\cdot)$ and $\nabla_{\!x}H(\cdot,\cdot)$ are continuous by \eqref{theta-def1}-\eqref{h-def1}. Passing the limit $\mathcal{K}\ni k\to\infty$ to the above inclusions and using the osc property of $\partial\phi$ and  $\mathcal{N}_{\mathbb{R}_{-}^m}$ leads to 
 \[
 0\in\nabla_{\!x}\theta(u^*,x^*)+\partial\phi(x^*)+\nabla_{\!x}H(u^*,x^*)\overline{\lambda}^*\ \ {\rm with}\ \ \overline{\lambda}^*\in\mathcal{N}_{\mathbb{R}_{-}^m}(H(u^*,x^*)). 
 \]
 In view of \eqref{theta-def1}-\eqref{h-def1}, we have $\nabla_{\!x}\theta(u^*,x^*)=\nabla\!f(x^*)+\xi^*,H(u^*,x^*)=G(x^*,x^*,L^*)=g(x^*)$ and $\nabla_{\!x}H(u^*,x^*)=\nabla_{x}G(x^*,x^*,L^*)=\nabla g(x^*)$. Then, the above two inclusions reduce to
 \[
   0\in\nabla\!f(x^*)+\xi^*+\partial\phi(x^*)+\nabla g(x^*)\overline{\lambda}^*\ \ {\rm with}\ \ \overline{\lambda}^*\in\mathcal{N}_{\mathbb{R}_{-}^m}(g(x^*)).
 \]
 Note that $\xi^*\in\partial(-\psi)(x^*)$ because $\xi^k\in\partial(-\psi)(x^k)$ for each $k\in\mathbb{N}$ and $\partial(-\psi)$ is osc. The above equation shows that $x^*$ is a strong stationary point of \eqref{prob}. The inclusion $\omega(x^0)\subset X_{s}^*$ follows by the arbitrariness of $x^*\in\omega(x^0)$. Note that $F$ is continuous relative to its domain $\Gamma$ and $\{x^k\}_{k\in\mathbb{N}}\subset\Gamma$. Therefore, $F(x^*)=\lim_{\mathcal{K}\ni k\to\infty}F(x^k)= \varpi^*$. 
\end{proof} 

It is worth noting that the multiplier sequence $\{\lambda^k\}_{k\in\mathbb{N}}$ generated by Algorithm \ref{iMBA} does not take a part in the proof of the subsequential convergence. Instead, it joins in the application of Proposition \ref{prop-ebound} to establish the following result, which bounds the distance of the inexact solution $x^{k+1}$ of \eqref{subprob} from its unique solution $\overline{x}^k$ by the difference $\|x^{k+1}\!-x^k\|$. This result, as will be shown in Section \ref{sec4.2}, plays a crucial role in achieving the full convergence of the sequence $\{x^k\}_{k\in\mathbb{N}}$. 
\begin{proposition}\label{prop-eboundk}
 Under Assumptions \ref{ass2}-\ref{ass4}, there exists $\widehat{\gamma}>0$ such that for all $k\ge\widehat{k}$,
 \[
 \|x^{k+1}-\overline{x}^{k}\|\le\widehat{\gamma}\sqrt{R_k(x^{k+1},v^{k+1},\lambda^{k+1})}\le\widehat{\gamma}\sqrt{\beta_{R}/2}\|x^{k+1}-x^k\|.
 \]
\end{proposition}
\begin{proof}
 Fix any $k\ge\widehat{k}$. From \eqref{theta-def1}-\eqref{h-def1}, every component of the mapping $H(u^k,\cdot)$ is a convex function with $H(u^k,x^k)=G(x^k,x^k,L^k)\le 0$ and $\theta(u^k,\cdot)=\vartheta_k(\cdot)$ is strongly convex with modulus $\mu_{\min}>0$. Invoking Proposition \ref{prop-ebound} with $(\overline{u},\overline{x})=(u^k,\overline{x}^k)$ for $\overline{\lambda}=\overline{\lambda}^k$ with $\overline{\lambda}^k$ from Proposition \ref{prop-lambark} and $(x,v,\lambda)=(x^{k+1},v^{k+1},\lambda^{k+1})\in\mathbb{R}^n\times\partial\phi(x^{k+1})\times\mathbb{R}_+^m$ yields that
 \begin{align}\label{keybound0}
  \|x^{k+1}-\overline{x}^{k}\|^2&\le\mu_{\min}^{-1}\big[\langle x^{k+1},\nabla\vartheta_k(x^{k+1})+v^{k+1}+\nabla_{\!x} G(x^{k+1},x^k,L^k)\lambda^{k+1}\rangle\nonumber\\		
  &\quad\ +\|\overline{x}^k\|\|\nabla\vartheta_k(x^{k+1})+v^{k+1}+\nabla_{\!x} G(x^{k+1},x^k,L^k)\lambda\|\nonumber\\
  &\quad\ +\|\overline{\lambda}^k\|_1\|[G(x^{k+1},x^k,L^k)]_{+}\|_{\infty}-\langle\lambda^{k+1},G(x^{k+1},x^k,L^k)\rangle\big].
 \end{align} 
 Since the sequence $\{\overline{x}^k\}_{k\in\mathbb{N}}$ is bounded by Lemma \ref{lemma-bound} (i), there exists a constant $\widehat{\gamma}_0>0$ such that $\|\overline{x}^k\|\le\widehat{\gamma}_0$. While by Proposition \ref{prop-lambark}, there exists a constant $\widehat{\gamma}_1>0$ such that $\|\overline{\lambda}^k\|_1\le\widehat{\gamma}_1$ for all $k\ge\widehat{k}$. Let $\widehat{\gamma}:=\mu_{\min}^{-1/2}\max\{1,\widehat{\gamma}_0,\widehat{\gamma}_1\}^{1/2}$. The first inequality follows \eqref{keybound0} and the expression of $R_k$, and the second one is due to \eqref{inexact1}. The proof is completed.
\end{proof}

\subsection{Full convergence of iterate sequence}\label{sec4.2}

As well known, potential functions play a key part in the full convergence analysis of algorithms for nonconvex and nonsmooth optimization (see \cite{Latafat22,Ouyang24,Aragon25}). To construct the desirable potential function, for any $z=(x,s,L,\xi)\in\mathbb{Z}:=\mathbb{R}^n\times\mathbb{R}^n\times\mathbb{R}_{++}^m\times\mathbb{R}^n$, let 
\begin{equation}\label{Phi-def}
\Phi(z):=f(x)+\phi(x)+\delta_{G^{-1}(\mathbb{R}_{-}^m)}(x,s,L)+\langle x,\xi\rangle+\psi^*(-\xi).
\end{equation}
For each $k\in\mathbb{N}$, write $z^k\!:=(\overline{x}^{k},x^k,L^k,\xi^k)\in\mathbb{Z}$. We first disclose the relation between the function value $\Phi(z^k)$ and the objective value $F(x^{k+1})$ in the following lemma.  
\begin{lemma}\label{lemma-potential} 
 Under Assumptions \ref{ass2}-\ref{ass4}, there exist an index $\mathbb{N}\ni\overline{k}\ge\widehat{k}$ and a compact convex set $D\subset\mathcal{O}$ such that $\{(x^k,\overline{x}^k)\}_{k\ge\overline{k}}\subset D\times D$, and for each $k\ge\overline{k}$, 
 \begin{equation*}
  F(x^{k+1})-\widetilde{c}\,\|x^{k+1}-x^k\|^2\le f(\overline{x}^k)+\phi(\overline{x}^k)+\langle\overline{x}^{k},\xi^k\rangle+\psi^*(-\xi^k)=\Phi(z^k)
 \end{equation*}
 with $\widetilde{c}=(L_{\!f}+\beta_{\mathcal{Q}}\!+{\beta_{F}}/{2})+(L_{\!f}+\beta_{\mathcal{Q}}/2)\widehat{\gamma}^2\beta_{R}$, where $L_{\!f}$ is the Lipschitz constant of $\nabla\!f$ on $D$.
\end{lemma}
\begin{proof}
Recall that $\{x^k\}_{k\in\mathbb{N}}\subset\Gamma$ and $\lim_{k\to\infty}\|\overline{x}^k-x^k\|=0$, so there exists $\overline{k}\ge\widehat{k}$ such that $\overline{x}^k\in x^k\!+(1/\overline{k})\mathbb{B}\subset\mathcal{O}$ for all $k\ge\overline{k}$. Together with  Assumption \ref{ass3}, there exists a compact convex set $D\subset\mathcal{O}$ such that $\{(x^k,\overline{x}^k)\}_{k\ge\overline{k}}\subset D\times D$. Note that $\nabla\!f$ is Lipschitz continuous on $D$ because it is strictly continuous on $\mathcal{O}$ by Assumption \ref{ass1} (ii). Then,  
\begin{equation}\label{flip-D}
 \|\nabla\!f(x)-\nabla\!f(y)\|\le L_{\!f}\|x-y\|\quad\forall x,y\in D.
\end{equation} 
For each $k\in\mathbb{N}$, from the second inequality of \eqref{inexact2} with $j=j_k$ and Lemma \ref{lemma-bound} (iii), it holds 
\begin{equation}\label{temp-ineq41}
\langle\nabla\!f(x^k)+\xi^k,x^{k+1}\!-\overline{x}^{k}\rangle\le\frac{\beta_{\mathcal{Q}}}{2}\|\overline{x}^{k}-x^k\|^2+\frac{\beta_{F}}{2}\|x^{k+1}\!-x^k\|^2+\phi(\overline{x}^k)-\phi(x^{k+1}).
\end{equation}  
 Now invoking the descent lemma and the expression of $F$, for each $k\ge\overline{k}$, we have
 \begin{align}\label{temp-fineq0}
 F(x^{k+1})&\le f(\overline{x}^{k})+\langle\nabla\!f(\overline{x}^{k}),x^{k+1}-\overline{x}^{k}\rangle+\frac{L_{\!f}}{2}\|x^{k+1}-\overline{x}^{k}\|^2+\phi(x^{k+1})-\psi(x^{k+1})\nonumber\\
 &\stackrel{\eqref{temp-ineq41}}{\le}f(\overline{x}^{k})+\phi(\overline{x}^{k})+\langle\nabla\!f(\overline{x}^{k})-\nabla\!f(x^k),x^{k+1}-\overline{x}^{k}\rangle-\langle\xi^k,x^{k+1}-\overline{x}^k\rangle\nonumber\\
 &\quad+\frac{\beta_{\mathcal{Q}}}{2}\|\overline{x}^{k}-x^k\|^2+\frac{\beta_{F}}{2}\|x^{k+1}\!-x^k\|^2+\frac{L_{\!f}}{2}\|x^{k+1}-\overline{x}^{k}\|^2-\psi(x^{k+1})\nonumber\\
 &\stackrel{\eqref{flip-D}}{\le} f(\overline{x}^{k})+\phi(\overline{x}^{k})+L_{\!f}\|x^k-\overline{x}^{k}\|\|x^{k+1}-\overline{x}^{k}\|+\frac{\beta_{\mathcal{Q}}}{2}\|\overline{x}^{k}-x^k\|^2\nonumber\\
 &\quad\ +\frac{L_{\!f}}{2}\|x^{k+1}-\overline{x}^{k}\|^2+\frac{\beta_{F}}{2}\|x^{k+1}\!-x^k\|^2-\psi(x^{k+1})-\langle\xi^k,x^{k+1}-\overline{x}^k\rangle\nonumber\\
 &\le f(\overline{x}^{k})+\phi(\overline{x}^{k})+(L_{\!f}+\beta_{\mathcal{Q}}+{\beta_{F}}/{2})\|x^{k+1}\!-x^k\|^2\nonumber\\
 &\quad\ +(2L_{\!f}+\beta_{\mathcal{Q}})\|x^{k+1}-\overline{x}^{k}\|^2+\langle\xi^k,\overline{x}^k\rangle-[\psi(x^{k+1})+\langle\xi^k,x^{k+1}\rangle].
 \end{align}
 From the previous \eqref{psi-cvx} for $x=x^{k+1}$ and the definition of conjugate functions, it follows 
 \[
  \psi(x^{k+1})+\langle\xi^k,x^{k+1}\rangle\ge\psi(x^k)+\langle\xi^k,x^k\rangle\ge -\psi^*(-\xi^k)\quad\forall k\in\mathbb{N}.
 \]
 Together with the above \eqref{temp-fineq0} and Proposition \ref{prop-eboundk}, we obtain the desired inequality. The equality follows the expression of $\Phi(z^k)$ and $G(\overline{x}^k,x^k,L^k)\le 0$. The proof is finished.
\end{proof}

Motivated by Lemma \ref{lemma-potential}, we define the potential function $\Phi_{\widetilde{c}}:\mathbb{Z}\times\mathbb{R}^n\to\overline{\mathbb{R}}$ by 
\begin{equation}\label{Phic-def}
\Phi_{\widetilde{c}}(w):=\Phi(z)+\widetilde{c}\|d\|^2\quad\forall w=(z,d)\in\mathbb{Z}\times\mathbb{R}^n,
\end{equation}
where $\widetilde{c}$ is the same as in Lemma \ref{lemma-potential}. For each $k\in\mathbb{N}$, let $w^k\!:=(z^k,x^{k+1}\!-\!x^k)$. We next prove that the function $\Phi_{\widetilde{c}}$ keeps unchanged on the set of cluster points of $\{w^k\}_{k\in\mathbb{N}}$.  
\begin{proposition}\label{prop-ZWstar}
 Denote by $Z^*$ the set of cluster points of $\{z^k\}_{k\in\mathbb{N}}$, and by $W^*$ the set of cluster points of $\{w^k\}_{k\in\mathbb{N}}$. Then, under Assumptions \ref{ass2}-\ref{ass4}, the following assertions hold.
 \begin{description}
 \item[(i)] The set $W^*$ is nonempty and compact with $W^*=Z^*\times\{0\}$. 
		
 \item[(ii)] For every $z^*=(\overline{x}^*,x^*,L^*,\xi^*)\in Z^*$, $\overline{x}^*=x^*\in X_{s}^*$ and $\xi^*\in\partial (-\psi)(x^*)$.
		
 \item[(iii)] $\Phi(z)=\varpi^*=\Phi_{\widetilde{c}}(w)$ for all $z\in Z^*$ and $w\in W^*$
 
 \item[(iv)]  $\lim_{k\to\infty}\Phi(z^k)=\varpi^*=\lim_{k\to\infty}\Phi_{\widetilde{c}}(w^k)$.
\end{description}
\end{proposition}
\begin{proof}
 {\bf(i)-(ii)} By Lemma \ref{lemma-bound}, the sequence $\{z^k\}_{k\in\mathbb{N}}$ is bounded, so the set $Z^*$ is nonempty and compact. Recall that $w^k=(z^k,x^{k+1}-x^k)$ for all $k\in\mathbb{N}$. Item (i) then follows Corollary \ref{corollary-objF}. Pick any $z^*=(\overline{x}^*,x^*,L^*,\xi^*)\in Z^*$. There exists an index set $\mathcal{K}\subset\mathbb{N}$ such that $z^*=\lim_{\mathcal{K}\ni k\to\infty}z^k$. Then, $\overline{x}^*=x^*\in X_{s}^*$ by Lemma \ref{lemma-xbark} and Theorem \ref{theorem-sconverge}. Recall that $\xi^k\in\partial(-\psi)(x^k)$ for each $k\in\mathbb{N}$ and the mapping $\partial(-\psi)\!:\mathbb{R}^n\rightrightarrows\mathbb{R}^n$ is osc. Hence, $\xi^*\in\partial (-\psi)(x^*)$.

 \noindent
 {\bf(iii)} Pick any $z=(\overline{x},x,L,\xi)\in Z^*$. From item (ii), $\overline{x}=x\in X_{s}^*$ and $\xi\in\partial (-\psi)(x)\subset-\partial\psi(x)$. The latter implies $\psi(x)+\psi^*(-\xi)=-\langle x,\xi\rangle=-\langle\overline{x},\xi\rangle$ by \cite[Theorem 23.5]{Roc70}. Consequently, 
 \[
  \Phi(z)=f(\overline{x})+\phi(\overline{x})+\langle \overline{x},\xi\rangle+\psi^*(-\xi)=f(x)+\phi(x)-\psi(x)=F(x)=\varpi^*, 
 \]
 where the fourth equality is due to Theorem \ref{theorem-sconverge} and $x\in\omega(x^0)$. For any $w=(z,d)\in W^*$, from item (i) and the definition of $\Phi_{\widetilde{c}}$, we have $z\in Z^*$ and $\Phi_{\widetilde{c}}(w)=\Phi(z)$. 
	
 \noindent
 {\bf(iv)} From the expression of $\Phi_{\widetilde{c}}$ in \eqref{Phic-def} and Lemma \ref{lemma-potential}, for each $k\ge\overline{k}$, it holds that
 \begin{align}\label{relationPhi}
 \Phi_{\widetilde{c}}(w^k)-\varpi^*
 &=\Phi(z^k)+\widetilde{c}\|x^{k+1}-x^k\|^2-\varpi^*\nonumber\\
 &=f(\overline{x}^k)+\phi(\overline{x}^k)+\langle \overline{x}^k,\xi^k\rangle+\psi^*(-\xi^k)+\widetilde{c}\|x^{k+1}-x^k\|^2-\varpi^*\nonumber\\
 &\ge F(x^{k+1})-\varpi^*\ge0.
 \end{align}  
 Suppose on the contrary that $\lim_{k\to\infty}\Phi_{\widetilde{c}}(w^k)\ne\varpi^*$. From \eqref{relationPhi} and the boundedness of $\{w^k\}_{k\in\mathbb{N}}$, there exists an index set $\mathcal{K}\subset\mathbb{N}$ such that 
 $\limsup_{\mathcal{K}\ni k\to\infty}\Phi_{\widetilde{c}}(w^k)>\varpi^*$ and $\lim_{\mathcal{K}\ni k\to\infty}w^k=w^*\in W^*$. Along with $\psi(x^k)+\psi^*(-\xi^k)=-\langle x^k,\xi^k\rangle$ for each $k\in\mathbb{N}$, it follows $\lim_{\mathcal{K}\ni k\to\infty }\psi^*(-\xi^k)=-\psi(x^*)-\langle x^*,\xi^*\rangle=\psi^*(-\xi^*)$. Note that $\delta_{G^{-1}(\mathbb{R}_{-}^m)}(\cdot,\cdot,\cdot)$ is continuous relative to its domain and $\{w^k\}_{k\in\mathbb{N}}\subset{\rm dom}\,\Phi$. Then, from the expression of $\Phi_{\widetilde{c}}$,  
 \[
  \Phi_{\widetilde{c}}(w^*)=\lim_{\mathcal{K}\ni k\to\infty}\Phi_{\widetilde{c}}(w^k)=\limsup_{\mathcal{K}\ni k\to\infty}\Phi_{\widetilde{c}}(w^k)>\varpi^*,
 \]
 which is a contradiction to $\Phi_{\widetilde{c}}(w^*)=\varpi^*$ by item (iii). Thus, $\lim_{k\to\infty}\Phi_{\widetilde{c}}(w^k)=\varpi^*$. Together with \eqref{relationPhi} and Corollary \ref{corollary-objF}, we have $\lim_{k\to\infty}\Phi(z^k)=\varpi^*$. The proof is finished. 
\end{proof}

 To achieve the full convergence of the sequence $\{x^k\}_{k\in\mathbb{N}}$, we also need to measure the approximate stationarity of its augmented sequence $\{w^k\}_{k\in\mathbb{N}}$ by $\|x^{k+1}\!-x^k\|^2$.  
\begin{proposition}\label{prop-relgap}
 Suppose that Assumptions \ref{ass2}-\ref{ass4} hold, and that the mapping $g$ is also twice differentiable. Then, there exists $\gamma>0$ such that ${\rm dist}(0,\widehat{\partial}\Phi_{\widetilde{c}}(w^k))\le\gamma\|x^{k+1}-x^k\|$ for all $k\ge\overline{k}$, where $\overline{k}\in\mathbb{N}$ is the same as in Lemma \ref{lemma-potential}.
\end{proposition}
\begin{proof}
 First of all, we claim that at any $z=((x,s,L),\xi)\in G^{-1}(\mathbb{R}_{-}^m)\times\mathbb{R}^n$, it holds 
 \begin{equation}\label{aim-inclusion}
 \widehat{\partial}\Phi(z)\supset\!\begin{pmatrix}
  [\nabla\!f(x)+\xi\!+\!\partial\phi(x)]\times
  \{0_{\mathbb{R}^n}\}\times\{0_{\mathbb{R}^m}\}
 +\nabla G(x,s,L)\mathcal{N}_{\mathbb{R}_{-}^m}(G(x,s,L))\\
  x-\partial\psi^*(-\xi)
 \end{pmatrix}.
 \end{equation}
 Fix any $z=((x,s,L),\xi)\in G^{-1}(\mathbb{R}_{-}^m)\times\mathbb{R}^n$. Let $\Phi_1(z')\!:=\delta_{G^{-1}(\mathbb{R}_{-}^m)}(x',s',L')$ and $\Phi_2(z')\!:=f(x')+\!\phi(x')+\!\langle x',\xi'\rangle+\psi^*(-\xi')$ for $z'\!=(x',s',L',\xi')\in\!\mathbb{Z}$ with $L'\in\mathbb{R}_{++}^m$. Apparently, $\Phi=\Phi_1+\Phi_2$. Then, it holds  $\widehat{\partial}\Phi(z)\supset\widehat{\partial}\Phi_1(z)+\widehat{\partial}\Phi_2(z)$. By the definition of tangent cones (see \cite[Definition 6.1]{RW98}), it is not hard to obtain
 \[
 \mathcal{T}_{G^{-1}(\mathbb{R}_{-}^m)}(x,s,L)\subset[G'(x,s,L)]^{-1}\mathcal{T}_{\mathbb{R}_{-}^m}(G(x,s,L)).
 \]
 Recall that $\widehat{\mathcal{N}}_{G^{-1}(\mathbb{R}_{-}^m)}(x,s,L)=[\mathcal{T}_{G^{-1}(\mathbb{R}_{-}^m)}(x,s,L)]^{\circ}$. The following inclusion holds:
 \[
  \widehat{\partial}\Phi_1(z)=\widehat{\mathcal{N}}_{G^{-1}(\mathbb{R}_{-}^m)}(x,s,L)\times\{0_{\mathbb{R}^n}\}\supset \nabla G(x,s,L)\mathcal{N}_{\mathbb{R}_{-}^m}(G(x,s,L))\times\{0_{\mathbb{R}^n}\}.
\]
 In addition, from Assumption \ref{ass1} and \cite[Corollary 10.9]{RW98}, it is immediate to obtain 
 \[
 \widehat{\partial}\Phi_2(z)=\partial\Phi_2(z)=[\nabla\!f(x)+\xi+\partial\phi(x)]\times\{0_{\mathbb{R}^n}\}\times\{0_{\mathbb{R}^m}\}\times[x-\partial\psi^*(-\xi)].
 \]
 The claimed \eqref{aim-inclusion} follows the above two equations and the inclusion $\widehat{\partial}\Phi(z)\supset\widehat{\partial}\Phi_1(z)+\widehat{\partial}\Phi_2(z)$.

 Fix any $k\ge\overline{k}$. By the definition of $G$ in equation \eqref{def-Gmap}, for any $\lambda\in\mathbb{R}^m$, we calculate that
 \begin{equation}\label{temp1-gradG}
 \nabla G(\overline{x}^k,x^k,L^k)\lambda
 =\begin{pmatrix}
  \nabla g(x^k)\lambda+\langle L^k,\lambda\rangle(\overline{x}^k\!-\!x^k)\\
  \sum_{i=1}^m[\lambda_i\nabla^2g_i(x^k)(\overline{x}^k\!-\!x^k)]+\langle L^k,\lambda\rangle(x^k\!-\!\overline{x}^k)\\
  \frac{1}{2}\|\overline{x}^k-x^k\|^2\lambda
  \end{pmatrix}.
 \end{equation} 
 Let $\overline{\lambda}^k\in\mathcal{M}(u^k,\overline{x}^k)$ for $u^k\!=(L^k,\mathcal{Q}_k,\xi^k,x^k)\!\in\mathbb{U}$ be the same as in Proposition \ref{prop-lambark}. By the expression of $\mathcal{M}$ in \eqref{para-multiplier} with $\theta$ and $H$ from \eqref{theta-def1}-\eqref{h-def1}, there exists $\overline{v}^k\in\partial\phi(\overline{x}^k)$ such that 
\begin{subnumcases}{}\label{subkkt1}
\nabla\!f(x^k)+\mathcal{Q}_k(\overline{x}^{k}\!-\!x^k)+\xi^k+\overline{v}^k+\nabla g(x^k)\overline{\lambda}^{k}+\langle L^k,\overline{\lambda}^k\rangle(\overline{x}^{k}\!-\!x^k)=0,\\
 \label{subkkt2}
 \overline{\lambda}^k\in\mathcal{N}_{\mathbb{R}_{-}^m}(G(\overline{x}^k,x^k,L^k)).
\end{subnumcases}
By \eqref{subkkt1}, we have
$\nabla g(x^k)\overline{\lambda}^{k}+\langle L^k,\overline{\lambda}^{k}\rangle(\overline{x}^{k}\!-\!x^k)=-\nabla\!f(x^k)-\mathcal{Q}_k(\overline{x}^{k}\!-\!x^k)-\xi^k-\overline{v}^k$. Along with the expression of $\nabla G(\overline{x}^k,x^k,L^k)\overline{\lambda}^k$ in equation \eqref{temp1-gradG}, it follows that
 \begin{align*}
 \nabla G(\overline{x}^k,x^k,L^k)\overline{\lambda}^k=\begin{pmatrix}
  -\nabla\!f(x^k)-\mathcal{Q}_k(\overline{x}^{k}\!-\!x^k)-\xi^k-\overline{v}^k\\
  \sum_{i=1}^{m}\overline{\lambda}^{k}_i\nabla^2g_i(x^k)(\overline{x}^{k}\!-\!x^k)+\langle L^k,\overline{\lambda}^k\rangle(x^k\!-\!\overline{x}^{k})\\
  \frac{1}{2}\|x^k-\overline{x}^{k}\|^2\overline{\lambda}^{k}
 \end{pmatrix}.
 \end{align*} 
 Comparing with the inclusion \eqref{aim-inclusion} for $z=z^k$ and using $x^k\in\partial\psi^*(-\xi^k)$ and \eqref{subkkt2}, we have  
 \begin{equation}\label{dist-zetak}
 \zeta^k:=\!\begin{pmatrix}
  \nabla\!f(\overline{x}^{k})-\nabla\!f(x^k)-\mathcal{Q}_k(\overline{x}^{k}\!-\!x^k)\\
  \sum_{i=1}^{m}\overline{\lambda}^{k}_i\nabla^2g_i(x^k)(\overline{x}^{k}\!-\!x^k)+\langle L^k,\overline{\lambda}^k\rangle(x^k\!-\!\overline{x}^{k})\\
  \frac{1}{2}\|x^k-\overline{x}^{k}\|^2\overline{\lambda}^{k}\\
  \overline{x}^{k}-x^k\end{pmatrix}
  \in\widehat{\partial}\Phi(z^k).
 \end{equation}
 Notice that $\widehat{\partial}\Phi_{\widetilde{c}}(w^k)=\widehat{\partial}\Phi(z^k)\times\{2\widetilde{c}(x^{k+1}\!-\!x^k)\}$ for each $k\ge\overline{k}$. 
 It suffices to claim that there exists $\gamma>0$ such that $\|\zeta^k\|\le \gamma\|x^{k+1}-x^k\|$ for all $k\ge\overline{k}$. Indeed, for each $k\ge\overline{k}$, 
 \begin{align*}
  \|\nabla\!f(\overline{x}^{k})-\nabla\!f(x^{k})-\mathcal{Q}_k(\overline{x}^k-x^k)\|
  &\le\|\nabla\!f(\overline{x}^{k})-\nabla\!f(x^k)\|+\beta_{\mathcal{Q}}\|\overline{x}^{k}-x^k\|\\
  &\stackrel{\eqref{flip-D}}{\le} (L_{\!f}\!+\!\beta_{\mathcal{Q}})\|\overline{x}^{k}-x^k\|.
 \end{align*}
 Next we claim that for each $i\in[m]$ the sequence $\{\nabla^2\!g_i(x^k)\}_{k\ge\overline{k}}$ is bounded. Fix any $i\in[m]$. Since the mapping $\nabla\!g_i$ is differentiable, for each $k\in\mathbb{N}$, $\nabla^2g_i(x^k)\in\partial_{C}\nabla\!g_i(x^k)$, where $\partial_{C}\nabla\!g_i(x^k)$ denotes the Clarke Jacobian of $\nabla\!g_i$ at $x^k$. Since the mapping  $\partial_{C}\nabla\!g_i\!:\mathbb{R}^n\rightrightarrows\mathbb{R}^n$ is locally bounded, and the sequence $\{x^k\}_{k\in\mathbb{N}}$ is bounded, from \cite[Proposition 5.15]{RW98}, there exists $\gamma_0>0$ such that $\|\nabla^2g_i(x^k)\|\le\gamma_0$ for all $k\in\mathbb{N}$. In addition, the sequence $\{\overline{\lambda}^k\}_{k\ge\overline{k}}$ is bounded due to Proposition \ref{prop-lambark}. Thus, there exists $\gamma_1>0$ such that for all $k\ge\overline{k}$, 
 \begin{equation*} \big\|\textstyle{\sum_{i=1}^{m}}\overline{\lambda}^{k}_i\nabla^2g_i(x^k)(\overline{x}^{k}-x^k)+\langle L^k,\overline{\lambda}^k\rangle(x^k-\overline{x}^{k})\big\|\le \gamma_1\|\overline{x}^{k}-x^k\|.
\end{equation*}
Combining the above two inequalities with the expression of $\zeta^k$ in \eqref{dist-zetak}, we obtain
\[
  \|\zeta^k\|\le (L_{\!f}+\beta_{\mathcal{Q}}+\gamma_1)\|\overline{x}^{k}-x^k\|+\frac{1}{2}\|\overline{\lambda}^{k}\|\|\overline{x}^{k}-x^k\|^2+\|\overline{x}^{k}-x^k\|\quad\forall k\ge\overline{k}.
\]
Recall that $\lim_{k\to\infty}\|\overline{x}^{k}-x^k\|=0$ by Lemma \ref{lemma-xbark}, if necessary by increasing $\overline{k}$, $\|\overline{x}^k-x^k\|\le1$ for all $k\ge\overline{k}$. While by Proposition \ref{prop-eboundk}, for each $k\ge\overline{k}$, $\|\overline{x}^{k}-x^k\|\le\|\overline{x}^{k}-x^{k+1}\|+\|x^{k+1}-x^k\|\le(\widehat{\gamma}\sqrt{\beta_r/2}+1)\|x^{k+1}-x^k\|$.  Thus, from the above inequality, there exists $\gamma>0$ such that $\|\zeta^k\|\le\gamma\|x^{k+1}-x^k\|$ for all $k\ge\overline{k}$.  
\end{proof}

Proposition \ref{prop-relgap} provides a relative inexact optimality condition for minimizing $\Phi_{\widetilde{c}}$ with the sequence $\{w^k\}_{k\in\mathbb{N}}$, but the sufficient decrease of $\{\Phi_{\widetilde{c}}(w^k)\}_{k\in\mathbb{N}}$ is unavailable. Thus, we cannot apply directly the recipe developed in \cite{Attouch13,Bolte14} to get the convergence of $\{w^k\}_{k\in\mathbb{N}}$. Next we establish the full convergence of $\{x^k\}_{k\in\mathbb{N}}$ by combining the KL inequality on the function $\Phi_{\widetilde{c}}$ and the decrease of $\{F(x^k)\}_{k\in\mathbb{N}}$ skillfully.
\begin{theorem}\label{globalconv}
 Suppose that Assumptions \ref{ass2}-\ref{ass4} hold, that the mapping $g$ is also twice differentiable, and that $\Phi_{\widetilde{c}}$ has the KL property on $W^*$. Then, $\sum_{k=0}^{\infty}\|x^{k+1}\!-x^k\|<\infty$, and consequently the sequence $\{x^k\}_{k\in\mathbb{N}}$ converges to some $x^*\in\omega(x^0)\subset X_{s}^*$.
\end{theorem}
\begin{proof}
 If there exists some $k_0\in\mathbb{N}$ such that $F(x^{k_0})=F(x^{k_0+1})$, then $x^{k_0+1}=x^{k_0}$ follows Corollary \ref{corollary-objF}. By Remark \ref{remark-alg} (d), Algorithm \ref{iMBA} finds a strong stationary point within a finite number of steps. Hence, it suffices to consider that $F(x^k)>F(x^{k+1})>\varpi^*$ for all $k\in\mathbb{N}$. From \eqref{relationPhi} it follows $\Phi_{\widetilde{c}}(w^k)-\varpi^*>0$ for all $k\ge\overline{k}$, where $\overline{k}$ is the same as in Lemma \ref{lemma-potential}. By Proposition \ref{prop-ZWstar}, the set $W^*$ is nonempty and compact, and $\Phi_{\widetilde{c}}(w)=\varpi^*$ for all $w\in W^*$. From the KL property of $\Phi_{\widetilde{c}}$ on $W^*$ and \cite[Lemma 6]{Bolte14}, there exist $\varepsilon>0,\eta>0$ and $\varphi\in\Upsilon_{\!\eta}$ such that for all $w\in[\varpi^*<\Phi_{\widetilde{c}}<\varpi^*+\eta]\cap\mathfrak{B}(W^*,\varepsilon)$ with $\mathfrak{B}(W^*,\varepsilon)\!:=\big\{w\in\mathbb{W}\,|\,{\rm dist}(w,W^*)\le\varepsilon\big\}$,
 \[
  \varphi'(\Phi_{\widetilde{c}}(w)-\varpi^*){\rm dist}(0,\partial\Phi_{\widetilde{c}}(w))\ge 1.
 \]
 Recall that $\lim_{k\to\infty}\Phi_{\widetilde{c}}(w^k)=\varpi^*$ by Proposition \ref{prop-ZWstar} (iv) and $\lim_{k\to\infty}{\rm dist}(w^k,W^*)=0$. There exists $\mathbb{N}\ni\widehat{k}_1\ge\overline{k}$ such that for all $k\ge \widehat{k}_1$, $w^k\in[\varpi^*<\Phi_{\widetilde{c}}<\varpi^*\!+\!\eta]\cap\mathfrak{B}(W^*,\varepsilon)$, so 
 \[
  \varphi'\big(\Phi_{\widetilde{c}}(w^k)\!-\!\varpi^*\big){\rm dist}(0,\widehat{\partial}\Phi_{\widetilde{c}}(w^k))\ge\varphi'\big(\Phi_{\widetilde{c}}(w^k)\!-\!\varpi^*\big){\rm dist}(0,\partial\Phi_{\widetilde{c}}(w^k))\geq 1. 
 \]
 Together with Proposition \ref{prop-relgap} and $\varphi\in\Upsilon_{\eta}$, it follows that for all $k\ge\widehat{k}_1$, 
 \begin{equation*}
 \gamma\varphi'\big(\Phi_{\widetilde{c}}(w^{k-1})-\varpi^*\big)\|x^k-x^{k-1}\|\ge 1.
 \end{equation*}
 Since $\varphi'$ is nonincreasing on $(0,\eta)$, the above inequality along with \eqref{relationPhi} implies that
 \begin{equation}\label{equa1-later}
 \varphi'(F(x^k)-\varpi^*)\ge\varphi'(\Phi_{\widetilde{c}}(w^{k-1})-\varpi^*)
  \ge\frac{1}{\gamma\|x^k-x^{k-1}\|}\quad \forall k\ge\widehat{k}_1.
 \end{equation}
 Now using inequality \eqref{equa1-later} and the sufficient decrease of $\{F(x^k)\}_{k\in\mathbb{N}}$ in \eqref{F-decrease} and following the same arguments as those in \cite[Theorem 1 (i)]{Bolte14} leads to the conclusion. For completeness, we include the details. Invoking the concavity of $\varphi$ and inequality \eqref{F-decrease} leads to
 \begin{align*}
 \Delta_{k,k+1}&:=\varphi(F(x^k)-\varpi^*)-\varphi(F(x^{k+1})-\varpi^*)
  \ge\varphi'(F(x^k)-\varpi^*)(F(x^{k})\!-\!F(x^{k+1}))\\
  &\stackrel{\eqref{equa1-later}}{\ge}\frac{F(x^{k})\!-\!F(x^{k+1})}{\gamma\|x^k-x^{k-1}\|}
  \ge\frac{\alpha\|x^{k+1}-x^k\|^2}{2\gamma\|x^k-x^{k-1}\|}\quad\forall k\ge\widehat{k}_1.
 \end{align*}
 Then, $\|x^{k+1}\!-\!x^k\|\le\sqrt{{2\gamma}\alpha^{-1}\Delta_{k,k+1}\|x^k\!-\!x^{k-1}\|}$ for all $k\ge\widehat{k}_1$. Along with $2\sqrt{ab}\le a+b$ for $a\ge 0,b\ge 0$, we have 
 $2\|x^{k+1}\!-\!x^k\|\le 2\alpha^{-1}\gamma\Delta_{k,k+1}+\|x^k\!-\!x^{k-1}\|$ for all $k\ge\widehat{k}_1$. Summing this inequality from any $k\ge\widehat{k}_1$ to any $l>k$ yields that
 \begin{align*}
 2{\textstyle\sum_{i=k}^l}\|x^{i+1}-x^i\|
 &\le{\textstyle\sum_{i=k}^l}\|x^i-x^{i-1}\|+2\alpha^{-1}\gamma{\textstyle\sum_{i=k}^l}\Delta_{i,i+1}\\
 &\le{\textstyle\sum_{i=k}^l}\|x^{i+1}-x^i\|+\|x^k-x^{k-1}\|
     +2\alpha^{-1}\gamma\varphi\big(F(x^{k})-\varpi^*\big),
 \end{align*}
 where the second inequality is by the nonnegativity of $\varphi$. Thus, for any $k\ge\widehat{k}_1$, 
\begin{equation}\label{temp-ratek}
 {\textstyle\sum_{i=k}^l}\|x^{i+1}-x^i\|\le\|x^k-x^{k-1}\|+2\alpha^{-1}\gamma\varphi(F(x^{k})-\varpi^*).
\end{equation}
Passing the limit $l\to\infty$ to this inequality leads to $\sum_{k=0}^{\infty}\|x^{k+1}\!-x^k\|<\infty$. 
\end{proof}

\subsection{Convergence rate of iterate sequence}\label{sec4.3}

When $\Phi$ is a KL function of exponent $q\in[1/2,1)$, we can get the linear convergence rate of $\{x^k\}_{k\in\mathbb{N}}$ for $q=1/2$, and its sublinear one for $q\in(1/2,1)$ as follows. 
\begin{theorem}\label{KL-rate}
 Suppose that Assumptions \ref{ass2}-\ref{ass4} hold, that $g$ is also twice differentiable and $\psi^*$ is continuous relative to its domain, and that $\Phi$ has the KL property of exponent $q\in[1/2,1)$ on $Z^*$. Then, 
 \begin{description}
 \item [(i)] when $q=1/2$, there exist $\widehat{a}_1>0$ and $\rho,\widehat{\rho}\in(0,1)$ such that for all $k$ large enough, 
 \[
  F(x^{k+1})-\varpi^*\le \rho(F(x^k)-\varpi^*)\ \ {\rm and}\ \ \|x^k-x^*\|\le \widehat{a}_1\widehat{\rho}^{k}.
 \]
		
 \item[(ii)] when $q\in(1/2,1)$, there exist $\widehat{a}_2>0$ and $\widehat{a}_3>0$ such that for sufficiently large $k$, 
 \[
  F(x^k)-\varpi^*\le \widehat{a}_2 k^{\frac{1}{1-2q}}\ \ {\rm and}\ \ \|x^k-x^*\|\le \widehat{a}_3 k^{\frac{1-q}{1-2q}}.
 \]
\end{description}
\end{theorem}
\begin{proof}
 Note that ${\rm dom}\,\partial\Phi\subset{\rm dom}\,\Phi=\big\{(x,s,L,\xi)\in\mathbb{Z}\,|\,G(x,s,L)\in\mathbb{R}_{-}^m,-\xi\in{\rm dom}\,\psi^*\big\}$. Since $\psi^*$ is continuous relative to its domain, the function $\Phi$ is continuous relative to its domain ${\rm dom}\,\Phi$. By \cite[Theorem 3.3]{LiPong18}, if $\Phi$ has the KL property of exponent $q\in[1/2,1)$ at a point $\overline{z}\in Z^*$, so does $\Phi_{\widetilde{c}}$ at $(\overline{z},0)$ because the strongly convex function $\mathbb{R}^n\ni d\mapsto \widetilde{c}\|d\|^2$ is a KL function of exponent $1/2$. According to Theorem \ref{globalconv}, the sequence $\{x^k\}_{k\in\mathbb{N}}$ converges to $x^*$. To prove items (i)-(ii), for each $k\in\mathbb{N}$, write $\Delta_k\!:=\!\sum_{j=k}^{\infty}\|x^{j+1}\!-\!x^j\|$ and $V_k:=F(x^k)-\varpi^*$. From the proof of Theorem \ref{globalconv}, it suffices to consider that $F(x^k)>F(x^{k+1})>\varpi^*$ for all $k\in\mathbb{N}$. Note that \eqref{equa1-later} holds for $\mathbb{R}_{+}\ni t\mapsto\varphi(t)=ct^{1-q}\ (c>0)$, i.e.,   
\begin{equation}\label{ineq-F}
 (F(x^{k})\!-\!\varpi^*)^{q}\le(\Phi_{\widetilde{c}}(w^{k-1})-\varpi^*)^{q} 
 \le\gamma c(1-q)\|x^k-x^{k-1}\|\quad\forall k\ge\widehat{k}_1.
\end{equation}
Combining \eqref{ineq-F} with \eqref{F-decrease} shows that  the following inequality holds for all $k\ge\widehat{k}_1$:
\begin{equation}\label{recursion1-Vk}
V_k^{2q}\le[\gamma c(1-q)]^2\|x^{k}-x^{k-1}\|^2\le c_1(V_{k-1}-V_k)\ \ {\rm with}\ c_1=2\alpha^{-1}[\gamma c(1-q)]^2.
\end{equation}
In addition, by passing the limit $l\to\infty$ to \eqref{temp-ratek} and using $\varphi(t)=ct^{1-q}$, it follows
\[
 \Delta_k\le \|x^k\!-\!x^{k-1}\|+2\alpha^{-1}\gamma c\big(F(x^k)-\varpi^*\big)^{1-q}\quad\forall k\ge\widehat{k}_1.
\]
Recall that $\lim_{k\to\infty}\|x^{k+1}-x^k\|=0$ and $\frac{1-q}{q}\le1$. If necessary by increasing $\widehat{k}_1$, we have $\|x^k\!-\!x^{k-1}\|\le \|x^k\!-\!x^{k-1}\|^{\frac{1-q}{q}}$. Then, the above inequality implies that for all $k\ge\widehat{k}_1$, 
\begin{align}\label{recursion1-Deltak}
\Delta_k&\le \|x^k\!-\!x^{k-1}\|^{\frac{1-q}{q}}+2\alpha^{-1}\gamma c\big(F(x^k)-\varpi^*\big)^{1-q}\nonumber\\
&\stackrel{\eqref{ineq-F}}{\le}\|x^k-x^{k-1}\|^{\frac{1-q}{q}}+2\alpha^{-1}(\gamma c)^{\frac{1}{q}}(1\!-\!q)^{\frac{1-q}{q}}\|x^k-x^{k-1}\|^{\frac{1-q}{q}}\nonumber\\ 
 &=c_2(\Delta_{k-1}-\Delta_k)^{\frac{1-q}{q}}\ \ {\rm with}\ c_2=1\!+2\alpha^{-1}(\gamma c)^{\frac{1}{q}}(1\!-\!q)^{\frac{1-q}{q}}.
\end{align}
 With \eqref{recursion1-Vk}-\eqref{recursion1-Deltak}, using the same arguments as those for \cite[Theorem 2] {Attouch09} yields the result. \end{proof}

The KL property of $\Phi$ with exponent $q\in[1/2,1)$ is the key to achieve the convergence rate of the sequence $\{x^k\}_{k\in\mathbb{N}}$. However, unlike the KL property, the KL property of exponent $q\in[1/2,1)$ is few and far between for nonconvex and nonsmooth functions except for those with some special structures; see \cite{LiPong18,WuPanBi21,YuLiPong21}. Thus, for a general nonconvex and nonsmooth function, it is important to provide a reasonable condition to ensure its KL property of exponent $q\in[1/2,1)$. Motivated by this, the rest of this section focuses on the condition for the KL property of $\Phi$ with exponent $q\in[1/2,1)$. This requires the following exact characterization for the subdifferential of $\Phi$.  

\begin{lemma}\label{lemma-subdiff}
Let $\mathcal{G}(x,s,L):=G(x,s,L)-\mathbb{R}_{-}^m$ for $(x,s,L)\in\mathbb{R}^n\times\mathbb{R}^n\times\mathbb{R}_{++}^m$. Consider any point $\overline{z}\!:=\!((\overline{x},\overline{s},\overline{L}),\overline{\xi})\in G^{-1}(\mathbb{R}_{-}^m)\times(-{\rm dom}\,\partial\psi^*)$. Suppose that $g$ is twice continuously differentiable at $\overline{x}$, and that the mapping $\mathcal{G}$ is subregular at $((\overline{x},\overline{s},\overline{L}),0_{\mathbb{R}^m})$. Then,
 \begin{equation*}
 \widehat{\partial}\Phi(\overline{z})=\partial\Phi(\overline{z})=\!\begin{pmatrix}			[\nabla\!f(\overline{x})+\overline{\xi}\!+\!\partial\phi(\overline{x})]\times\{0_{\mathbb{R}^n}\}\times\{0_{\mathbb{R}^m}\}
 \!+\!\nabla G(\overline{x},\overline{s},\overline{L})\mathcal{N}_{\mathbb{R}_{-}^m}(G(\overline{x},\overline{s},\overline{L}))\\
 \overline{x}-\partial\psi^*(-\overline{\xi})
 \end{pmatrix}.
 \end{equation*}
 When $\overline{x}=\overline{s}$, the system $g(x)\in\mathbb{R}_{-}^m$ satisfies the MFCQ at $\overline{x}$ iff the system $G(x,s,L)\in\mathbb{R}_{-}^m$ satisfies the MFCQ at $(\overline{x},\overline{s},\overline{L})$, or equivalently $\mathcal{G}$ is metrically regular at $((\overline{x},\overline{s},\overline{L}),0_{\mathbb{R}^m})$.
\end{lemma}
\begin{proof}
 From the subregularity of $\mathcal{G}$ at $((\overline{x},\overline{s},\overline{L}),0_{\mathbb{R}^m})$ and \cite[Page 211]{Ioffe08}, it follows that
 \begin{equation*}
 \partial\Phi(\overline{z})\subset\!\begin{pmatrix}			[\nabla\!f(\overline{x})+\overline{\xi}\!+\!\partial\phi(\overline{x})]\times\{0_{\mathbb{R}^n}\}\times\{0_{\mathbb{R}^m}\}\!+\!\nabla G(\overline{x},\overline{s},\overline{L})\mathcal{N}_{\mathbb{R}_{-}^m}(G(\overline{x},\overline{s},\overline{L}))\\
 \overline{x}-\partial\psi^*(-\overline{\xi})
 \end{pmatrix}.
\end{equation*}
 Along with \eqref{aim-inclusion} for $z=\overline{z}$ and $\partial\Phi(\overline{z})\supset\widehat{\partial}\Phi(\overline{z})$, we get the desired equality. Now assume that $\overline{x}=\overline{s}$. Then, $G(\overline{x},\overline{s},\overline{L})=g(\overline{x})$ and ${\rm Ker}\nabla G(\overline{x},\overline{s},\overline{L})={\rm Ker}\nabla g(\overline{x})$, and the former implies $\mathcal{N}_{\mathbb{R}_{-}^m}(G(\overline{x},\overline{s},\overline{L}))=\mathcal{N}_{\mathbb{R}_{-}^m}(g(\overline{x}))$. From Section \ref{sec2.1}, the system $g(x)\in\mathbb{R}_{-}^m$ satisfies the MFCQ at $\overline{x}$ iff ${\rm Ker}(\nabla g(\overline{x}))\cap\mathcal{N}_{\mathbb{R}_{-}^m}(g(\overline{x}))=\{0_{\mathbb{R}^m}\}$, while the system  $G(x,s,L)\in\mathbb{R}_{-}^m$ satisfies the MFCQ at $(\overline{x},\overline{s},\overline{L})$ iff ${\rm Ker}(\nabla G(\overline{x},\overline{s},\overline{L}))\cap\mathcal{N}_{\mathbb{R}_{-}^m}(G(\overline{x},\overline{s},\overline{L}))=\{0_{\mathbb{R}^m}\}$. The conclusion follows.
\end{proof}

Let  $h:\mathbb{R}^n\times\mathbb{R}^m\times\mathbb{R}^n\to\overline{\mathbb{R}}$ denote the following extended real-valued function 
\begin{equation}\label{hfun}
h(x,\omega,\xi):=f(x)+\phi(x)+\delta_{\mathbb{R}_{-}^m}(\omega)+\langle x,\xi\rangle+\psi^*(-\xi).
\end{equation}
Comparing with \eqref{Phi-def},  $\Phi(z)=h(x,G(x,s,L),\xi)$ for all $z=(x,s,L,\xi)\in\mathbb{Z}$. Observe that $h$ is almost separable, so its KL property of exponent $q\in[1/2,1)$ is relatively easy to establish. Based on this and the composite structure of $\Phi$, we are ready to provide a checkable condition for $\Phi$ to have the KL property of exponent $q\in[1/2,1$. 
\begin{proposition}\label{prop-KLexp}
 Consider any $\overline{z}=(\overline{x},\overline{s},\overline{L},\overline{\xi})\in Z^*$. Suppose that $g$ is twice continuously differentiable at $\overline{x}$ and that the function $h$ in \eqref{hfun} has the KL property of exponent $q\in[1/2,1)$ at $(\overline{x},g(\overline{x}),\overline{\xi})$. Then, $\Phi$ has the KL property of exponent $q$ at $\overline{z}$ under the condition
 \begin{equation}\label{cond-KL}
 {\rm Ker}[\mathcal{I}\ \ \nabla\!g(\overline{x})]\cap\big[\mathbb{R}^n\times\mathcal{N}_{\mathbb{R}_{-}^m}(g(\overline{x}))\big]=\{(0_{\mathbb{R}^n},0_{\mathbb{R}^m})\}.
 \end{equation}
\end{proposition}
\begin{proof}
 Note that the condition \eqref{cond-KL} implies ${\rm Ker}\,\nabla g(\overline{x})\cap\mathcal{N}_{\mathbb{R}_{-}^m}(g(\overline{x}))=\{0_{\mathbb{R}^m}\}$, i.e., the system $g(x)\in\mathbb{R}_{-}^m$ satisfies the MFCQ at $\overline{x}$ according to the discussion in Section \ref{sec2.1}. Since $\overline{x}=\overline{s}$ by Proposition \ref{prop-ZWstar} (ii), the mapping $\mathcal{G}$ in Lemma \ref{lemma-subdiff} is metrically regular at  $((\overline{x},\overline{s},\overline{L}),0_{\mathbb{R}^m})$. From the robustness of metric regularity, there exists $\varepsilon>0$ such that $\mathcal{G}$ is metrically regular at any $((x,s,L),0_{\mathbb{R}^m})$ with $(x,s,L)\in\mathbb{B}((\overline{x},\overline{s},\overline{L}),\varepsilon)\cap G^{-1}(\mathbb{R}_{-}^m)$, and consequently, 
 \begin{equation*}
 \partial\Phi(z)=\!\begin{pmatrix}			[\nabla\!f(x)+\xi\!+\!\partial\phi(x)]\times\{0_{\mathbb{R}^n}\}\times\{0_{\mathbb{R}^m}\}
 \!+\!\nabla G(x,s,L)\mathcal{N}_{\mathbb{R}_{-}^m}(G(x,s,L))\\
 x-\partial\psi^*(-\xi)
 \end{pmatrix}.
 \end{equation*}  
 Now suppose on the contrary that $\Phi$ does not have the KL property of exponent $q$ at $\overline{z}$. There exists a sequence $z^k:=(x^k,s^k,L^k,\xi^k)\to\overline{z}$ with $\Phi(\overline{z})<\Phi(z^k)<\Phi(\overline{z})+\frac{1}{k}$ such that 
 \[
  {\rm dist}(0,\partial\Phi(z^k))<\frac{1}{k}(\Phi(z^k)-\Phi(\overline{z}))^{q}=\frac{1}{k}(h(x^k,\omega^k,\xi^k)-\overline{\nu})^{q}\quad\forall k\in\mathbb{N},
 \] 
 where $\omega^k\!:=G(x^k,s^k,L^k)$ for each $k\in\mathbb{N}$ and $\overline{\nu}\!:=h(\overline{x},\overline{\omega},\overline{\xi})$ for $\overline{\omega}:=g(\overline{x})$. Obviously, for each $k\in\mathbb{N}$, $z^k\in{\rm dom}\,\Phi$, so $(x^k,s^k,L^k)\in G^{-1}(\mathbb{R}_{-}^m)$. Along with $z^k\to\overline{z}$, there exists $\widetilde{k}\in\mathbb{N}$ such that $\mathcal{G}$ is metrically regular at each $(x^k,s^k,L^k)$ for $k\ge\widetilde{k}$. From the above two equations, for each $k\ge\widetilde{k}$, there exist $v^k\in\partial\phi(x^k)$, $\lambda^k\in\mathcal{N}_{\mathbb{R}_{-}^m}(\omega^k)$ and $y^k\in\partial\psi^*(-\xi^k)$ such that
 \begin{equation}\label{kl1}
 \left\|\begin{pmatrix}
 \nabla\!f(x^k)+\xi^k+v^k+\nabla\!g(s^k)\lambda^k+(x^k-s^k)\langle L^k,\lambda^k\rangle\\
 (s^k-x^k)\langle L^k,\lambda^k\rangle+\sum_{i=1}^m[\lambda_i^k\nabla^2g_i(s^k)(x^k-s^k)]\\
 \frac{1}{2}\|x^k-s^k\|^2\lambda^k\\
  x^k-y^k
 \end{pmatrix}\right\|<\frac{1}{k}(h(x^k,\omega^k,\xi^k)-\overline{\nu})^{q}.
 \end{equation}
 Since $h$ has the KL property of exponent $q$ at $(\overline{x},\overline{\omega},\overline{\xi})$, there exist $\delta>0,\eta>0$ and $c>0$ such that for all $(x,\omega,\xi)\in\mathbb{B}((\overline{x},\overline{\omega},\overline{\xi}),\delta)\cap[\overline{\nu}<h<\overline{\nu}+\eta]$,
 \begin{equation*}
 {\rm dist}(0,\partial h(x,\omega,\xi))\ge c(h(x,\omega,\xi)-\overline{\nu})^{q}.
 \end{equation*}
 From $z^k\to\overline{z}$ and the expression of $G$ in \eqref{def-Gmap}, we have $\omega^k\to\overline{\omega}$. Recall that for each $k\in\mathbb{N}$, $\Phi(\overline{z})<\Phi(z^k)<\Phi(\overline{z})+\frac{1}{k}$, i.e., $\overline{\nu}<h(x^k,\omega^k,\xi^k)<\overline{\nu}+\frac{1}{k}$. If necessary by increasing $\widetilde{k}$, for each $k\ge\widetilde{k}$, $(x^k,\omega^k,\xi^k)\in\mathbb{B}((\overline{x},\overline{\omega},\overline{\xi}),\delta)\cap[\overline{\nu}<h<\overline{\nu}+\eta]$, and from the above inequality, 
 \begin{equation}\label{kl3}
 c(h(x^k,\omega^k,\xi^k)-\overline{\nu})^q\le{\rm dist}(0,\partial h(x^k,\omega^k,\xi^k))
 \le\|(\nabla\!f(x^k)+v^k+\xi^k,\lambda^k,x^k\!-\!y^k)\|,
 \end{equation}
 where the second inequality is due to $(\nabla\!f(x^k)+\xi^k+v^k,\lambda^k,x^k\!-\!y^k)\in\partial h(x^k,\omega^k,\xi^k)$ for each $k\in\mathbb{N}$. Furthermore, this inclusion along with the first inequality in \eqref{kl3} implies that 
 \[
  t_k\!:=\|(\nabla\!f(x^k)+v^k+\xi^k,\lambda^k,x^k\!-\!y^k)\|\ge c(h(x^k,\omega^k,\xi^k)-\overline{\nu})^q>0\quad\forall k\ge\widetilde{k}. 
 \]
 For each $k\ge\widetilde{k}$, write $(\widetilde{\varpi}^k,\widetilde{\xi}^k,\widetilde{\lambda}^k,\widetilde{\zeta}^k)\!:=\frac{1}{t_k}(\nabla\!f(x^k)+v^k,\xi^k,\lambda^k,x^k-y^k)$.  
 Together with the above inequalities \eqref{kl1} and \eqref{kl3}, for each $k\ge\widetilde{k}$, it holds 
 \begin{equation}\label{kl4}
 \left\|\begin{pmatrix}
\widetilde{\varpi}^k+\widetilde{\xi}^k+\nabla\!g(s^k)\widetilde{\lambda}^k+\langle L^k,\widetilde{\lambda}^k\rangle(x^k-s^k)\\
 \langle L^k,\widetilde{\lambda}^k\rangle(s^k-x^k)+\sum_{i=1}^m[\widetilde{\lambda}_i^k\nabla^2g_i(s^k)(x^k-s^k)]\\
 \frac{1}{2}\|x^k-s^k\|^2\widetilde{\lambda}^k\\
 \widetilde{\zeta}^k
 \end{pmatrix}\right\|\le\frac{1}{ck}.
 \end{equation}
 From the definition of $\{(\widetilde{\varpi}^k,\widetilde{\xi}^k,\widetilde{\lambda}^k,\widetilde{\zeta}^k)\}_{k\ge\widetilde{k}}$, there exists an index set $\mathcal{K}\subset\mathbb{N}$ such that $\{(\widetilde{\varpi}^k,\widetilde{\xi}^k,\widetilde{\lambda}^k,\widetilde{\zeta}^k)\}_{k\in\mathcal{K}}$ is convergent with the limit $(\widetilde{\varpi},\widetilde{\xi},\widetilde{\lambda},\widetilde{\zeta})$ satisfying $\|(\widetilde{\varpi}+\widetilde{\xi},\widetilde{\lambda},\widetilde{\zeta})\|=1$.
	
 Next we claim that the sequence $\{(\lambda^k,y^k)\}_{k\in\mathcal{K}}$ is bounded. Recall that $v^k\in\partial\phi(x^k)$ for each $k\ge\widetilde{k}$ and the sequence $\{x^k\}_{k\in\mathbb{N}}$ is bounded. From Remark \ref{remark-subdiff} and \cite[Proposition 5.15]{RW98}, the sequence $\{v^k\}_{k\in\mathbb{N}}$ is bounded. Suppose on the contrary that $\{(\lambda^k,y^k)\}_{k\in\mathcal{K}}$ is unbounded. By the definition of $t_k$ and the boundedness of $\{(x^k,v^k,\xi^k)\}_{k\in\mathbb{N}}$, we get $\lim_{\mathcal{K}\ni k\to\infty}t_k=\infty$. This means that $\widetilde{\varpi}=0$ and $\widetilde{\xi}=0$. 
 If necessary by taking a subsequence, we can assume $\lim_{\mathcal{K}\ni k\to\infty}(\widetilde{\lambda}^k,\widetilde{\zeta}^k)=(\widetilde{\lambda},\widetilde{\zeta})$. Then, $\|(\widetilde{\lambda},\widetilde{\zeta})\|=1$. Now passing the limit $\mathcal{K}\ni k\to\infty$ to \eqref{kl4} results in $\nabla\!g(\overline{x})\widetilde{\lambda}=0$ and $\widetilde{\zeta}=0$. The latter, along with $\|(\widetilde{\lambda},\widetilde{\zeta})\|=1$, implies $\|\widetilde{\lambda}\|=1$. Recall that $\widetilde{\lambda}^k\in\mathcal{N}_{\mathbb{R}_{-}^m}(\omega^k)$ for each $k\ge\widetilde{k}$. The osc property of the mapping $\mathcal{N}_{\mathbb{R}_{-}^m}$ implies $\widetilde{\lambda}\in\mathcal{N}_{\mathbb{R}_{-}^m}(g(\overline{x}))$. Thus, we obtain
 \(
  0\ne\widetilde{\lambda}\in{\rm Ker}(\nabla g(\overline{x}))\cap\mathcal{N}_{\mathbb{R}_{-}^m}(g(\overline{x})), 
 \)
 a contradiction to the MFCQ for the system $g(x)\in\mathbb{R}_{-}^m$ at $\overline{x}$. Hence, $\{(\lambda^k,y^k)\}_{k\in\mathcal{K}}$ is bounded. If necessary by taking a subsequence, we assume $\lim_{\mathcal{K}\ni k\to\infty}y^k=\overline{y}$. Note that $\xi^k\in-\partial\psi(y^k)$ for each $k\ge\widetilde{k}$. Together with the expression of $(\widetilde{\varpi}^k,\widetilde{\xi}^k,\widetilde{\lambda}^k,\widetilde{\zeta}^k)$, for each $k\ge\widetilde{k}$, it holds
 \[
(\widetilde{\varpi}^k+\widetilde{\xi}^k,\widetilde{\lambda}^k)\in{\rm pos}(\nabla\!f(x^k)\!+\!\partial\phi(x^k)-\partial\psi(y^k))\times\mathcal{N}_{\mathbb{R}_{-}^m}(\omega^k),
 \]
 where ${\rm pos}(C):=\{0\}\cup\big\{tx\,|\,t>0,x\in C\big\}$ is the positive hull of a set $C$ (see \cite[Section 3G]{RW98}).
 Passing the limit $\mathcal{K}\ni k\to\infty$ to the above inclusion and using the osc property of the mapping $\mathcal{N}_{\mathbb{R}_-^m}$ results in 
 \[
 (\widetilde{\varpi}+\widetilde{\xi},\widetilde{\lambda})\in\mathbb{R}^n\times\mathcal{N}_{\mathbb{R}_-^m}(g(\overline{x})).
 \]
On the other hand, passing the limit $\mathcal{K}\ni k\to\infty$ to the above \eqref{kl4} leads to
\[
\widetilde{\varpi}+\widetilde{\xi}+\nabla\!g(\overline{x})\widetilde{\lambda}=0\ \ {\rm and}\ \ \widetilde{\zeta}=0. 
\]
Then, the above two equations and $\|(\widetilde{\varpi}+\widetilde{\xi},\widetilde{\lambda})\|=1$ yield a contradiction to \eqref{cond-KL}. 
\end{proof}
\begin{remark}
 {\bf(a)}
 Note that $\Phi(z)=h(K(z))$ where $K\!:\mathbb{Z}\to\mathbb{R}^{n}\times\mathbb{R}^m\times\mathbb{R}^n$ is a smooth mapping defined by $K(z):=(x;G(x,s,L);\xi)$ for $z=(x,s,L,\xi)\in\mathbb{Z}$. For such a composite function, a criterion to identify its KL property of exponent $q\in[1/2,1)$ at $\overline{z}$ was provided in \cite[Theorem 3.2]{LiPong18} by the surjectivity of $K'(\overline{z})$ and the KL property of $h$ with exponent $q\in[1/2,1)$. Note that the surjectivity of $K'(\overline{z})$ is equivalent to ${\rm Ker}[I\ \ \nabla\!g(\overline{x})]=\{0_{\mathbb{R}^n}\}\times\{0_{\mathbb{R}^m}\}$ in view of $\overline{s}=\overline{x}$, which is obviously stronger than \eqref{cond-KL} because the set $\mathbb{R}^n\times\mathcal{N}_{\mathbb{R}_{-}^m}(g(\overline{x}))$ is  smaller than the whole space ${\mathbb{R}^n}\times\mathbb{R}^m$. Thus, the criterion in Proposition \ref{prop-KLexp} to identify the KL property of $\Phi$ with exponent $q\in[1/2,1)$ is weaker than that of \cite[Theorem 3.2]{LiPong18}. Moreover, the outer $h$ is a KL function of exponent $1/2$ when $f,\phi$ and $\psi$ are piecewise linear-quadratic functions; or when $\psi\equiv 0$ and $f+\phi$ is a composite function from \cite{ZhouSo17}.

\end{remark}

\section{Iteration complexity analysis}\label{sec5}

Complexity analysis of algorithms for constrained optimization problems concerns convergence to an approximate KKT point, so a reasonable measure for it is needed. For any $(x,\xi,\lambda)\in\mathbb{R}^n\times(-{\rm dom}\,\psi^*)\times\mathbb{R}_{+}^m$, define the Lagrange function of \eqref{prob} by 
\begin{equation}\label{Lfun}
 \mathcal{L}(x,\xi,\lambda):=f(x)+\phi(x)+\langle x,\xi\rangle+\psi^*(-\xi)+\langle\lambda,g(x)\rangle.
\end{equation}
 Compared with Definition \ref{spoint-def}, a vector $x\in\Gamma$ is a stationary point of \eqref{prob} if and only if there exist $\xi\in-\partial\psi(x)$ and $\lambda\in\mathbb{R}_{+}^m$ such that $(x,\xi,\lambda)$ satisfies the following system
\begin{subnumcases}{}     
 0\in\partial_x\mathcal{L}(x,\xi,\lambda)=\nabla\!f(x)+\partial\phi(x)+\xi+\nabla g(x)\lambda,\nonumber\\  
 0\in\partial_{\xi}\mathcal{L}(x,\xi,\lambda)=x-\partial\psi^*(-\xi),\nonumber\\     g(x)\in\mathbb{R}_{-}^m,\,\lambda\in\mathbb{R}_{+}^m,\,\langle g(x),\lambda\rangle=0,\nonumber
\end{subnumcases}
and the triple $(x,\xi,\lambda)$ is naturally called a KKT point of \eqref{prob}. Consider that the iterate sequence $\{x^k\}_{k\in\mathbb{N}}$ of Algorithm \ref{iMBA} belongs to the feasible set $\Gamma$. We adopt the measure
\[
 \chi(x,\xi,\lambda):=\max\big\{{\rm dist}(0,\partial_x\mathcal{L}(x,\xi,\lambda)),{\rm dist}(0,\partial_{\xi}\mathcal{L}(x,\xi,\lambda)),\sqrt{-\langle g(x),\lambda\rangle}\big\}
\]
for  $(x,\xi,\lambda)\in\Gamma\times(-{\rm dom}\,\psi^*)\times\mathbb{R}_{+}^m$ to introduce the notion of the $\epsilon$-KKT points. 
\begin{definition}\label{aKKT}
 A tripe $(x,\xi,\lambda)\in\mathbb{R}^n\times(-{\rm dom}\,\psi^*)\times\mathbb{R}_{+}^m$ is called an $\epsilon$-KKT point of \eqref{prob} if $\chi(x,\xi,\lambda)\le\epsilon$.
\end{definition}
When $\psi\equiv 0$, the $\epsilon$-KKT point in Definition \ref{aKKT} coincides with the one in \cite[Section 4.1]{Nabou24} for $q=1$, and the $\sqrt{\epsilon}$-KKT point in Definition \ref{aKKT} corresponds to the $\epsilon$-type I KKT point in \cite[Definition 3]{Boob24} that is superior to the $(\epsilon,\nu)$ type-II KKT point there. To analyze the iteration complexity of Algorithm \ref{iMBA} to achieve an $\epsilon$-KKT point, we first provide an upper bound for $\chi(x^{k+1},\xi^k,\lambda^{k+1})$ with $\|x^{k+1}-x^k\|$ as follows. 
\begin{proposition}\label{prop-cbound}
 Suppose that Assumption \ref{ass3} holds, that the system $g(x)\in\mathbb{R}_{-}^m$ satisfies the MFCQ at all feasible points, and that $\nabla\!f,\nabla\!g_1,\ldots,\nabla\!g_m$ are Lipschitz continuous on $\mathcal{O}$ with Lipschitz constants $L_{\!f},L_{g_1},\ldots,L_{g_m}$, respectively. Let $L_{g}\!:=\|(L_{g_1},\ldots,L_{g_m})\|$. Then, 
\begin{description}
\item[(i)] there exists $\widehat{c}_{R}>0$ such that $({\beta_{R}}/{2})\|x^{k+1}-x^k\|\le\widehat{c}_{R}$ for all $k\in\mathbb{N}$;

\item [(ii)] there exists $\beta_{\lambda}>0$ such that $\|\lambda^k\|\le\beta_{\lambda}$ for all $k\in\mathbb{N}$;

\item [(iii)] for each $k\in\mathbb{N}$, 
$\chi(x^{k+1},\xi^k,\lambda^{k+1})\le C(L_{\!f},L_{g},\beta_{R},\widehat{c}_{R},\beta_{L},\beta_{\mathcal{Q}},\beta_{\lambda})\|x^{k+1}\!-x^k\|$ with   $C(L_{\!f},L_{g},\beta_{R},\widehat{c}_{R},\beta_{L},\beta_{\mathcal{Q}},\beta_{\lambda})=\max\big\{(L_{\!f}+\widehat{c}_{R}+\beta_{\mathcal{Q}}\!+\!\beta_{\lambda}(L_g\!+\!\beta_{L})),1,\sqrt{\frac{\beta_{\lambda}(L_g\!+\!\beta_{L})+\beta_{R}}{2}}\big\}$.
\end{description}
\end{proposition}
\begin{proof}
 {\bf (i)-(ii)} Item (i) is immediate by Assumption \ref{ass3}, so it suffices to prove item (ii). Suppose on the contrary the sequence $\{\lambda^k\}_{k\in\mathbb{N}}$ is unbounded, i.e., $\lim_{k\to\infty}\|\lambda^k\|=\infty$. For each $k\in\mathbb{N}$, let $\widetilde{\lambda}^k\!:=\frac{\lambda^{k}}{\|\lambda^{k}\|}$. Then there exists an index set $\mathcal{K}\subset\mathbb{N}$ such that $\lim_{\mathcal{K}\ni k\to\infty}\widetilde{\lambda}^k=\widetilde{\lambda}$ for some $\widetilde{\lambda}\in\mathbb{R}_+^m$ with $\|\widetilde{\lambda}\|=1$. By Assumption \ref{ass3}, if necessary by taking a subsequence, we can assume $\lim_{\mathcal{K}\ni k\to\infty}x^{k}=\widetilde{x}$. For each $k\in\mathbb{N}$, let $\zeta^k\!:=\nabla\vartheta_k(x^{k+1})+v^{k+1}+\nabla_xG(x^{k+1},x^k,L^k)\lambda^{k+1}$. From the definition of $R_{k,j_k}$ in \eqref{resi-sub} and inequality \eqref{inexact1}, for each $k\in\mathbb{N}$, it holds that
 \begin{equation}\label{ineq0-itcomplex}
 \max\big\{\|\zeta^k\|,\|[G(x^{k+1},x^k,L^k)]_{+}\|_{\infty},\big[\!-\!\langle\lambda^{k+1},G(x^{k+1},x^k,L^k)\rangle\big]_+\big\}\le\frac{\beta_{R}}{2}\|x^{k+1}\!-x^k\|^2.
 \end{equation}
 Recall that $v^k\in\partial\phi(x^k)$ for all $k\in\mathbb{N}$. The sequence $\{v^k\}_{k\in\mathbb{N}}$ is bounded by Remark \ref{bound-subdiff}, so is the sequence $\{\nabla\vartheta_k(x^{k+1})\!+\!v^{k+1}\}_{k\in\mathbb{N}}$. Now dividing the inequality \eqref{ineq0-itcomplex} by $\|\lambda^{k+1}\|$, passing the limit $\mathcal{K}\ni k\to\infty$ and using $\lim_{k\to\infty}\|\lambda^k\|=\infty$ and $\lim_{k\to\infty}\|x^{k+1}\!-x^k\|=0$ leads to 
 \begin{equation}\label{result1}
 \nabla\!g(\widetilde{x})\widetilde{\lambda}=0,\,\|[g(\widetilde{x})]_{+}\|_{\infty}=0\ \ {\rm and}\ \ [-\langle\widetilde{\lambda},g(\widetilde{x})\rangle]_{+}=0.
 \end{equation}
 The third equality in \eqref{result1}, along with $\widetilde{\lambda}\in\mathbb{R}_+^m$ and $g(\widetilde{x})\in\mathbb{R}_-^m$, is equivalent to $\langle\widetilde{\lambda},g(\widetilde{x})\rangle=0$, and consequently, $\widetilde{\lambda}\in\mathcal{N}_{\mathbb{R}_{-}^m}(g(\widetilde{x}))$. Along with the first equality in \eqref{result1}, it follows $0\ne\widetilde{\lambda}\in {\rm Ker}(\nabla g(\widetilde{x}))\cap\mathcal{N}_{\mathbb{R}_{-}^m}(g(\widetilde{x}))$, a contradiction to the assumption that the system $g(x)\in\mathbb{R}_{-}^m$ satisfies the MFCQ at all feasible points. Thus, item (ii) holds. 

 \noindent
 {\bf (iii)} Fix any $k\in\mathbb{N}$. By the expression of $\mathcal{L}$ and the definition of $\zeta^k$ in the proof of item (ii), 
 \begin{align}\label{measure1}
 {\rm dist}(0,\partial_x\mathcal{L}(x^{k+1},\xi^k,\lambda^{k+1}))&=\|\nabla\!f(x^{k+1})+v^{k+1}+\xi^k+\nabla\!g(x^{k+1})\lambda^{k+1}\|\nonumber\\
 &=\|\nabla\!f(x^{k+1})\!+\!\zeta^k\!-\!\nabla\!f(x^k)\!-\!\mathcal{Q}_k(x^{k+1}\!-\!x^k)\nonumber\\
 &\quad +(\nabla\!g(x^{k+1})-\nabla\!g(x^{k}))\lambda^{k+1}\!-\!\langle L^k,\lambda^{k+1}\rangle(x^{k+1}\!-\!x^k)\|\nonumber\\
 &\le L_{\!f}\|x^{k+1}\!-\!x^k\|+\|\zeta^k\|+[\beta_{\mathcal{Q}}\!+\!\beta_{\lambda}(L_g\!+\beta_{L})]\|x^{k+1}\!-\!x^k\|\nonumber\\
 &\le\big(L_f+\widehat{c}_{R}+\beta_{\mathcal{Q}}+\beta_{\lambda}(L_g\!+\beta_{L})\big)\|x^{k+1}\!-\!x^k\|,
 \end{align}
 where the first inequality is got by using the Lipschitz continuity of $\nabla\!f,\nabla\!g_1,\ldots,\nabla\!g_m$ on $\mathcal{O}$ and Lemma \ref{lemma-bound} (ii)-(iii), and the second is due to $\|\zeta^k\|\le\widehat{c}_{R}\|x^{k+1}\!-x^k\|$ in view of \eqref{ineq0-itcomplex} and item (i). Since $\xi^k\in-\partial\psi(x^k)$, we have $x^k\in\partial\psi^*(-\xi^k)$, so that
 \begin{equation}\label{measure2}
 {\rm dist}(0,\partial_{\xi}\mathcal{L}(x^{k+1},\xi^k,\lambda^{k+1}))={\rm dist}(x^{k+1},\partial\psi^*(-\xi^k))\le\|x^{k+1}-x^k\|.
 \end{equation}
 Next we bound the complementarity violation. Recall that $\nabla\!g_1,\ldots,\nabla\!g_m$ are assumed to be Lipschitz continuous. Using the expression of $G$ and the descent lemma, for each $i\in[m]$, we have $\|g_i(x^{k+1})-G_i(x^{k+1},x^k,L^k)\|\le \frac{L_{g_i}+L_i^k}{2}\|x^{k+1}-x^k\|^2$, and consequently, 
 \begin{equation}\label{ineq-gG}
 \|g(x^{k+1})-G(x^{k+1},x^k,L^k)\|\le \frac{L_{g}+\beta_{L}}{2}\|x^{k+1}-x^k\|^2.
 \end{equation}
 Now combining inequality \eqref{ineq-gG} with the above \eqref{ineq0-itcomplex} leads to  
 \begin{align*}
  -\langle\lambda^{k+1},g(x^{k+1})\rangle
  &=\langle\lambda^{k+1},G(x^{k+1},x^k,L^k)-g(x^{k+1})-G(x^{k+1},x^k,L^k)\rangle\\
  &\le\langle\lambda^{k+1},G(x^{k+1},x^k,L^k)-g(x^{k+1})\rangle+[-\langle\lambda^{k+1},G(x^{k+1},x^k,L^k)\rangle]_{+}\\
  &\stackrel{\eqref{ineq0-itcomplex}}{\le}\|\lambda^{k+1}\|\|G(x^{k+1},x^k,L^k)-g(x^{k+1})\|+\frac{\beta_{R}}{2}\|x^{k+1}-x^k\|^2\\
  &\stackrel{\eqref{ineq-gG}}{\le}\frac{\beta_{\lambda}(L_g+\beta_{L})+\beta_{R}}{2}\|x^{k+1}-x^k\|^2.
 \end{align*}
 Together with \eqref{measure1}-\eqref{measure2} and the definition of $\chi(x,\xi,\lambda)$, we obtain the desired result.  
\end{proof}

Now we are ready to establish the iteration complexity result of Algorithm \ref{iMBA}. 
\begin{theorem}\label{complexity}
 Under the assumptions of Proposition \ref{prop-cbound}, for any $K\in\mathbb{N}$, it holds
 \[
  \min_{0\le k\le K}\chi(x^{k+1},\xi^k,\lambda^{k+1}) \le\frac{C(L_{\!f},L_{g},\beta_{R},\widehat{c}_{R},\beta_{L},\beta_{\mathcal{Q}},\beta_{\lambda})}{\sqrt{K+1}}\sqrt{\frac{2(F(x^0)-{\rm val}^*)}{\alpha}},
 \]
 So an $\epsilon$-KKT point is achieved within  $\lceil\frac{2C^2(L_{\!f},L_{g},\beta_{R},\widehat{c}_{R},\beta_{L},\beta_{\mathcal{Q}},\beta_{\lambda})(F(x^0)-{\rm val}^*)}{\alpha\epsilon^2}\!-\!1\rceil$ steps. 
\end{theorem}
\begin{proof} 
 Fix any $K\in\mathbb{N}$. From Proposition \ref{prop-cbound} (iii) and the previous \eqref{F-decrease}, it holds 
 \begin{align*}
 \sum_{k=0}^K\chi^2(x^{k+1},\xi^k,\lambda^{k+1})&\le\sum_{k=0}^K \frac{2C^2(L_{\!f},L_{g},\widehat{c}_{R},\beta_{\lambda},\beta_{L},\beta_{R})(F(x^k)-F(x^{k+1}))}{\alpha}\\
 &\le \frac{2C^2(L_{\!f},L_{g},\widehat{c}_{R},\beta_{\lambda},\beta_{L},\beta_{R})(F(x^0)-{\rm val}^*)}{\alpha},
 \end{align*}
 which implies the desired conclusion. The proof is completed.
\end{proof}
\begin{remark}\label{remark-complexity}
The complexity bound $O(1/\epsilon^2)$ in Theorem \ref{complexity} has the same order as the one in \cite[Theorem 4.3]{Nabou24} for the exact MBA method to achieve an $\epsilon$-KKT point of \eqref{prob} with $\psi\equiv 0$, and does as the one in \cite[Theorem 8]{Boob24} for the inexact LCPG method to seek an $(\epsilon^2,\delta)$ type-II KKT point, worse than the $\epsilon$-KKT point. The complexity bound depends on the constant $C(L_{\!f},L_{g},\beta_{R},\widehat{c}_{R},\beta_{L},\beta_{\mathcal{Q}},\beta_{\lambda})$, 
determined by the data constants $L_{\!f},L_{g}$, the parameter $\beta_{R}$ of Algorithm \ref{iMBA}, and the constants $\widehat{c}_{R},\beta_{L},\beta_{\mathcal{Q}}$ and $\beta_{\lambda}$. Among others, $\beta_{R}$ is provided by the user, and $\beta_{L},\beta_{\mathcal{Q}}$ can be estimated with the setup of parameters of Algorithm \ref{iMBA} in view of Lemma \ref{inner-step}. Recall that $\{x^k\}_{k\in\mathbb{N}}\subset\mathcal{L}_{F(x^0)}$. As long as $F$ is level-coercive, $\widehat{c}_{R}$ can be bounded by the diameter of $\mathcal{L}_{F(x^0)}$. The constant $\beta_{\lambda}$ is dependent on the CQ for the system $g(x)\in\mathbb{R}_{-}^m$, and its estimation is extremely challenging. When the subproblems are exactly solved, Boob et al. provided an estimation of $\beta_{\lambda}$ by requiring the compactness of $\Gamma$ and strong feasibility CQ (see \cite[Lemma 2]{Boob24}). When the feasible set of \eqref{prob} satisfies the strict feasibility as in \cite[Assumption 2]{Boob24}, by requiring the initial $x^0$ to satisfy $g(x^0)\le\eta^0$ for some $\eta^0<0$ and replacing the constraint sets of subproblems by $\Gamma_{k,j}\!:=[G(\cdot,x^k,L^{k,j})+\eta^0]^{-1}(\mathbb{R}_{-}^m)$, we can employ the inexactness criterion \eqref{inexact1} to estimate $\beta_{\lambda}$. Since this requires more restriction on problem \eqref{prob} and is out of the theme of this work, we leave it for a future research topic. 
\end{remark}

\section{Numerical experiments}\label{sec6}

In this section, we conduct some numerical studies to examine our theoretical results and the performance of Algorithm \ref{iMBA}. All experiments are performed on a desktop running on Matlab R2020b and 64-bit Windows System with an Intel(R) Core(TM) i9-10850K CPU 3.60GHz and 32.0 GB RAM. 
\subsection{Implementation details of Algorithm \ref{iMBA}}\label{sec6.1}

The core of Algorithm \ref{iMBA} at each iteration is to seek an inexact minimizer of subproblem \eqref{subprobkj} satisfying conditions \eqref{inexact1}-\eqref{inexact2}, so we first take a look at the solving of \eqref{subprobkj}.
\subsubsection{Solving of subproblems}\label{sec6.1.1}

Fix any $k,j\in\mathbb{N}$. For the PD linear mapping $\mathcal{Q}_{k,j}\!:\mathbb{R}^n\to\mathbb{R}^n$ in step 5, we choose it to be of the form $\mathcal{Q}_{k,j}=\mu_{k,j}\mathcal{I}+\!\mathcal{A}_{k,j}^*\mathcal{A}_{k,j}$, where $\mathcal{A}_{k,j}\!:\mathbb{R}^n\to\mathbb{R}^{p}$ is a linear mapping. Such a structure of $\mathcal{Q}_{k,j}$ makes the dual of \eqref{subprobkj} have a simple composite form. Indeed, by letting $b^k\!:=\nabla\!f(x^k)+\xi^k,d^{k,j}\!:=\mathcal{A}_{k,j}x^k$ and $C^k\!:=f(x^k)-\psi(x^k)$, subproblem \eqref{subprobkj} can equivalently be written as 
\begin{align}\label{EEsubprob}
&\min_{x,s\in\mathbb{R}^n,y\in\mathbb{R}^p}\frac{1}{2}\|y\|^2+\frac{1}{2}\mu_{k,j}\|x-x^k\|^2+\langle b^k,x-x^k\rangle+\phi(s)+C^k\nonumber\\
&\qquad {\rm s.t.}\quad G(x,x^k,L^{k,j})\le 0,\ x-s=0,\ \mathcal{A}_{k,j}x-y=d^{k,j}.
\end{align}
After an elementary calculation, the dual of \eqref{EEsubprob} has the following composite form
\begin{equation}\label{dual}	
\!\min_{w\in\mathbb{W}}\Xi_{k,j}(w)\!:=\!\underbrace{\frac{\|\mathcal{B}_{k,j}(w)+b^k\|^2}{2(\mu_{k,j}\!+\!\langle\lambda,L^{k,j}\rangle)}\!-\!\langle\eta,x^k\rangle-\!\langle\lambda,g(x^k)\rangle\!+\!\frac{1}{2}\|\zeta\|^2\!-\!C^k}_{\Theta_{k,j}(w)}+\delta_{\mathbb{R}_{+}^m}(\lambda)\!+\phi^*(\eta)
\end{equation}
where $\Theta_{k,j}$ is a twice continuously differentiable convex function on the space $\mathbb{W}:=\mathbb{R}^{m}\times\mathbb{R}^{n}\times\mathbb{R}^{p}$, and $\mathcal{B}_{k,j}\!:\mathbb{W}\to\mathbb{R}^n$ is a linear mapping defined by
\[
 \mathcal{B}_{k,j}(w):=\nabla\!g(x^k)\lambda+\eta+\mathcal{A}_{k,j}^*\zeta\quad\forall w\in\mathbb{W}.
\]

According to \cite[Theorem 2.142]{BS00}, the strong convexity of \eqref{subprobkj} guarantees that there is no dual gap between \eqref{subprobkj} and \eqref{dual}, and \eqref{dual} has a nonempty optimal solution set. Inspired by this, we seek an inexact minimizer of \eqref{subprobkj} satisfying \eqref{inexact1}-\eqref{inexact2} by solving its dual. Considering that the gradient of $\Theta_{k,j}$ is not Lipschitz continuous, we choose the PG method with line search (PGls) as the solver to \eqref{dual}. The iteration steps of PGls are described in Algorithm \ref{PGls} below, and its convergence certificate can be found in \cite{Kanzow22}. In the sequel, we call Algorithm \ref{iMBA} armed with Algorithm \ref{PGls} iMBA-pgls. 
\begin{algorithm}
 \renewcommand{\thealgorithm}{A}
 \caption{\label{PGls}{\bf (PGls for solving problem \eqref{dual})}}
 \begin{algorithmic}[1]
 \State{\textbf{Input:}\ $0<\tau_{\rm min}\le\tau_{\rm max},\,\varrho>1,\,\delta\in(0,1),w^0=(\lambda^0,\eta^0,\zeta^0)\in\mathbb{R}_{+}^m\times\mathbb{R}^n\times{\rm dom}\,\phi^*$.}
 \For{$l=0,1,2,\ldots$}
 \State{Choose $\tau_{l,0}\in[\tau_{\rm min},\tau_{\rm max}]$.} 

  \For{$\nu=0,1,2,\ldots$}		
  
  \State{Let $\tau_{l,\nu}=\tau_{l,0}\varrho^{\nu}$. Compute an optimal solution $w^{l,\nu}$ of the problem
  \begin{equation*}
   \qquad\min_{w\in\mathbb{W}}\Theta_{k,j}(w^{l})+\langle\nabla\Theta_{k,j}(w^{l}),w-w^{l}\rangle+\frac{\tau_{l,\nu}}{2}\|w-w^{l}\|^2+\delta_{\mathbb{R}_{+}^m}(\lambda)\!+\phi^*(\eta).
  \end{equation*}
  \hspace*{1.0cm} If $\Xi_{k,j}(w^{l,\nu})\le \Xi_{k,j}(w^{l})-(\delta\tau_{l,\nu}/2)\|w^{l,\nu}-w^l\|^2$, go to step 7.}
\EndFor		
\State{Set $\nu_l:=\nu$ and $w^{l+1}:=w^{l,\nu_l}$. }
\EndFor
\end{algorithmic}
\end{algorithm}
\subsubsection{Stop conditions of iMBA-pgls}\label{sec6.1.2}

According to Remark \ref{remark-alg} (d), we terminate Algorithm \ref{iMBA} at the $k$th iteration whenever $\|x^k\!-\!x^{k-1}\|\le\epsilon$ or $k>10^4$. Since Algorithm \ref{iMBA} produces in each iteration a feasible point $x^k\in\Gamma$ and a nonnegative $\lambda^k$, we also terminate its iteration once $[-\langle\lambda^k,g(x^k)\rangle]_{+}\le\epsilon_1$ and $k\ge 500$, where the tolerances $\epsilon>0$ and $\epsilon_1>0$ are specified in the experiments. 

In what follows, we discuss how to terminate Algorithm \ref{PGls}, i.e., how to seek an inexact minimizer of \eqref{subprobkj} satisfying \eqref{inexact1}-\eqref{inexact2} by solving its dual. For each $w\in\mathbb{W}$, let $\mathcal{L}(x,s,y;w)$ be the Lagrange function of \eqref{EEsubprob} associated with $w$. Note that the minimization problem $\min_{x,s\in\mathbb{R}^n,y\in\mathbb{R}^p}\mathcal{L}(x,s,y;w)$ is strongly convex w.r.t. $x$ and $y$, and $(x^*,s^*,y^*)$ is its optimal solution iff 
$x^*\!=x^k-\frac{b^k+\mathcal{A}_{k,j}^*\zeta+\eta+\nabla\!g(x^k)\lambda}{\mu_{k,j}+\langle\lambda,L^{k,j}\rangle}$, $\eta\in\partial\phi(s^*)$ and $y^*=\zeta$. It is not hard to verify that if $\eta^{+}={\rm prox}_{\tau^{-1}\phi^*}(\eta-\tau^{-1}\nabla_{\!\eta}\Theta_{k,j}(w))$ for a step size $\tau>0$, then $\eta^+\in\partial\phi(\widetilde{s})$ with $\widetilde{s}=\tau(\eta-\eta^{+})-\nabla_{\!\eta}\Theta_{k,j}(w)=\tau(\eta-\eta^{+})+x^*$. Apparently, when $\eta^{+}$ is sufficiently close to $\eta$, such $\widetilde{s}$ is actually a small perturbation of $x^*$, and when $w$ is an optimal solution of \eqref{dual}, the associated $(x^*,\widetilde{s}^*,y^*)$ is optimal to \eqref{EEsubprob}. When applying the PGls to solve \eqref{dual}, the generated iterate sequence $\{\eta^l\}_{l\in\mathbb{N}}$ does satisfy such a relation, i.e., $\eta^{l+1}\!={\rm prox}_{\tau_l^{-1}\phi^*}(\eta^l\!-\!\tau_l^{-1}\nabla_{\!\eta}\Theta_{k,j}(w^l))$. In view of this, at the current $w^l=(\lambda^l,\eta^l,\zeta^l)$, we set $x_{k,j}^{l}:=x^k-\frac{b^k+\mathcal{A}_{k,j}^*\zeta^l+\eta^l+\nabla\!g(x^k)\lambda^l}{\mu_{k,j}+\langle\lambda^l,L^{k,j}\rangle}$, and check whether the triple $(x_{k,j}^{l},\eta^l,\lambda^l)$ satisfies \eqref{inexact1}-\eqref{inexact2}. If the following inequalities 
\begin{subnumcases}{}\label{inexact11}
 R_{k,j}(x_{k,j}^{l},\eta^{l},\lambda^{l})\le \frac{\beta_{R}}{2}\|x^{k,l}-x^k\|^2,\nonumber\\
 \label{inexact21}
 F_{k,j}(x_{k,j}^{l})\le F_{k,j}(x^k),\,F_{k,j}(x_{k,j}^{l})-\Xi_{k,j}(w^{l})\le \frac{\beta_{F}}{2}\|x^{k,l}-x^k\|^2\nonumber
\end{subnumcases}
hold, we terminate Algorithm \ref{PGls} at the iterate $w^{l}$, and $x_{k,j}^{l}$ serves as the inexact minimizer $y^{k,j}$ of \eqref{subprob}, which along with $(\eta^{l},\lambda^{l})$ satisfies conditions \eqref{inexact1}-\eqref{inexact2}. 
\subsubsection{Setup of parameters}\label{sec6.1.3}

As mentioned in Section \ref{sec6.1.1}, we take $\mathcal{Q}_{k,j}=\mu_{k,j}\mathcal{I}+\mathcal{A}_{k,j}^*\mathcal{A}_{k,j}$. The choice of the linear operator $\mathcal{A}_{k,j}$ depends on the second-order information of $f$; see the subsequent experiments. Since $\mathcal{A}_{k,j}^*\mathcal{A}_{k,j}$ is usually positive semidefinite, the choice of $\mu_{k,j}$ has a great influence on the approximation effect of $F_{k,j}$ to $F$. Similarly, the choice of $L^{k,j}$ determines the approximation effect of $\Gamma_{k,j}$ to $\Gamma$. Our preliminary tests show that the updating rules of $L^{k,j}$ and $\mu_{k,j}$ in steps 7-8 work well with $\tau=2$, from $\mu_{k,0}$ and $L^{k,0}$ given by \eqref{BB-muk0}-\eqref{BB-Lk0} with $\mu_{\min}\!=\!10^{-16},\mu_{\max}\!=\!10^{16},L_{\min}\!=\!10^{-16},L_{\max}=10^{16}$, as long as the initial $\mu_{0,0}$ and $L^{0,0}$ are appropriately chosen. Apparently, the larger $\mu_{0,0}$ and $L^{0,0}$ make the first subproblem (so the subsequent ones) have a big gap from the original problem \eqref{prob} though its computation is relatively easy, and now iMBA-pgls will return a bad stationary point. In view of this, we choose $\mu_{0,0}=\,{\rm lip}\,\nabla\!f(x^0)$ and $L^{0,0}=0.05\,\ell_{\nabla\!g}(x^0)$, where ${\rm lip}\,\nabla\!f(x^0)$ and $\ell_{\nabla\!g}(x^0)$ can be estimated by the Barzilai-Borwein rule \cite{Barzilai88}. Of course, they can be got by estimating $\|\nabla^2\!f(x^0)\|$ and $\|\nabla^2g_i(x^0)\|$ for each $i\in[m]$ if $f$ and $g$ are twice continuously differentiable on an open set containing $x^0$. 

From the preliminary tests, we observe that $\|y^{k,j}\!-x^k\|$ is usually $O(10^{-5})$ smaller than $R_{k}(y^{k,j},v^{k,j},\lambda^{k,j})$. Therefore, we choose $\beta_{R}=10^{10}$ and $\beta_{F}=10^8$ so that the inexactness conditions \eqref{inexact1}-\eqref{inexact2} are easily satisfied. We also find that Algorithm \ref{iMBA} is robust to $\alpha\in[10^{-12},10^{-4}]$, so choose $\alpha=10^{-6}$ for the tests. To sum up, for the subsequent tests, the parameters of Algorithm \ref{iMBA} are chosen as follows: 
\begin{align*} \mu_{\min}=10^{-16},\,\mu_{\max}=10^{16},\,L_{\min}=10^{-16},\,L_{\max}=10^{16},\,\tau=2,\qquad\\
\beta_{R}\!=\!10^{10},\,\beta_{F}\!=\!10^8,\,\alpha=10^{-6},\,\mu_{0,0}={\rm lip}\,\nabla\!f(x^0),\,L^{0,0}=0.05\,\ell_{\nabla\!g}(x^0).
\end{align*}

Next we look at the choice of the parameters in Algorithm \ref{PGls}. Fix any $k,j\in\mathbb{N}$. Our preliminary tests indicate that Algorithm \ref{PGls} is very robust to the parameter $\delta$ whenever $\delta\le 10^{-2}$. However, the initial $\tau_{0,0}$ or the step-size $\tau_{l,\nu}^{-1}$ has a great impact on the efficiency of Algorithm \ref{PGls}. Apparently, a smaller initial $\tau_{0,0}$ will lead to a larger step-size. Based on this, we choose $\tau_{0,0}=10^{-8}\|\nabla\!g(x^k)\|^2$. The following parameters
\[  \delta=10^{-6},\,\varrho=10,\,\tau_{0,0}=10^{-8}\|\nabla\!g(x^k)\|^2\ \ {\rm and}\ \ l_{\rm max}=2000
\]
are used for the subsequent tests, where $l_{\rm max}$ is the maximum number of iterations. 
\subsection{Validation of theoretical results}\label{sec5.2}

For each $k$ with $x^k$ being a non-stationary point of problem \eqref{prob}, by Lemma \ref{lemma-welldef} (ii) or Lemma \ref{inner-step}, the inner loop of Algorithm \ref{iMBA} necessarily stops within a finite number of steps. Figure \ref{fig1} (a) plots the number of steps required by the inner loop in each iteration. We see that the number of steps for the inner loop is basically not more than $\textbf{3}$, and attains $\textbf{3}$ only at several iterations of the outer loop. This validates the results obtained in Lemma \ref{lemma-welldef} (ii) and Lemma \ref{inner-step}. Fix any $k\in\mathbb{N}$. When the inner loop stops within a finite number of steps, it is natural to ask how many iterations does Algorithm \ref{PGls} require to solve every subproblem. Figure \ref{fig1} (b) plots the maximum number of iterations required by Algorithm \ref{PGls} for solving the total $j_k+1$ subproblems. We see that as $k$ increases, the maximum number of iterations for solving the total $j_k+1$ subproblems gradually becomes larger, and is generally not more than $\textbf{40}$. This indicates that for this class of nonconvex and nonsmooth test instances, the iterates satisfy the inexactness conditions well, which validates the result of Lemma \ref{lemma-welldef} (i).  
\begin{figure}[h]
 \centering
 \subfigure[\label{fig1a}]{\includegraphics[scale=0.43]{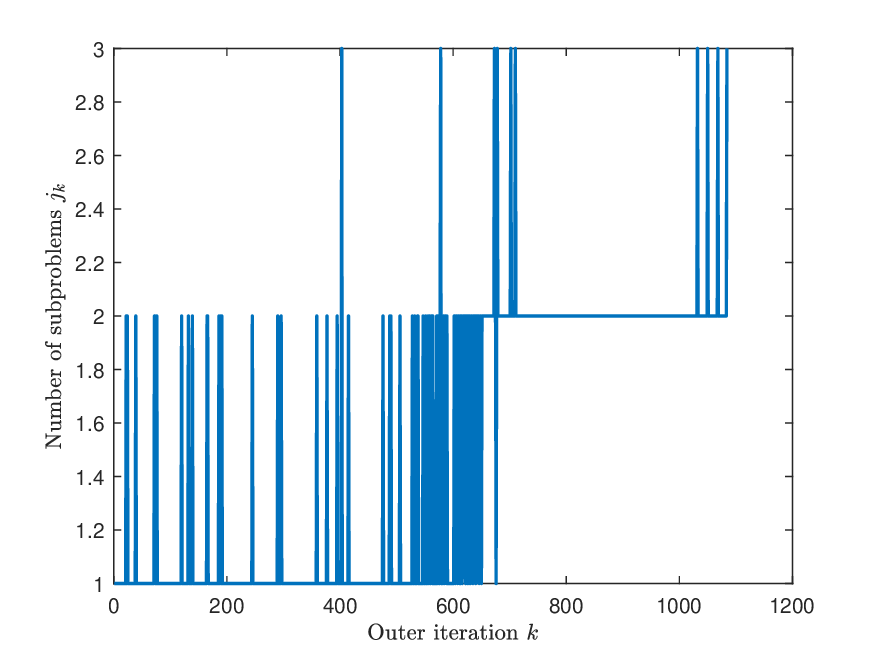}}
 \subfigure[\label{fig1b}]{\includegraphics[scale=0.43]{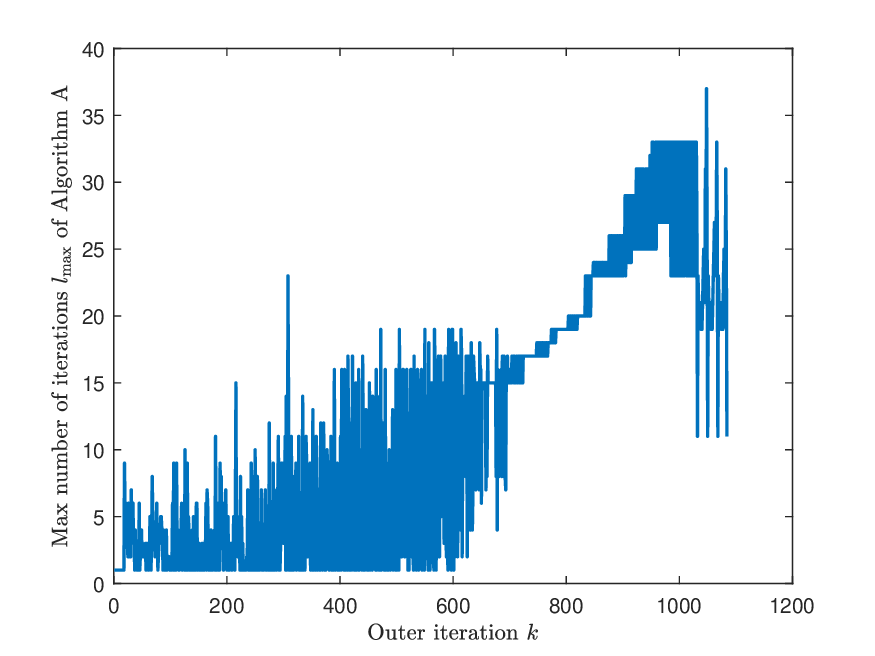}} 
 \setlength{\abovecaptionskip}{2pt}
 \setlength{\belowcaptionskip}{2pt}
 \caption{Performance of inner loop of Algorithm \ref{iMBA} for the test problem with $(n,m)=(1000,100)$ and $\omega_0=10^3$ from Section \ref{sec6.3.2}}
 \label{fig1}
 \end{figure}

 Figure \ref{fig_rate} plots the convergence behavior of the iterate sequence of Algorithm \ref{iMBA}. By Theorem \ref{globalconv}, the iterate sequence is convergent under Assumptions \ref{ass2}-\ref{ass4} and the KL property of $\Phi$. The function $\Phi$ corresponding to $F$ with $f,\phi$ and $g$ from Section \ref{sec6.3.2} is semialgebraic, so is a KL function of exponent $q$ for some $q\in(0,1)$. Since $f$ is a convex quadratic function, $\phi$ is the $\ell_1$-norm and $\psi\equiv 0$, the function $F$ is coercive, so Assumption \ref{ass3} holds. While Assumption \ref{ass4} automatically holds if the constraint $g(x)\in\mathbb{R}_{-}^m$ satisfies the MFCQ at every feasible point. Since it is highly possible for the latter to hold when $n\approx m$ or $n\gg m$, the iterate sequences of Algorithm \ref{iMBA} for those problems with $n\approx m$ or $n\gg m$ will have better convergence. The curves of Figure \ref{fig_rate} indeed validate this result, and also demonstrate the linear convergence of the iterate sequences. Note that  $\Phi(z)=h(x,G(x,s,L),\xi)$ with $h$ being a piecewise linear-quadratic function for this group of test instances. Such $h$ satisfies the requirement of Proposition \ref{prop-KLexp}. Thus, the linear convergence behavior of the curves in Figure \ref{fig_rate} means that the condition \eqref{cond-KL} also holds for  this group of test instances.     
\begin{figure}[h]
\centering
\subfigure[\label{fig1a}]{\includegraphics[scale=0.44]{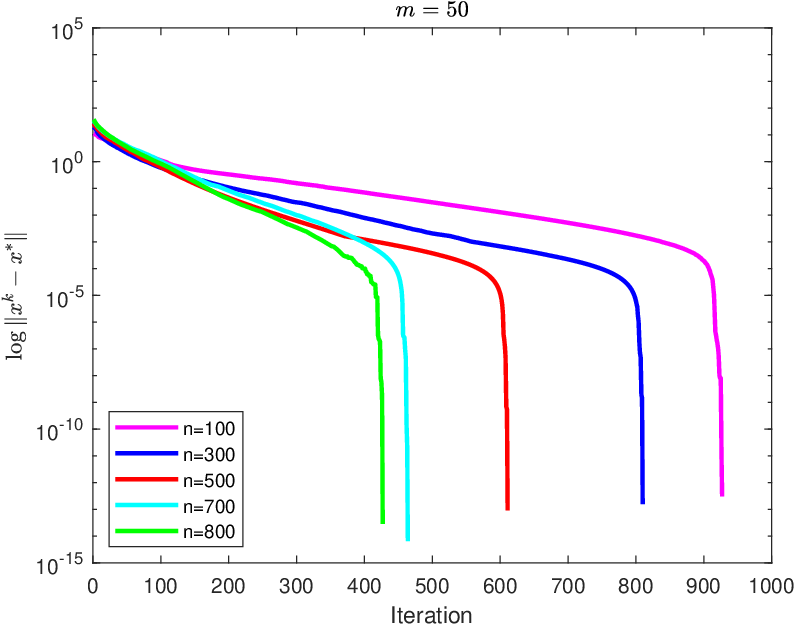}}
\subfigure[\label{fig1b}]{\includegraphics[scale=0.43]{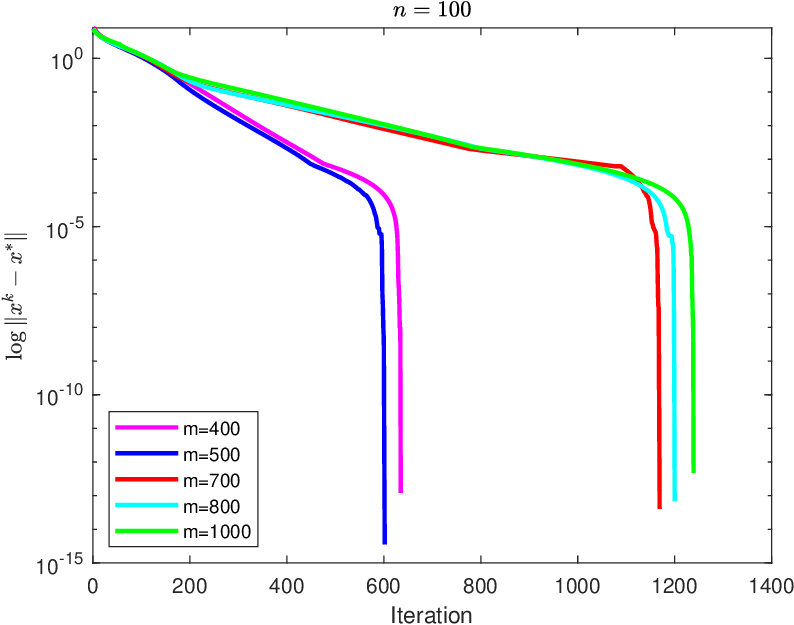}}
\setlength{\abovecaptionskip}{2pt}
\setlength{\belowcaptionskip}{2pt}
\caption{Convergence behavior of the iterate sequences generated by Algorithm \ref{iMBA} for solving the test problems from Section \ref{sec6.3.2} with $\psi\equiv 0$ and $\omega^0=10^4$}
\label{fig_rate}
\end{figure}
\subsection{Numerical comparisons for QCDC problems}\label{sec5.3}
 
We shall test the performance of iMBA-pgls for solving quadratically constrained DC (QCDC) problems. The constraint functions of QCDC problems take the form of 
\begin{equation}\label{mapg}
 g_i(x):=x^{\top}Q_ix-x^{\top}P_ix+2\langle b_i,x\rangle +c_i\quad\forall i\in[m],
\end{equation}
where $Q_1,\ldots,Q_m$ are the $n\times n$ positive definite matrices, $P_1,\ldots,P_m$ are the $n\times n$ positive semidefinite matrices, $b_1,\ldots,b_m\in\mathbb{R}^n$, and $c_1,\ldots,c_m\in\mathbb{R}$. We follow the same way as in \cite{Auslender10} to generate the $m$ positive definite matrices $Q_i$. That is, every $Q_i$ is of the form $Q_i=Y_iD_iY_i$, where $D_i$ is an $n\times n$ diagonal matrix with diagonal elements shuffled randomly from the set $\{10^{\frac{10(j-1)}{n-1}}\,|\,j\in[n]\}$, and $Y_i=I-2\frac{y_iy_i^{\top}}{\|y_i\|^2}$ is an $n\times n$ random Householder orthogonal matrix with $I$ being the $n\times n$ identity matrix, and the components of $y_i\in\mathbb{R}^n$ chosen randomly from $(-1,1)$. Clearly, the spectral norm $\|Q_i\|$ of every $Q_i$ equals $10^{10}$. The $m$ positive semidefinite $P_i$ are specified later. To generate the data $b_i$ and $c_i$ for $i\in[m]$ such that the feasible set $\Gamma=g^{-1}(\mathbb{R}_{-}^m)$ is nonempty, we consider an equivalent reformulation of $g(x)\le 0$. By the expression of $Q_i$, we have $Q_i=B_i^{\top}B_i$ with $B_i=D_i^{1/2}Y_i$ being nonsingular. Then, for each $i\in[m]$, with $h_i=(B_i^{-1})^{\top}b_i$ and $d_i^2=\|h_i\|^2-c_i$, the constraint $g_i(x)\le 0$ is reformulated as 
\[
 g_i(x)=\|B_ix+h_i\|^2-x^{\top}P_ix-d_i^2\le 0\quad\forall i\in[m].
\]
 Let $x^0\in\mathbb{R}^n$ be a vector generated randomly. To ensure that $x^0\in\Gamma$, for each $i\in[m]$, we randomly choose the components of $h_i$ from $(-1,1)$ and $s_i$ from $[0,1)$, and then set $d_i^2=\|B_ix^0+h_i\|^2-x^{\top}P_ix+s_i$. Clearly, $\|B_ix^0+h_i\|^2-x^{\top}P_ix-d_i^2\le 0$ due to  the nonnegativity of $s_i$, and consequently $x^0\in\Gamma$. For each $i\in[m]$, after generating $h_i$ and $d_i$, we set $b_i=B_i^{\top}h_i$ and $c_i=\|h_i\|^2-d_i^2$. 
\subsubsection{$\ell_1$-norm regularized convex QCQPs}\label{sec6.3.1}

This class of problems takes the form of problem \eqref{prob} where every component of $g$ is given by \eqref{mapg} with $P_1=\cdots=P_m=0$, and the functions $f,\phi$ and $\psi$ are specified as 
\[
 f(x)\!:=\|Y_0x\|^2+2\omega_{0}(b_0/\|b_0\|)^{\top}x,\ \phi(x)\!:=0.01\|x\|_1\ \ {\rm and}\ \ \psi\equiv 0. 
 \]
 Here, $\omega_0$ is a constant, $Y_0\in\mathbb{R}^{p\times n}$ with $p=\lfloor n/2\rfloor$ is a matrix, and $b_0\in\mathbb{R}^n$ is a vector. We choose $\omega_0=10$ and $\omega_0=10^4$ to generate two classes of problems whose objective values have different scales. The entries of $Y_0$ and $b_0$ are generated randomly to obey the standard normal distribution. For each $k,j\in\mathbb{N}$, we choose the linear mapping $\mathcal{A}_{k,j}\!:\mathbb{R}^n\to\mathbb{R}^p$ to be $\mathcal{A}_{k,j}x=Y_0x$ for $x\in\mathbb{R}^n$, so the linear operator $\mathcal{Q}_{k,j}$ has the expression $\mathcal{Q}_{k,j}x=\mu_{k,j}x+Y_0^{\top}Y_0x$ for $x\in\mathbb{R}^n$. 

 We set the starting point for iMBA-pgls as $x^0$, which is precisely the one used to generate $c_i$ in \eqref{mapg}.
 We compare the performance of iMBA-pgls with that of the off-the-shelf software MOSEK \cite{MOSEK}, which is recognized as the leader of the commercial software for nonlinear convex optimization. Since applying MOSEK \cite{MOSEK} to solve the QCQP reformulation of this class of problems returns bad solutions, we compare the performance of iMBA-pgls with that of MOSEK \cite{MOSEK} for solving its quadratic conic reformulation. We adopt the default setting of MOSEK for our tests. 

Table \ref{convex-QCQPs} reports the results of iMBA-pgls and MOSEK with $\epsilon=10^{-5}$ and $\epsilon_1=10^{-7}$, where the ``compl'' column lists the violation of the complemetarity condition, i.e., the value of $[-\langle\lambda^k,g(x^k)\rangle]_{+}$. We see that iMBA-pgls has a different performance for the instances with $\omega_0=10^4$ and $\omega_0=10$, and for those instances with $\omega_0=10^4$, it requires less iterations but needs more running time when $n$ is fixed to be $100$. When comparing with MOSEK, iMBA-pgls produces better objective values for all instances with $\omega_0=10^4$ and \textbf{7} instances with $\omega_0=10$. When $m$ is fixed to be $100$, for the instances with $\omega_0=10^4$, as the dimension $n$ increases, the running time of iMBA-pgls increases slowly and $120$s is enough for $n=2000$, whereas MOSEK meets out of memory once $n>1000$. When $n$ is fixed to be $100$, the running time of iMBA-pgls is more than that of MOSEK for $m\le 1000$ but much less than that of MOSEK for $m\ge 2000$; while for those instances with $\omega_0=10$, its running time is much less than that of MOSEK when $m\ge 1000$. Since MOSEK solves the quadratic conic reformulation of \eqref{prob}, its running time increases rapidly as $m$ or $n$ becomes larger. Thus, we conclude that iMBA-pgls is superior to MOSEK for this class of convex QCQPs with a larger $m$ or $n$ in terms of the accuracy of objective values and the running time.  
\begin{table}[h]
 \centering
 \setlength{\tabcolsep}{6.0pt}
 \setlength{\belowcaptionskip}{2pt}
 \caption{Numerical comparisons with MOSEK on $\ell_1$-norm regularized convex QCQPs}
 \label{convex-QCQPs}
 \begin{tabular}{@{\extracolsep{\fill}}cccccccccc@{\extracolsep{\fill}}}
  \toprule
  \multicolumn{1}{c}{}&\multicolumn{1}{c}{}& \multicolumn{1}{c}{}  & \multicolumn{4}{c}{iMBA-pgls} & \multicolumn{3}{c}{MOSEK} \\
 \midrule
 $\omega_0$&$n$ & $m$ &  iter & Fval &  time(s) & compl &  iter & Fval & time(s) \\
 \midrule
\multirow{10}{*}{$10^4$} &100 &\multirow{5}{*}{100} & 802 &	 {\bf-1.4229e+5} & 8.19  &  1.1978e-5
 & 13  &  -1.4221e+5 &  1.86  \\ 
 &200 &\multirow{5}{*}{}& 549  & {\bf-2.1466e+5} &  9.78  &  9.4857e-6
 &15  &  -2.1404e+5 &  5.15\\
 &500 &\multirow{5}{*}{}& 396 &  {\bf-3.3824e+5} &  10.40  &  3.7922e-6
 &15  &  -3.3146e+5 &  34.01\\
 &1000 &\multirow{5}{*}{}& 322 &  {\bf -4.6975e+5} &  17.08  &  1.0804e-4
 & 13  &  -4.5215e+5 &  160.20 \\ 
 &2000 &\multirow{5}{*}{} & 431 &  {\bf -6.5727e+5} & 112.40  &  1.3539e-6
 & -- & -- & --\\ 	
 \cmidrule{2-10}
 &\multirow{5}{*}{100} & 200 & 482  &  {\bf-1.3288e+5} & 10.94  &  9.4283e-6
 & 14  &  -1.3284e+5 &	  5.26\\ 
 &\multirow{5}{*}{} &500 & 451 &	{\bf -1.2583e+5} &	 21.89 &  7.3402e-6
 &15  &  -1.2579e+5 &	  21.54\\
 &\multirow{5}{*}{} &1000 & 805 & {\bf-1.2414e+5} &	  120.37  &  8.2490e-6
 &17  &  -1.2409e+5 &	  77.07\\
 &\multirow{5}{*}{} &2000 & 618 &  {\bf-1.2116e+5} & 76.97  	& 5.6348e-5
 &19  & -1.2113e+5 &  290.08\\ 
 &\multirow{5}{*}{} &3000 & 819 &	{\bf-1.2059e+5} &  215.46  &  7.6449e-6
 &19  &  -1.2055e+5 &  638.81\\
 \midrule
 \multirow{10}{*}{$10$} &100 &\multirow{5}{*}{100} 
 & 1021& {\bf -1.0124e+2}&  	  2.77  	&  1.9749e-8 
 & 14  &  -1.0012e+2 	&  2.22\\
 & 200 &\multirow{5}{*}{} & 2775  & {\bf-1.5458e+2} &  12.55 &	  9.7504e-8
 & 14  &  -1.5398e+2 &  6.28\\
 & 500 &\multirow{5}{*}{} & 7910 & -2.6156e+2 &  104.87 &  9.8635e-8
 & 15  &  -2.6156e+2 &  35.56\\
 & 1000 &\multirow{5}{*}{} & 10000  &  -3.8260e+2 & 476.22  &  1.8662e-5
 & 14  &  -3.9238e+2 &  158.05\\
 & 2000 &\multirow{5}{*}{} & 500  &  -2.5967e+2 & 83.22  &  0
 & -- & -- & --\\
 \cmidrule{2-10}
 &\multirow{5}{*}{100} & 200 & 1402  & {\bf-8.9678e+1} &  5.82  &  9.6931e-8
 & 17  &  -8.8272e+1 &  4.78\\
 &\multirow{5}{*}{} & 500 & 1472  &  {\bf-8.0971e+1} &  10.84  & 9.8957e-8
 & 18  &  -7.9667e+1 &  21.97\\
 &\multirow{5}{*}{} & 1000 & 1560  &  {\bf-7.6383e+1} &  18.18  &  8.5374e-8
 & 17  &  -7.4931e+1 &  76.99\\
 &\multirow{5}{*}{} & 2000 & 2028  &  {\bf-6.9809e+1} &  43.48  &  8.5874e-8
 & 16 &  -6.8305e+1 &	  290.09\\
 &\multirow{5}{*}{} & 3000 & 1856  & {\bf -6.9285e+1} 	&  58.67  &	  9.0017e-8
 & 17  &  -6.7784e+1 &  638.87\\
 \botrule
\end{tabular}
{\it\small where ``$-$'' signifies out of memory.}
\end{table}
\subsubsection{$\ell_1\!-\!\ell_2$ regularized QCQPs}\label{sec6.3.2}

We test the performance of iMBA-pgls for solving $\ell_1-\ell_2$ regularized QCQPs, which have the same form as the problem in Section \ref{sec6.3.1} except that $\psi(\cdot)=0.01\|\cdot\|$ and every component of $g$ is given by \eqref{mapg} with $P_i=10^{5}I$ for $i\in[m]$. The DC regularizer aims at testing the efficiency of iMBA-pgls rather than seeking a sparse solution. For each $k,j\in\mathbb{N}$, we choose the same linear mapping $\mathcal{A}_{k,j}$ as in Section \ref{sec6.3.1}. 
\begin{algorithm}
\renewcommand{\thealgorithm}{2}
 \caption{\label{DCA}{\bf (DCA for solving problem \eqref{prob})}}
 \begin{algorithmic}[1]
 \State{\textbf{Input:}\ $x^0\in\Gamma$.}
 \For{$k=0,1,2,\ldots$}
 
 \State{Choose $\xi^k\in\partial\psi(x^k)$;}
 
 \State{Seek an optimal solution $x^{k+1}$ of the convex quadratic program
       \begin{align*}
        &\min_{x\in\mathbb{R}^n} f(x)+\phi(x)-\langle\xi^k,x\rangle-\psi(x^k)\\
        &\ \ {\rm s.t.}\ x^{\top}Q_ix-2\langle P_ix^k,x-x^k\rangle+2b_i^{\top}x+c_i-(x^k)^{\top}P_ix^k\le 0\quad\forall i\in[m].
       \end{align*}}
\EndFor		
\end{algorithmic}
\end{algorithm}

We compare the performance of iMBA-pgls with that of Algorithm \ref{DCA} above, a DCA for solving \eqref{prob}. For Algorithm \ref{DCA}, there is no full convergence guarantee for its iterate sequence, and we use it just for numerical comparisons. We apply MOSEK to solve the quadratic conic reformulations of the subproblems of Algorithm \ref{DCA}, and now Algorithm \ref{DCA} is called DCA-MOSEK. During the tests, we adopt the default setting of MOSEK and we terminate DCA-MOSEK at the $k$th iteration whenever $\|x^k-x^{k-1}\|\le\epsilon$. For a fair comparison, the two solvers starts from the same feasible point $x^0$.

Table \ref{L12-QCQPs} reports the results of iMBA-pgls and DCA-MOSEK with $\epsilon=10^{-5}$ and $\epsilon_1=10^{-7}$. Since MOSEK solves the quadratic conic reformulations of the subproblems with an interior point solver, DCA-MOSEK starting from $x^0\in\Gamma$ always returns a feasible solution to this class of problems. We see that for this class of nonconvex QCQPs, iMBA-pgls has a similar performance to the $\ell_1$-norm regularized convex QCQPs, and requires less iterations for the test instances with $\omega_0=10^4$ but more running time when $n$ is fixed. iMBA-pgls produces better objective values than DCA-MOSEK for all test instances, and requires much less running time than the latter for $n\ge 1000$ when $m$ is fixed, and for $m\ge 1000$ when $n$ is fixed. This demonstrates the advantage of iMBA-pgls for solving this class of nonconvex QCQPs with a larger $m$ or $n$.

\begin{table}[h]
 \centering
 \setlength{\tabcolsep}{6.0pt}
 \setlength{\belowcaptionskip}{2pt}
 \caption{Numerical comparisons with DCA-MOSEK on $\ell_1\!-\!\ell_2$  regularized QCQPs}
 \label{L12-QCQPs}
 \begin{tabular}{@{\extracolsep{\fill}}cccccccccc@{\extracolsep{\fill}}}
 \toprule
  \multicolumn{1}{c}{}&\multicolumn{1}{c}{}& \multicolumn{1}{c}{}  & \multicolumn{4}{c}{iMBA-pgls} & \multicolumn{3}{c}{DCA-MOSEK} \\
 \midrule
 $\omega_0$ &$n$ & $m$ &  iter & Fval &  time(s) & compl &  iter & Fval & time(s) \\
 \midrule
 \multirow{10}{*}{$10^4$} & 100 &\multirow{5}{*}{100} 
 & 550.0 & 	  {\bf-1.5763e+5} &	   6.34  &	  1.0389e-5
 & 3.0  	&  -1.5758e+5 	  & 3.61\\ 
 & 200 &\multirow{5}{*}{} 
 & 526.0 & 	  {\bf-2.3142e+5} &	   5.61  &	  1.4801e-5 
 & 6.0  	&  -2.3055e+5 	  & 15.38\\
 & 500 & \multirow{5}{*}{} 
 & 325.0 & 	  {\bf-3.6575e+5} &	   5.63  &	  5.6332e-4 
 & 3.0  	&  -3.5813e+5 	  & 63.71\\
 & 1000 &\multirow{5}{*}{} 
 & 478.0 & 	  {\bf-4.8139e+5} &	  228.42  &	  5.0061e-6  
 & - & -&-\\
 & 2000 &\multirow{5}{*}{} 
 & 405.0  &	  {\bf-6.5331e+5} &	  106.41  &	  1.7908e-6
 & - & - & - \\ 
 \cmidrule{2-10}
 & \multirow{5}{*}{100} & 200 
 & 558.0 & 	  {\bf-1.4415e+5} &	   7.88  &	  6.2327e-6
 & 3.0  	&  -1.4410e+5 	  & 7.90\\
 & \multirow{5}{*}{} &500 
 & 1095.0 & 	  {\bf-1.3054e+5} &	  39.26  &	  9.6592e-6
 & 3.0  	 & -1.3049e+5 	  & 30.81\\
 & \multirow{5}{*}{} &1000 
 & 1130.0  &	  {\bf-1.2770e+5} &	  55.17  &	  1.2426e-5
 & 4.0  	  &-1.2765e+5 	  & 106.89\\
 & \multirow{5}{*}{} &2000 
 & 1003.0  &	  {\bf-1.2441e+5} &	  108.43  &	  7.1233e-6
 & 4.0  	  &-1.2436e+5 	  & 351.72\\ 
 & \multirow{5}{*}{} &3000 
 & 930.0  &	  {\bf-1.2340e+5} &	  218.66  &	  6.6391e-6
 & 4.0  	 & -1.2336e+5 	  & 737.58\\
 \midrule
 \multirow{10}{*}{10} & 100 &\multirow{5}{*}{100} 
 & 1275.0 & 	  {\bf -1.0960e+2} 	&   2.89  &	  9.0792e-8
 & 4.0  	 & -1.0837e+2 	&   4.40\\
 & 200 &\multirow{5}{*}{ } 
 & 3068.0 & 	  {\bf -1.6670e+2} 	&  13.94  &	  9.7689e-8
 & 4.0  	 & -1.6627e+2 	&  11.21\\
 & 500 &\multirow{5}{*}{ } 
 & 9153.0 & 	  {\bf-2.8080e+2} 	&  127.49  &	  9.7682e-8
 & 4.0  	 & -2.8078e+2 	&  80.87\\
 & 1000 &\multirow{5}{*}{ } 
 & 10000.0 &	  {\bf-3.9194e+2} 	&  491.83  	&  8.9557e-6
 & - & - & -\\
 & 2000 &\multirow{5}{*}{ } 
 & 500.0  &	  {\bf -2.5515e+2} 	& 134.51  &  0
 & - & - & -\\
 \cmidrule{2-10}
 & \multirow{5}{*}{100} & 200 
 & 1255.0 & 	  {\bf-1.0171e+2} &	   6.02  &	  7.6043e-8
 & 3.0  	 & -1.0052e+2 	  & 8.37\\
 & \multirow{5}{*}{} & 500 
 & 1272.0 & 	  {\bf-8.6065e+1} &	   8.05  &	  9.0931e-8
 & 4.0  	 & -8.4153e+1 	  & 35.14\\
 & \multirow{5}{*}{} & 1000 
 & 1130.0 & 	  {\bf-8.5009e+1} &	  13.32  &	  9.4000e-8
 & 4.0  	 & -8.3492e+1 	  & 107.11\\
 & \multirow{5}{*}{} & 2000
 & 1693.0 & 	  {\bf -7.8086e+1} &  44.49  &	  8.7478e-8
 & 4.0  	 & -7.6325e+1 	  & 356.70\\
 & \multirow{5}{*}{} & 3000 
 & 1699.0 & 	  {\bf-7.7462e+1} &  52.38  &	  5.4841e-8
 & 6.0  	 & -7.5910e+1 	  & 805.80\\
 \botrule
\end{tabular}
\end{table}
\subsubsection{$\ell_1\!-\!\ell_2$ regularized Student's $t$-regression with constraints}

This class of problem takes the form of \eqref{prob} with every component of $g$ given by \eqref{mapg}, $\phi(\cdot)\!=0.01\|\cdot\|_1,\psi(\cdot)\!=0.01\|\cdot\|$, but a highly nonlinear $f$. We consider the nonconvex loss function $f$ introduced in \cite{Aravkin12} to deal with the data contaminated by heavy-tailed Student-$t$ noise, which has the form $f(x)=\theta(Ax-b)$ with $\theta(u):=\sum_{i=1}^N\log[1\!+\!4u_i^2]$. The data $A\in\mathbb{R}^{N\times n}$ and $b\in\mathbb{R}^{N}$ are randomly generated in the same way as in \cite{Milzarek14}. Specifically, we generate a true sparse signal $x^{\rm true}$ of length $n$ with $s=\lfloor\frac{n}{40}\rfloor$ nonzero entries whose indices are chosen randomly, and then compute every nonzero entry in terms of $x^{\rm true}_i=\eta_1(i)10^{4\eta_2(i)}$, where $\eta_1(i)\in\{\pm1\}$ is a random sign and $\eta_2(i)$ is uniformly distributed in $[0,1]$. The matrix $A\in\mathbb{R}^{N\times n}$ takes $N=\lfloor{n}/{8}\rfloor$ random cosine measurements, i.e., $Ax = (\verb"dct"(x))_J$, where $\verb"dct"$ is the discrete cosine transform and $J\subset[n]$ with $|J|=N$ is an index set chosen randomly. The vector $b$ is obtained by adding Student's $t$-noise with degree of freedom $4$ and rescaled by $0.1$ to $Ax^{\rm true}$. 

Note that $\nabla^2f(x)=A^{\top}\nabla^2\theta(Ax-b)A$ for $x\in\mathbb{R}^n$. When testing this group of instances, for each $k,j\in\mathbb{N}$, we take $\mathcal{A}_{k,j}={\rm diag}([\omega^k]_{+}^{1/2})A$, where $\omega^k\in\mathbb{R}^N$ is the diagonal vector of the diagonal matrix $\nabla^2\theta(Ax^k\!-\!b)$. Since MOSEK can be employed to solve the quadratic conic reformulation of \eqref{subprobkj}, we compare the performance of iMBA-pgls with that of Algorithm \ref{iMBA} armed with MOSEK to solve its subproblems (iMBA-MOSEK). When applying MOSEK to solve the quadratic conic reformulation of \eqref{subprobkj}, we set the stop tolerances for the dual feasibility, primal feasibility and relative gap of its interior-point solver to be $10^{-5}$, and the other parameters are all set to be the default ones. For the fairness of comparison, iMBA-MOSEK starts from the same feasible point $x^0$, and stops under the same termination condition as for iMBA-pgls. 

Table \ref{table_dct} reports the results of iMBA-pgls and iMBA-MOSEK with $\epsilon=10^{-5}$ and $\epsilon_1=10^{-7}$. Observe that iMBA-MOSEK returns better objective values for more test examples. This indicates that, when $f$ is highly nonlinear, a powerful solver to the subproblems of Algorithm \ref{iMBA} will improve its performance. Unfortunately, the running time of iMBA-MOSEK is much more than that of iMBA-pgls, and the former is at least \textbf{7} times the latter when $m$ is fixed to be $100$. Thus, when $f$ is highly nonlinear, seeking a more robust and faster solver than PGls to the subproblems is still requisite.   
\begin{table}[h]
 \centering
 \setlength{\tabcolsep}{2.0pt}
 \setlength{\belowcaptionskip}{2.0pt}
 \caption{Numerical comparisons on $\ell_1\!-\!\ell_2$ regularized Student's $t$-regressions with QCs}
 \label{table_dct}
 \begin{tabular}{@{\extracolsep{\fill}}ccccccccccc@{\extracolsep{\fill}}}
  \toprule
  \multicolumn{1}{c}{}&\multicolumn{1}{c}{}& \multicolumn{1}{c}{}  & \multicolumn{4}{c}{iMBA-pgls} & \multicolumn{4}{c}{iMBA-MOSEK} \\
 \midrule
 $P_i$ & $n$ & $m$ &  iter & Fval &  time(s) & compl &  iter & Fval & time(s) & compl\\
 \midrule
 \multirow{8}{*}{$0$} & 300 &\multirow{4}{*}{50} 
 & 530  &	  3.9331e+2 	&  2.61  &	  6.8446e-8
 & 147  &	  {\bf3.9330e+2} 	&  9.78  &	  1.3974e-7\\
 \multirow{8}{*}{} & 500 &\multirow{4}{*}{}
 & 517  &      7.3941e+2 	&  5.60  &	  3.0758e-10
 & 500  &	  {\bf7.3896e+2} 	&  58.10 & 	  2.2884e-8\\
 \multirow{8}{*}{} & 800 &\multirow{4}{*}{}
 & 556  &	  {\bf1.2668e+3} 	&  11.07 & 	  4.8093e-9
 & 500  &	  1.2686e+3 	&  88.82 & 	  2.2294e-8\\
 \multirow{8}{*}{} & 1000 &\multirow{4}{*}{}
 & 669  &	  1.5210e+3 	&  18.76 & 	  2.0649e-8
 & 231  &	  {\bf1.5201e+3} 	&  59.01 & 	  2.0264e-7\\ \cmidrule{2-11}
 \multirow{8}{*}{} & 300 &\multirow{4}{*}{100} 
 & 511  &	  3.9384e+2 	&  4.09  &	  7.4481e-9
 & 507  &	  {\bf3.9357e+2} 	&  694.28 &	  6.6532e-8\\ 
 \multirow{8}{*}{} & 500 &\multirow{4}{*}{}
 & 518  &	  7.3987e+2 	&  8.58  &	  1.9922e-10
 & 500  &	  {\bf7.3940e+2} 	&  121.13 &	  1.6693e-8\\
 \multirow{8}{*}{} & 800 &\multirow{4}{*}{}
 & 737  &	  {\bf1.2676e+3} 	&  23.22  &	  6.0658e-8
 & 500  &	  1.2688e+3 	&  181.60 &	  1.2504e-9\\
 \multirow{8}{*}{} & 1000 &\multirow{4}{*}{}
 & 636  &	  1.5225e+3 	&  31.34  &	  0
 & 500  &	  {\bf1.5212e+3} 	&  1550.29 & 	  1.8061e-8\\
 \midrule
 \multirow{8}{*}{$10^{5}I$} & 300 &\multirow{4}{*}{50} 
 & 168  &	  4.6455e+2 	&  6.76  &	  4.3972e-10
 & 500  &	  {\bf4.6409e+2} 	&  30.48  &	  3.9712e-9\\
 \multirow{8}{*}{} & 500 &\multirow{4}{*}{}
 & 960  &	  {\bf6.8314e+2} 	&  9.73  &	  3.3109e-8
 & 500  &	  6.8379e+2 	&  58.64  &	  3.8367e-8\\
 \multirow{8}{*}{} & 800 &\multirow{4}{*}{}
 & 519  &	  1.0528e+3 	&  9.66  &	  1.9724e-8
 & 3000 &	  {\bf1.0522e+3} 	&  597.96 &	  9.4486e-7\\
 \multirow{8}{*}{} & 1000 &\multirow{4}{*}{}
 & 818  &	  1.3984e+3 	&  21.31  &	  9.4625e-8
 & 3000 &	  {\bf1.3977e+3} 	&  769.32 &	  3.3973e-7\\ \cmidrule{2-11}
 \multirow{8}{*}{} & 300 &\multirow{4}{*}{100} 
 & 121  &	  4.6471e+2 	&  6.29  &	  1.2036e-10
 & 500  &	  {\bf4.6420e+2} 	&  59.00 & 	  1.5984e-9	\\ 
 \multirow{8}{*}{} & 500 &\multirow{4}{*}{}
 & 822  &	  6.8489e+2 	&  12.80  &	  7.2792e-8
 & 500  &	  {\bf6.8427e+2} 	&  734.79  &  6.9279e-9	\\
 \multirow{8}{*}{} & 800 &\multirow{4}{*}{}
 & 556  &	  1.0580e+3 	&  17.42  &	  1.8459e-8
 & 530  &	  {\bf1.0575e+3} 	&  1198.25 & 	  9.8133e-8	\\
 \multirow{8}{*}{} & 1000 &\multirow{4}{*}{}
 & 547  &	  1.4043e+3 	&  26.17  &	  6.3158e-10
 & 894  &	  {\bf1.4036e+3} 	&  2689.58 & 	  9.9363e-8\\
 \botrule
\end{tabular}
\end{table}
\section{Conclusion}\label{sec6.0}

We developed an inexact MBA method with an implementable inexactness criterion for the constrained DC problem \eqref{prob}, thereby resolving the dilemma that the current MBA methods and their variants still lack of practical inexact criteria for the solving of subproblems. We conducted the systematic convergence analysis for the proposed inexact MBA method under the MSCQ and the BMP, and achieved the full convergence of the iterate sequence under the KL property of $\Phi_{\widetilde{c}}$, and the convergence rate of the iterate and objective value sequences under the KL property of $\Phi$ with exponent $q\in[1/2,1)$. As far as we know, this is the first full convergence analysis of the iterate sequence for the MBA methods without the MFCQ for the constraint system $g(x)\in\mathbb{R}_{-}^m$. Moreover, a checkable condition, weaker than the one in \cite[Theorem 3.2]{LiPong18}, was provided to identify whether $\Phi$ has the KL property of exponent $q\in[1/2,1)$. We also established the iteration complexity of $O(1/\epsilon^2)$ for seeking an $\epsilon$-KKT point under the common MFCQ, which has the same order as the existing exact MBA method. 

The inexact MBA method was implemented by applying PGls to solve the dual of subproblems, and numerical comparisons were conducted with the software MOSEK for $\ell_1$-norm regularized convex QCQPs and with the DCA-MOSEK for $\ell_1\!-\!\ell_2$ regularized QCQPs, which demonstrate the remarkable superiority of iMBA-pgls in the quality of solutions and running time for large $m$ or $n$. Numerical tests for $\ell_1\!-\!\ell_2$ regularized Student's $t$-regression with constraints showed that for a highly nonlinear $f$, a powerful solver to subproblems is still requisite, and we leave this for a future topic.

\section*{Declarations}




\begin{itemize}
\item {\bf Funding} This work was funded by the National Natural Science Foundation of China under project 12371299.
\item {\bf Competing Interests} The authors declare that they have no conflict of interest.
\item {\bf Data availability} All data for numerical tests are generated randomly. 
\end{itemize}


\begin{appendices}
\section{Number of steps for inner-loop}\label{secA}
\begin{alemma}\label{inner-step}
 Suppose that $\nabla\!f$ is strictly continuous on $\mathbb{R}^n$, and that $x^k$ is a non-stationary point of \eqref{prob}. Let
 \(
  X_k\!:=\big\{x\in\mathbb{R}^n\,|\,\frac{\mu_{\rm min}}{2}\|x-x^k\|^2+\langle\nabla\!f(x^k)+\xi^k,x-x^k\rangle+\phi(x)\le\phi(x^k)\big\}. 
 \)
 Then the inner loop stops once $j>\lceil\max\limits_{i\in[m]}\frac{\log[(\beta_{R}+L_{\nabla\!g_i}^k)/L_i^{k,0}]}{\log\tau}\rceil+\frac{\log[(L_{\nabla\!f}^k+\alpha)/\mu_{k,0}]}{\log\tau}$, where $L_{\nabla\!f}^k$ and $L_{\nabla \!g_i}^k$ for $i\in[m]$ are the Lipschitz constant of $\nabla\!f$ and $\nabla g_i$ on $X_k$. 
\end{alemma}
\begin{proof}
 Since $\phi$ is a finite convex function, there exists $\widehat{x}\in\mathbb{R}^n$ and $\widehat{\zeta}\in\partial\phi(\widehat{x})$ such that $\phi(x)\ge\phi(\widehat{x})+\langle\widehat{\zeta},x-\widehat{x}\rangle$ for all $x\in\mathbb{R}^n$. Then, $\mathbb{R}^n\ni x\mapsto\frac{\mu_{\rm min}}{2}\|x-x^k\|^2+\langle\nabla\!f(x^k)+\xi^k,x-x^k\rangle+\phi(x)$ is coercive, which implies that the set $X_k$ is compact. Along with the strict continuity of $\nabla\!f$ and Assumption \ref{ass1} (i), it follows that $\nabla\!f$ and $\nabla\!g_i$ for $i\in[m]$ are Lipschitz continuous on $X_k$. For each $j\in\mathbb{N}$, combining the expression of $R_{k,j}$ with inequality \eqref{inexact1} yields that 
 \begin{equation}\label{quant-1}
 g_i(x^k)+\langle\nabla\!g_i(x^k),y^{k,j}-x^k\rangle+\frac{L^{k,j}_i-\beta_{R}}{2}\|y^{k,j}-x^k\|^2\le0\quad\forall i\in[m].
 \end{equation}
 From the first inequality of \eqref{inexact2} and the expression of $F_{k,j}$, it is easy to check that $y^{k,j}\in X_k$ for each $j\in\mathbb{N}$. Along with the Lipschitz continuity of $\nabla\!g_i$ on $X_k$, for each $j\in\mathbb{N}$,
 \begin{align*}
 g_i(y^{k,j})&\le g_i(x^k)+\langle\nabla\!g_i(x^k),y^{k,j}-x^k\rangle+\frac{1}{2}(L^{k,j}_i-\beta_{R})\|y^{k,j}-x^k\|^2\nonumber\\
 &\quad\ +\frac{1}{2}[L_{\nabla g_i}^k-(L^{k,j}_i-\beta_{R})]\|y^{k,j}-x^k\|^2\nonumber\\
 &\stackrel{\eqref{quant-1}}{\le}\frac{1}{2}[L_{\nabla g_i}^k-(L^{k,j}_i-\beta_{R})]\|y^{k,j}-x^k\|^2\quad\forall i\in[m].
 \end{align*}
 Then, $g(y^{k,j})\le0$ if $\beta_{R}+L_{\nabla g_i}^k\le L^{k,j}_i=\tau^{j}L_i^{k,0}$ for all $i\in[m]$, so the first inequality in step 6 holds as long as $j>j_1:=\lceil\max\limits_{i\in[m]}\frac{\log(\beta_{R}+L_{\nabla\!g_i}^k)-\log L_i^{k,0}}{\log\tau}\rceil$. On this basis, by the Lipschitz continuity of $\nabla\!f$ on the set $X_k$ and inequality \eqref{psi-cvx}, for each $j\ge j_1$, it holds
 \begin{align}\label{inner-f}
  F(y^{k,j})&\le f(x^k)+\langle\nabla\!f(x^k),y^{k,j}-x^k\rangle+\frac{L_{\nabla\!f}^k}{2}\|y^{k,j}-x^k\|^2\nonumber\\
  &\quad+\phi(y^{k,j})-\psi(x^k)+\langle\xi^k,y^{k,j}-x^k\rangle\nonumber\\
  &=F_{k,j}(y^{k,j})+\frac{1}{2}L_{\nabla\!f}^k\|y^{k,j}\!-\!x^k\|^2-\frac{1}{2}\langle y^{k,j}\!-\!x^k,\mathcal{Q}_{k,j}(y^{k,j}\!-\!x^k)\rangle\nonumber\\
  &\le F(x^k)-\frac{\alpha}{2}\|y^{k,j}-x^k\|^2+\frac{L_{\nabla\!f}^k+\alpha-\tau^{j-j_1}\mu_{k,0}}{2}\|y^{k,j}-x^k\|^2,
 \end{align}
 where the second inequality is due to \eqref{inexact2} and $\mathcal{Q}_{k,j}\succeq\mu_{k,j}\mathcal{I}=\tau^{j-j_1}\mu_{k,j_1}\mathcal{I}\succeq\tau^{j-j_1}\mu_{k,0}\mathcal{I}$. Thus, the two inequalities in step 6 hold if $L_{\nabla\!f}^k\!+\!\alpha\le \tau^{j-j_1}\mu_{k,0}$ or equivalently $j>\frac{\log(L_{\nabla\!f}^k+\alpha)-\log\mu_{k,0}}{\log\tau}+j_1$. That is, the inner loop stops once $j>\frac{\log(L_{\nabla\!f}^k+\alpha)-\log\mu_{k,0}}{\log\tau}+j_1$. 
\end{proof}
\end{appendices}

\bibliography{references}

\end{document}